\documentclass[11pt]{amsart}
\usepackage[utf8]{inputenc}

\usepackage{amsmath,amsthm,amssymb,amsrefs}
\usepackage{url,enumerate,epsfig,color,graphicx,epstopdf,mathrsfs,MnSymbol}
\usepackage{hyperref}
\usepackage{enumitem}
\usepackage[all]{xy}
\usepackage{subcaption}
\usepackage{anyfontsize}
\usepackage{tikz-cd}
\usepackage{comment}

\usepackage[a4paper, top=3cm, bottom=3cm, left=2.5cm, right=2.5cm]{geometry}
\usepackage{tikz}
\usetikzlibrary{matrix,cd, arrows,intersections,tqft,decorations.markings, decorations.pathmorphing,plotmarks,decorations.pathreplacing}
\usepackage[colorinlistoftodos]{todonotes}
\usepackage{import}

\usepackage{array}
\newcolumntype{C}[1]{>{\centering\arraybackslash$}p{#1}<{$}}

\newtheorem{theorem}[equation]{Theorem}
\newtheorem{proposition}[equation]{Proposition}
\newtheorem{lemma}[equation]{Lemma}
\newtheorem*{myLemma}{Lemma}
\newtheorem{corollary}[equation]{Corollary}

\theoremstyle{definition}
\newtheorem{definition}[equation]{Definition}
\newtheorem{example}[equation]{Example}

\newtheorem*{notation}{Terminology}

\theoremstyle{remark}
\newtheorem*{ack}{Acknowledgments}
\newtheorem{remark}[equation]{Remark}

\makeatletter
\DeclareFontEncoding{LS1}{}{}
\DeclareFontSubstitution{LS1}{stix}{m}{n}
\DeclareMathAlphabet{\mathscr}{LS1}{stixscr}{m}{n}
\newcommand\xleftrightarrow[2][]{%
  \ext@arrow 9999{\longleftrightarrowfill@}{#1}{#2}}
\newcommand\longleftrightarrowfill@{%
  \arrowfill@\leftarrow\relbar\rightarrow}
\makeatother
\numberwithin{equation}{section}

\def\one{\mathbb{1}}

\def\dun{d^{\textrm{un}}}

\def\cI{\mathcal{I}}
\def\cD{\mathcal{D}}
\def\cC{\mathcal{C}}
\def\cF{\mathcal{F}}
\def\cH{\mathcal{H}}

\def\cG{\mathcal{G}}
\def\cB{\mathcal{B}}

\def\Z{\mathbb{Z}}

\def\R{\mathbb{R}}
\def\C{\mathbb{C}}
\def\F{\mathbb{F}}
\def\S{\mathbb{S}}

\def\cA{\mathcal{A}}

\def\cS{\mathcal{S}}
\def\cP{\mathcal{P}}

\def\sas{\mathscr{s}}
\def\sad{\mathscr{d}}
\def\saq{\mathscr{q}}
\def\sah{\mathscr{h}}
\def\sat{\mathscr{t}}
\def\sal{\mathscr{l}}
\def\sao{\mathscr{o}}

\def\sl{\mathfrak{sl}}
\def\sln{\sl_N} 
\def\one{\mathbf{1}}
\def\Lpm{L_+}
\def\Lmm{L_-}
\def\Lom{L_0}

\makeatletter
\newsavebox{\@brx}
\newcommand{\lllangle}[1][]{\savebox{\@brx}{\(\m@th{#1\langle}\)}%
  \mathopen{\copy\@brx\kern-0.5\wd\@brx\usebox{\@brx}}}
\newcommand{\rrrangle}[1][]{\savebox{\@brx}{\(\m@th{#1\rangle}\)}%
  \mathclose{\copy\@brx\kern-0.5\wd\@brx\usebox{\@brx}}}
\newcommand{\llbra}[1][]{\savebox{\@brx}{\(\m@th{#1[}\)}%
  \mathopen{\copy\@brx\kern-0.5\wd\@brx\usebox{\@brx}}}
\newcommand{\rrbra}[1][]{\savebox{\@brx}{\(\m@th{#1]}\)}%
  \mathclose{\copy\@brx\kern-0.5\wd\@brx\usebox{\@brx}}}
\makeatother
\newfont{\letterbeta}{diagram}
\newcommand{\smalldown}{\mskip0.2mu{\textrm{\letterbeta U}}\mskip-0.2mu}
\newcommand{\smallloop}{{\textrm{\letterbeta X}\mskip-0.1mu}}
\newcommand{\smallcircle}{{\textrm{\letterbeta V}}\mskip-0.1mu}
\newcommand{\otherdown}{{\textrm{\letterbeta Z}}\mskip-0.1mu}
\tikzset{myendarrow/.style={%
      decoration={%
	markings, %
	mark=at position 0.9 with {\arrow{latex}}%
      },%
      postaction={decorate}%
  }}
  \tikzset{mymidarrow/.style={%
      decoration={%
	markings, %
	mark=at position 0.5 with {\arrow{latex}}%
      },%
      postaction={decorate}%
  }}
  \tikzset{mybegarrow/.style={%
      decoration={%
	markings, %
	mark=at position 0.2 with {\arrow{latex}}%
      },%
      postaction={decorate}%
    }}

\def\rrbran{\rrbra} 

\DeclareMathOperator{\kh}{Kh}
\DeclareMathOperator{\Gr}{Gr}

\DeclareMathOperator{\KR}{KR}
\DeclareMathOperator{\EKR}{EKR}
\DeclareMathOperator{\Lee}{Lee}
\DeclareMathOperator{\ELee}{ELee}
\DeclareMathOperator{\LeeP}{LeeP}
\DeclareMathOperator{\KRP}{KRP}
\DeclareMathOperator{\RT}{RT}
\DeclareMathOperator{\DP}{DRT}
\DeclareMathOperator{\EM}{EM}
\DeclareMathOperator{\DDJ}{DJ}

\DeclareMathOperator{\Ext}{Ext}

\DeclareMathOperator{\Hom}{Hom}
\DeclareMathOperator{\Mor}{Mor}
\DeclareMathOperator{\Res}{Res}
\DeclareMathOperator{\Tot}{Tot}
\DeclareMathOperator{\Ind}{Ind}
\DeclareMathOperator{\Iso}{Iso}

\DeclareMathOperator{\lk}{lk}

\DeclareMathOperator{\Cube}{Cube}
\DeclareMathOperator{\SCube}{SCube}
\DeclareMathOperator{\ExCube}{\Cube^+}

\DeclareMathOperator{\cSym}{\mathbf{Sym}} 
\DeclareMathOperator{\Sym}{Sym} 
\DeclareMathOperator{\ev}{ev}
\DeclareMathOperator{\Cr}{Cr}
\DeclareMathOperator{\supp}{supp}

\newcommand{\bnbracket}[1]{[\kern-1.5pt [ #1 ]\kern-1.5pt]_{\operatorname{BN}}}
\newcommand{\Kom}[1]{\operatorname{Kom}\left( #1 \right)}
\newcommand{\ol}[1]{\overline{#1}}
\newcommand{\wt}[1]{\widetilde{#1}}
\newcommand{\wh}[1]{\widehat{#1}}
\newcommand{\Foam}{\mathbf{Foam}}
\newcommand{\Vect}{\mathbf{Vect}}
\newcommand{\Kar}{\mathbf{Kar}}
\def\SFoam{\S\Foam}
\def\SigFoam{\Sigma\Foam}
\def\SigPFoam{\Sigma'\Foam}

\def\OFoam{0\Foam}

\def\KO{\Kar^0(\SigPFoam_N)}
\def\KFoam{\Kom{\SFoam_N}}

\title{Khovanov-Rozansky $\sln$-homology for periodic links}
\author{Maciej Borodzik}
\address{Institute of Mathematics of Polish Academy of Science, ul \'Sniadeckich 8, 00-656 Warsaw, Poland}
\email{mcboro@mimuw.edu.pl}

\author{Wojciech Politarczyk}
\address{Institute of Mathematics, University of Warsaw, ul. Banacha 2,
02-097 Warsaw, Poland.}
\email{wpolitarczyk@mimuw.edu.pl}

\author{Ramazan Yozgyur}
\address{Institute of Mathematics, University of Warsaw, ul. Banacha 2,
02-097 Warsaw, Poland.}
\email{ryozgyur@mimuw.edu.pl}

\date{\today}

\makeatletter
\@namedef{subjclassname@2010}{\textup{2010} Mathematics Subject Classification}
\makeatother

\begin{document}
\begin{abstract}
  For an $m$-periodic link $L$, we show that the Khovanov-Rozansky $\sln$-homology carries an action of the group $\Z_m$.
  As an example of applications, we prove an analog of the periodicity criterion of Borodzik--Politarczyk using $\sln$-homology instead
  of Khovanov homology.
\end{abstract}
\maketitle

\section{Introduction}
\subsection{Overview}
Let $L\subset S^3$ be a link. For $m\ge 2$, we say that $L$ is \emph{$m$-periodic}, if it is invariant under a semi-free $\Z_m$-action on $S^3$
and $L$ is disjoint from the fixed point set. 
Given a periodic link, we ask in what way the symmetry of the link is reflected in the structure of link invariants.
An exemplary answer is the Murasugi formula \cite{Murasugi}. Not only it gives a useful periodicity
criterion, but also it relates the Alexander polynomial of $L$ and the Alexander polynomial of its quotient. 

The symmetry of the link is also visible in Khovanov homology.
Equivariant Khovanov homology for periodic links was defined in \cite{Politarczyk-Khovanov}. 
If $R$ is the coefficient ring for Khovanov homology and the group algebra $R[\Z_m]$ is semisimple, the
equivariant Khovanov homology of a link is equal as an $R$-module, to the Khovanov homology. That is to say, the group action on $S^3$
induces a well-defined group action on the Khovanov homology modules $\kh(L;R)$. Using this fact, the
first two authors established an obstruction for periodicity of link based on Khovanov homology. This obstruction refines
previously known obstructions based on the Jones polynomial \cite{Traczyk,Traczyk2,Politarczyk-Jones}.
Furthermore, a relation between Khovanov homology of the link and of its quotient was recently 
found in \cite{BPS,StoffregenZhang}. 

The goal of the present article is to prove the results of \cite{BP,Politarczyk-Khovanov} in
the case of $\sln$-homology.
$\sln$-homology for links was introduced by Khovanov and Rozansky \cite{Khovanov_Rozanski1,Khovanov_Rozanski2} 
as a generalization of Khovanov homology.
The original construction
used matrix factorizations. Over the years, an alternative approach using webs and foams has been  developed, see
\cite{Queffelec_Rose_foam,RobertWagner,ETW}. This approach, sketched in
Section~\ref{sec:webs_and_foams} below, is a generalization of Bar-Natan's definition of Khovanov homology~\cite{BarNatan}. 
This combinatorial definition
turns out to be very well-suited for studying periodic links.

\subsection{Main results}
In this paper we use the approach of \cite{Politarczyk-Jones,Politarczyk-Khovanov,BP} in the context of 
$\sln$-homology. Namely, we show that for a periodic link, the symmetry group $\Z_m$ of the link induces
a $\Z_m$-action on its $\sln$-homology.
The action does not depend on the choice of the link diagram, up to equivariant isotopy.
The existence of the action can be used to obtain periodicity obstructions, generalizing the results from~\cite{BP}.
Recall that the graded Euler characeristic of the Khovanov homology is the Jones polynomial.
Equivariant Khovanov homology allows for a refinement of the Jones polynomial, called \emph{difference Jones polynomials}, see
\cite{Politarczyk-Jones}. 
Analogously, the Euler characteristic of $\sln$-homology yields a well-studied Reshetikin-Turaev polynomial, known
also as the $\sln$-polynomial. For a periodic link, we define analogs of difference Jones polynomials for $\sln$-homology. 

We establish a skein spectral sequence for a change of an orbit of crossings
in $\sln$-homology. An analogous skein spectral sequence was considered in \cite{Politarczyk-Khovanov} for Khovanov homology of
a periodic link. 
The skein spectral sequence implies a relation between the difference $\sln$-polynomials under a change of an orbit of
crossings. As in \cite{BP}, this allows us to define periodicity criterion.

Unfortunately, the new criterion, 
like the criterion using Khovanov homology \cite{BP}, cannot obstruct periodicity of order $3$ and $4$.
The proof of this failure is slightly more involved than in the Khovanov homology case.

\subsection{Structure of paper}
Section~\ref{sec:webs_and_foams} recollects basic facts about webs and foams. We build mostly on previous papers like \cite{ETW,Queffelec_Rose_foam,RobertWagner}. We aim to give precise references for readers who are not that familiar with $\sln$-homology. 
Section~\ref{sec:webs_and_foams} finishes with the definition of Khovanov-Rozansky homology as a module over the ring of symmetric polynomials.
Section~\ref{sec:modules} discusses specializations of this definition. We give
a precise construction of the Lee-Gornik spectral sequence. Next, we study properties
of Lee-Gornik homology.

The core of the paper is Section~\ref{sec:periodic}. We study Khovanov-Rozansky homology for periodic links. We prove the main
result of this paper, namely Theorem~\ref{thm:group_action}, stating that for $m$ periodic links, there is a
well-defined action of $\Z_m$ on Khovanov-Rozansky homology.

The remaining part of the paper is devoted to exemplary applications of Theorem~\ref{thm:group_action}. In Section~\ref{sec:skein-spectr-sequ} we study the skein spectral sequence that appears when an \emph{orbit} of crossings on a periodic link diagram is changed. This section
is based on analogous results from Khovanov homology \cite{Politarczyk-Jones,Politarczyk-Khovanov}. In Section~\ref{sec:polynomial}
we relate equivariant Khovanov-Rozansky homology to Reshetikhin-Turaev polynomials. By analogy to \cite{Politarczyk-Jones}, we 
construct the so-called difference polynomials. Section~\ref{sec:polynomial} ends up with periodicity criterion, Theorem~\ref{thm:periodicity},
which is a generalization of the corresponding result~\cite{BP} for Khovanov homology.

The main difference between the present paper and previous papers on equivariant Khovanov homology, lies in the control of signs. While
Koszul's sign rule is sufficient to control signs for non-equivariant $\sln$-homology, one has to be much more careful, while constructing
the group actions. The situation is more complex than in Khovanov homology, because the underlying modules of Khovanov bracket have 
canonical bases. A careful analysis of sign assignments is the main technical difficulty in the present paper.

\begin{notation}
  In the literature, the notion of `equivariant $\sln$-homology' is often used in the context of homology over the ring of symmetric polynomials. In the present paper, we use the word `equivariant' only in the context of the group action on a periodic link.

  Throughout the paper, $N$ denotes a fixed integer greater than $1$. Unless specified otherwise, all links are oriented.
\end{notation}
\begin{ack}
  The authors are grateful to Louis-Hadrien Robert, Emmanuel Wagner, and Paul Wedrich for fruitful conversations and patiently
  clarifying definitions of $\sln$-homology. MB and RO were supported by NCN OPUS 2019/B/35/ST1/01120 grant.
\end{ack}

\section{Webs, foams and categories}\label{sec:webs_and_foams}

\subsection{Webs and foams}
\begin{definition}[\(N\)-webs]\label{def:web}
  Fix an integer \(N \geq 2\).
  A \emph{closed \(N\)-web} (shortly: a \emph{web}) is a finite oriented trivalent graph $V$ without sources and sinks
  properly embedded in $\R^2$. Each edge is assumed to be labelled by an integer $0,\dots,N$.
  The labellings are required to 
  satisfy the following rule, called the \emph{flow condition}; see Figure~\ref{fig:web_flow}.
  \begin{itemize}
  \item If two edges with labels $a$ and $b$ enter a vertex, then the outgoing edge is labelled by $a+b$. A vertex with two
    incoming edges is called a \emph{merge vertex};
  \item If two edges with labels $a$ and $b$ exit a vertex, then the incoming edge is labelled by $a+b$. A vertex with
    two outgoing vertices is called a \emph{split} vertex.
  \end{itemize}
\end{definition}
\begin{remark}
  By convention, edges labelled by $0$ might be erased without changing the web. Therefore, in some articles,
  the labelling is assumed to take values from $1$ to $N$.
\end{remark}

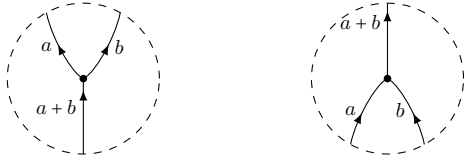
\begin{figure}
  \begin{tikzpicture}
    \begin{scope}
      \draw[dashed] (0,0) circle (1);
      \clip (0,0) circle (1);
      \draw[myendarrow] (0,-1.5) -- node[scale=0.7,left,near end] {$a+b$} (0,0);
      \draw[mymidarrow] (0,0) .. controls ++ (155:0.2) and ++(0,-0.3) .. node [scale=0.7,left] {$a$} (-0.5,1);
      \draw[mymidarrow] (0,0) .. controls ++ (25:0.2) and ++(0,-0.3) .. node [scale=0.7,right] {$b$} (0.5,1);
      \fill (0,0) circle (0.05);
    \end{scope}
    \begin{scope}[xshift=4cm]
      \draw[dashed] (0,0) circle (1);
      \clip (0,0) circle (1);
      \draw[myendarrow] (0,0) -- node[scale=0.7,left,near end] {$a+b$} (0,1);
      \draw[mymidarrow] (-0.5,-1) .. controls ++(0,0.3) and ++ (205:0.2) .. node [scale=0.7,left] {$a$} (0,0);
      \draw[mymidarrow] (0.5,-1) .. controls ++(0,0.3) and ++ (335:0.2) .. node [scale=0.7,left] {$b$} (0,0);
      \fill (0,0) circle (0.05);
    \end{scope}
  \end{tikzpicture}
  \caption{The flow condition of Definition~\ref{def:web}.}\label{fig:web_flow}
\end{figure}

Suppose $W_0$ and $W_1$ are two webs. We think of $W_0$ as a web in $\R^2\times\{0\}$ and of $W_1$ as a web in $\R^2\times\{1\}$.
\begin{definition}[\(N\)-Foams]
  Fix an integer \(N \geq 2\).
  An \emph{undecorated \(N\)-foam} (shortly an \emph{undecorated foam} or a \emph{foam} if it is clear from the context)
  $F\colon W_0\to W_1$ is a finite 2-dimensional CW-complex properly embedded in $\R^2\times[0,1]$ such that:
  \begin{itemize}
  \item If $x\in F\setminus (W_0\cup W_1)$, then there exists a neighborhood $U$ of $x$ in $F$ homeomorphic to one of the following three models:
    \begin{itemize}
    \item a \emph{smooth point}, $U$ is homeomorphic to a disk in $\R^2$;
    \item a $Y$-shaped point (codimension 1 singularity): $U$ is homeomorphic to union of three distinct rays stemming out of a commont point, times $(0,1)$;
    \item a cone over a 1-skeleton of a tetrahedron (codimension 2 singularity), when $x$ is a triple point. Compare Figure~\ref{fig:foam}.
    \end{itemize}
  \item Every \emph{facet} $F_i$ of $F$, i.e. a connected component of the set of smooth points
    carries an orientation and a label by an integer $0,\dots,N$;
  \item Every \emph{seam} $C_i$, that is a connected component of the set of $Y$-shaped points, of $F$ carries an orientation;
  \item The orientation of every seam agrees with the orientation of precisely two adjacent facets; if these two facets are labelled by $a$ and $b$, then the third facet is labelled by $a+b$;
  \item The \emph{bottom boundary} of each facet $F_i$, that is $\ol{F_i}\cap (\R^2\times\{0\})$ is an edge of $W_0$ with the same label and the orientation opposite to the orientation induced by $F_i$;
  \item The \emph{top boundary} of each facet $F_i$, that is $\ol{F_i}\cap (\R^2\times\{1\})$ is an edge of $W_1$ with the same label and the orientation agreeing with the orientation induced by $F_i$;
  \end{itemize}
\end{definition}
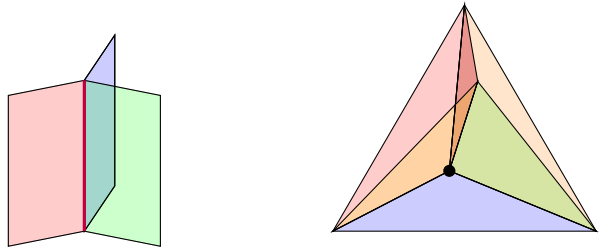
\begin{figure}
  \begin{tikzpicture}
    \begin{scope}
      \coordinate (OT) at (0,1);
      \coordinate (AT) at (0.4,1.6);
      \coordinate (BT) at (1,0.8);
      \coordinate (CT) at (-1,0.8);
      \coordinate (OD) at (0,-1);
      \coordinate (AD) at (0.4,-0.4);
      \coordinate (BD) at (1,-1.2);
      \coordinate (CD) at (-1,-1.2);
      \draw[fill=blue,opacity=0.2] (OT) -- (AT) -- (AD) -- (OD) -- (OT);
      \draw (OT) -- (AT) -- (AD) -- (OD) -- (OT);
      \draw[fill=green,opacity=0.2] (OT) --(BT) -- (BD) -- (OD) -- (OT);
      \draw[fill=red,opacity=0.2] (OT) -- (CT) -- (CD) -- (OD) -- (OT);
      
      \draw (OT) -- (AT) -- (AD) -- (OD) -- (OT);
      \draw (OT) -- (BT) -- (BD) -- (OD) -- (OT);
      \draw (OT) -- (CT) -- (CD) -- (OD) -- (OT);
      \draw[very thick,purple] (OT) -- (OD);
    \end{scope}
    \begin{scope}[xshift=5cm]
      \coordinate (O) at (-0.2,-0.2);
      \coordinate (A) at (90:2);
      \coordinate (B) at (210:2);
      \coordinate (C) at (330:2);
      \coordinate (D) at (80:1);
      \fill[purple,opacity=0.4] (O) -- (D) -- (A) -- (O);
      \fill[red,opacity=0.2] (O) -- (A) -- (B) -- (O);
      \fill[blue,opacity=0.2] (O) -- (B) -- (C) -- (O);
      \fill[green,opacity=0.2] (O) -- (C) -- (D) -- (O);
      \fill[orange,opacity=0.2] (O) -- (C) -- (A) -- (O);
      \fill[yellow,opacity=0.2] (O) -- (B) -- (D) -- (O);
      \draw (O) -- (C) -- (D) -- (O);
      \draw (O) -- (D) -- (A) -- (O);
      \draw (O) -- (A) -- (B) -- (O);
      \draw (O) -- (B) -- (C) -- (O);
      \draw (O) -- (B) -- (D) -- (O);
      \draw (O) -- (A) -- (C) -- (O);
      \fill (O) circle (0.08);
    \end{scope}
  \end{tikzpicture}
  \caption{Codimension 1 and 2 singular points of a foam.}\label{fig:foam}
\end{figure}

Suppose $F_{01}$ is a foam from $W_0$ to $W_1$ and $F_{12}$ is a foam from $W_1$ to $W_2$. There is an obvious notion of a composition of foams.
\begin{definition}[Composition of foams]\label{def:composition}
  The composition $F_{02}$ of webs $F_{01}$ and $F_{12}$ is the union $\wt{F}_{01}\cup\wt{F}_{12}$, where $\wt{F}_{01}$ is the foam
  scaled along the vertical axis to $\R^2\times[0,1/2]$ and $\wt{F}_{12}$ is the foam $F_{12}$ scaled and moved to $\R^2\times[1/2,1]$.
\end{definition}

A prominent role among foams is played by closed foams:
\begin{definition}[Closed foam]
  A \emph{closed foam} is a foam from an empty web to an empty web.
\end{definition}

\subsection{Colorings and decorations}
Colorings and decorations are additional structures assigned to foams and webs needed to define link invariants.

\begin{definition}[Coloring of a web]\label{def:coloring_web}
  Let $W$ be a web.
  A \emph{coloring} of \(W\) assigns to any \(a\)-labelled edge $e$ of $W$ a set of colors $c(e)\subset\{1,\dots,N\}$ such that $\# c(e) = a$.
  The assignment satisfies the following conditions, generalizing the flow condition of Definition~\ref{def:web}.
  \begin{itemize}
  \item If two edges \(e_{1},e_{2}\) enter a vertex, then the outgoing edge, denoted \(e_{3}\), has colors $c(e_{3}) = c(e_{1}) \cup c(e_{2})$;
    in particular $c(e_{1}) \cap c(e_{2})=\emptyset$;
  \item If two edges \(e_{1},e_{2}\) exit from a vertex, then the incoming edge \(e_{3}\) has colors $c(e_{3}) = c(e_{1}) \cup c(e_{2})$. 
  \end{itemize}
\end{definition}

\begin{figure}
  \begin{tikzpicture}
    \begin{scope}
      \draw[dashed] (0,0) circle (1);
      \clip (0,0) circle (1);
      \draw[myendarrow] (0,-1.5) -- node[scale=0.7,left,near end] {$A\cup B$} (0,0);
      \draw[mymidarrow] (0,0) .. controls ++ (155:0.2) and ++(0,-0.3) .. node [scale=0.7,left] {$A$} (-0.5,1);
      \draw[mymidarrow] (0,0) .. controls ++ (25:0.2) and ++(0,-0.3) .. node [scale=0.7,right] {$B$} (0.5,1);
      \fill (0,0) circle (0.05);
    \end{scope}
    \begin{scope}[xshift=4cm]
      \draw[dashed] (0,0) circle (1);
      \clip (0,0) circle (1);
      \draw[myendarrow] (0,0) -- node[scale=0.7,left,midway] {$A\cup B$} (0,1);
      \draw[mymidarrow] (-0.5,-1) .. controls ++(0,0.3) and ++ (205:0.2) .. node [scale=0.7,left] {$A$} (0,0);
      \draw[mymidarrow] (0.5,-1) .. controls ++(0,0.3) and ++ (335:0.2) .. node [scale=0.7,left] {$B$} (0,0);
      \fill (0,0) circle (0.05);
    \end{scope}
  \end{tikzpicture}
  \caption{The flow condition for colored webs (Definition~\ref{def:coloring_web})}\label{fig:web_flow2}
\end{figure}
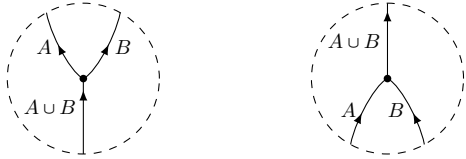
Colorings of webs are used extensively in the study of Lee homology, see Subsection~\ref{sub:properties_of_lee}.
The definition of a coloring of a foam is analogous.
\begin{definition}[Coloring of a foam]\label{def:coloring_foam}
  Let $F$ be a foam.
  A \emph{coloring} of $F$ assigns a set of colors $c(f)\subset\{1,\dots,N\}$  to each face $f$ of the foam such that $\#c(f)$ is the labelling of $f$.
  The assignment is required to satisfy the following compatibility relation, extending
  the compatibility relation for labels: near each seam meeting facets \(f_{1},f_{2},f_{3}\), where the orientations of \(f_{1}\) and \(f_{2}\) agree with the orientation of the seam,
  we have \(c(f_{3}) = c(f_{1}) \cup c(f_{2})\).

  A \emph{colored foam} is a foam equipped with a coloring.
\end{definition}

\subsection{Decorations, degrees and evaluations}

We first recall the following definition.
\begin{definition}[Degree of an undecorated foam]\label{def:degree}
  The degree $\dun(F)$ of a foam $F$ is the sum of the following contributions.
  \begin{itemize}
  \item Each facet $f$ with label $a$ contributes $-a(N-a)\chi(f)$. Here, $\chi$ denotes the Euler characteristic;
  \item Each seam, which is not a circle, and which is surrounded by faces with labels $a,b,a+b$ contributes $ab+(a+b)(N-a-b)$;
  \item Each singular point surrounded by faces with labels $a,b,c,a+b,b+c,a+b+c$ contributes $ab+bc+cd+da+ac+bd$, where $d=N-a-b-c$.
  \end{itemize}
\end{definition}

We now introduce the notion of a decoration. We note that various authors give slightly different ways of decorating foams, which eventually
lead to equivalent theories. Our approach follows mostly \cite{RobertWagner}.
\begin{definition}[Decoration of a foam]
  Let $F$ be a foam. A \emph{decoration} of \(F\) is an assignment of a symmetric homogeneous polynomial $P_f$ in $a$ variables to any \(a\)-labelled facet $f$.
  A \emph{decorated foam} is a foam together with a decoration.
\end{definition}
Suppose $F_{01}$ and $F_{12}$ are two \emph{decorated} foams. While stacking $F_{12}$ over $F_{01}$ as in Definition~\ref{def:composition},
we merge some pairs of facets into one. In that case, the polynomials decorating two facets get multiplied; see Figure~\ref{fig:get_multiplied}.
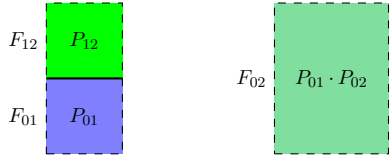
\begin{figure}
  \begin{tikzpicture}
    \fill[blue!50] (-3,-1) rectangle (-2,0);
    \fill[green] (-3,0) rectangle (-2,1);
    \draw[dashed] (-3,-1) -- (-2,-1) -- (-2,1) -- (-3,1) -- (-3,-1);
    \draw[thick] (-3,0) -- (-2,0);
    \draw (-2.5,-0.5) node [scale=0.7] {$P_{01}$};
    \draw (-2.5,0.5) node [scale=0.7] {$P_{12}$};
    \draw (-3.3,-0.5) node [scale=0.7] {$F_{01}$};
    \draw (-3.3,0.5) node [scale=0.7] {$F_{12}$};
    \fill[blue!25!green!50] (0,-1) rectangle (1.5,1);
    \draw[dashed] (0,-1) -- (1.5,-1) -- (1.5,1) -- (0,1) -- (0,-1);
    \draw(0.75,0) node [scale=0.7] {$P_{01}\cdot P_{02}$};
    \draw(-0.3,0) node [scale=0.7] {$F_{02}$};
  \end{tikzpicture}
  \caption{Rule for gluing decorated foams.}\label{fig:get_multiplied}
\end{figure}

\begin{definition}[Degree of a decorated foam]\label{def:decorated}
  Let $F$ be a decorated foam. The degree of $F$ is equal to $d(F):=\dun(F)+2\sum\deg P_f$, where the sum is taken over all facets $f$ of $F$.
\end{definition}

In \cite[Definition 2.12]{RobertWagner} Robert and Wagner assign a polynomial to any decorated closed foam.
We recall a few details of this construction for further use.
Fix once and for all variables $X_1,\dots,X_N$. 
For a decorated closed foam $F$ and a coloring $c$, one associates a rational function in $X_1,\dots,X_N$ denoted by $\langle F,c\rangle$.
One sets
\begin{equation}\label{eq:eval_color}
  \langle F\rangle = \sum_{c} \langle F,c\rangle,
\end{equation}
where the sum is taken over all colorings of $F$.
Note that $\langle F\rangle$ is always a polynomial in $X_1,\dots,X_N$; see~\cite[Proposition 2.18]{RobertWagner}.

We make an observation, which follows promptly from the definition of $\left\langle F \right\rangle$.
\begin{lemma}\label{lem:invariant}
  Suppose $F$ and $F'$ are two closed foams in $\R^2\times[0,1]$. If $F$ and $F'$ are
  isotopic, then $\left\langle F' \right\rangle = \left\langle F  \right\rangle $.
\end{lemma}
\subsection{Categories of foams}\label{sub:catego}
Webs and foams can be combined into a category.
In this subsection, we give several precise incarnations of this idea.
\begin{definition}[The category $\Foam^*_N$]\label{def:foam_cat}
  Fix an integer \(N \geq 2\).
  The category $\Foam^*_N$ is the graded category whose objects are formal shifts of $N$-webs denoted \(q^{a} W\), for \(a \in \Z\) and \(W\) an \(N\)-web.
  Morphisms between $q^{a_{0}} W_0$ and $q^{a_{1}} W_1$ are decorated \(N\)-foams~\emph{from $W_0$ to $W_1$} of
  (decorated) degree~\(a_{1}-a_{0}\). 
  Composition of morphisms is realized by composition of foams, as described in Definition~\ref{def:composition}.
\end{definition}
\begin{remark}
  In order to comply with terminology of \cite{ETW}, we should call this category $\Foam^*_N(cl)$, because we consider only closed webs. 
  In the present paper, unless specified otherwise,
  all webs are assumed to be closed.
\end{remark}

We recall now the definition of the category $\SFoam^*_N$; see \cite[Definition 2.13]{ETW}.
Henceforth, for a fixed integer \(N \geq 2\), we denote by 
\[\S_N:=\Sym[X_1,\dots,X_N]\]
the (graded) ring of symmetric polynomials with integer coefficients. We adopt a convention that the degree
of each variable $X_1,\dots,X_N$ is equal to $2$.
\begin{definition}[The category $\SFoam^*_N$]
  The category $\SFoam^*_N$ is the \(\S_{N}\)-linear closure of \(\Foam^{\ast}_{N}\).
  \begin{itemize}
  \item \(\SFoam_{N}^{\ast}\) is closed under finite direct sums, hence objects of \(\SFoam^{\ast}_{N}\) are finite formal direct sums $q^{a_{1}}W_{1} \oplus \cdots \oplus q^{a_{k}}W_{k}$, where $W_{1},\ldots,W_{k}$ are \(N\)-webs and $a_{1},\ldots,a_{k}\in\Z$ denote grading shifts;
  \item Morphisms from \(q^{a_{1}} V_{1} \oplus \cdots \oplus q^{a_{k}} V_{k}\) to \(q^{b_{1}} W_{1} \oplus \cdots \oplus q^{b_{l}} W_{l}\) are \(l \times k\) matrices \(M\) such that the \((i,j)\)-th coefficient of \(M\), for \(1 \leq i \leq l\) and \(1 \leq
    j \leq l\), is a formal \(\S_{N}\)-linear
    combination of decorated \(N\)-foams \(V_{j} \to W_{i}\) of degree \(b_{i}-a_{j}\).
  \end{itemize}
\end{definition}

The basic object in theory of $\sln$-homologies of links is the evaluation functor.
\begin{definition}[Naive evaluation functor]\label{def:evaluation}
  The \emph{naive evaluation functor} $\wt{\cF}$ from the category $\SFoam^*_N$ to the category $\cSym^*_N$ of graded projective (not necessarily finitely generated) $\S_N$-modules is the representable functor
  \[\wt{\cF}(-) := \Hom_{\SFoam^{\ast}_{N}}(\emptyset,-).\]
  Observe that the evaluation functor is grading-preserving and \(\S_{N}\)-linear.
  Alternatively, the naive evaluation functor can be described by its action on objects and morphisms.
  \begin{itemize}
  \item For a shifted web \(q^{a} V\), $\wt{\cF}(q^{a} V) := \Hom_{\SFoam^*_N}(\emptyset, q^{a} V)$.
    In other words, \(\wt{\cF}(q^{a}V)\) is the free \(\S_{N}\)-module spanned by \(\Hom_{\Foam_N}(\emptyset,q^{a}V)\).
  \item For any morphism $F\colon q^{a} V\to q^{b} W$ in $\Foam_N^{\ast}$, $\wt{\cF}$ assigns the postcomposition map
    \[\Hom_{\SFoam^*_N}(\emptyset,q^{a} V) \xrightarrow{\wt{\cF}(F)(-) := F\circ(-)} \Hom_{\SFoam^*_N}(\emptyset, q^{b} W).\]
    Observe that the degree of the map \(\wt{\cF}(F)\) is \(b-a\).
  \item The action of \(\wt{\cF}\) on direct sums of shifted webs and maps between them is determined by the fact that the functor preserves direct sums.
  \end{itemize}
\end{definition}

The objects in the image of the naive evaluation functor are rather large.
For instance, if $V$ is a web, any two isotopic non-equal foams from $\emptyset$ to $V$ will be mapped by $\wt{\cF}$
to different generators of the $\S_N$-module $\wt{\cF}(V)$. That is, $\wt{\cF}(V)$ will usually not be finitely generated.
We can rectify the naive evaluation functor so that the new evaluation functor assigns finitely generated graded projective \(\S_{N}\)-modules to any web \(V\).
Suppose $q^{a} V$ is a web and $F'\in\Hom_{\Foam_N^*}(q^{a}V,\emptyset)$.
Consider the \(\S_{N}\)-linear map $\phi_{F'} \colon \wt{\cF}(q^aV) \to \S_N$ given by the formula $\phi_{F'}(F)=\langle F'\circ F\rangle$.
Set
\begin{equation}\label{eq:civ}
  \cI(q^{a}V) := \bigcap_{F'\in\Hom_{\Foam_N^*}(V,\emptyset)} \ker\phi_{F'}.
\end{equation}
The submodule $\cI(V)$ consists of all $\S_N$-linear combinations of foams from $\emptyset$ to $q^{a}V$ that evaluate to zero, when capped by any foam from $q^{a}V$ to $\emptyset$.
For example, if $F$ and $F'$ are two isotopic foams from $\emptyset$ to $q^{a}V$, then $F-F'$ belongs to $\cI(q^{a}V)$ by Lemma~\ref{lem:invariant}.

\begin{definition}[The evaluation functor]\label{def:evaluation_2}
  The evaluation functor $\cF$ from the category $\SFoam^*_N$ to the category $\cSym_N$ of \emph{finitely generated} graded projective $\S_N$-modules assigns to any shifted web \(q^{a}V\) the quotient module
  \[\cF(q^{a}V) := \wt{\cF}(q^{a}V) / \cI(q^{a}V) = \Hom_{\SFoam^*_N}(\emptyset,q^{a}V)/\cI(q^{a} V).\]
  Since for any \(F \in \Hom_{\SFoam_{N}}(q^{a}V,q^{b}W)\), we have \(\wt{\cF}(F)(\cI(q^{a}V)) \subset \cI(q^{b}W)\), the naive evaluation functor descends to a well-defined quotient functor~\(\cF\).
  It is argued in~\cite[Section 3]{RobertWagner} that the target category of~$\cF$ is the category $\cSym_N$.
\end{definition} 

We now simplify our foam category further by introducing relations coming from the evaluation of functor.
\begin{definition}[The category $\SFoam_N$]\label{def:sfoam}
  The category $\SFoam_N$ is the category is the quotient category of \(\SFoam^{\ast}_{N}\), where
  \[\Hom_{\SFoam_N}(V,W):=\Hom_{\SFoam^*_N}(V,W)/\ker\cF,\]
  and \(\cF\) denotes the evaluation functor.
\end{definition}
For example, if two foams are isotopic,
they are already identified as morphisms in $\SFoam_N$. 
The functor $\cF$ descends to an evaluation functor on the quotient category \(\SFoam_{N}\); we will denote it by the same letter. To be more precise, \(\cF \colon \SFoam_{N} \to \Sym_{N}\) is the representable functor
\[\cF(V)=\Hom_{\SFoam_N}(\emptyset,V).\]
In particular, $\cF$ acts via compositions on morphisms.


\begin{definition}
  If \(\mathcal{A}\) is an additive category, we denote by \(\Kom{\mathcal{A}}\) the category of bounded cochain complexes in \(\mathcal{A}\).
  Morphisms in \(\Kom{\mathcal{A}}\) are cochain maps.
\end{definition}
In this paper we will be mainly interested in categories \(\KFoam\) and \(\Kom{\cSym_{N}}\).
In particular, the evaluation functor \(\cF \colon \SFoam_{N} \to \cSym_{N}\) extends to a functor \(\cF \colon \KFoam \to \Kom{\cSym_{N}}\).

\begin{figure}
  \begin{tikzpicture}[every node/.style={scale=0.7}]
    \begin{scope}
      \draw[thin] (-4.2,-1) -- (-4.3,-1) -- (-4.3,1) -- (-4.2,1);
      \draw[thin] (-4.25,-1) -- (-4.25,1);
      \draw[thin] (-2.8,-1) -- (-2.7,-1) -- (-2.7,1) -- (-2.8,1);
      \draw[thin] (-2.75,-1) -- (-2.75,1);
      \begin{scope}[thick]
        \draw[myendarrow,mybegarrow] (-3,-1) -- (-4,1);
        \draw(-3,1.1) node {$a$};
        \draw(-3,-1.1) node {$b$};
        \draw(-4,1.1) node {$b$};
        \draw(-4,-1.1) node {$a$};
        \fill[white] (-3.5,0) circle (0.2);
        \draw[myendarrow,mybegarrow] (-4,-1) -- (-3,1);
      \end{scope}
    \end{scope}
    \begin{scope}[xshift=0cm, very thick]
      \draw(-1.3,0) node[scale=1.2] {$q^{-x}$};
      \draw[myendarrow,mybegarrow] (-1,-1) -- (-1,1);
      \draw[myendarrow,mybegarrow] (0,-1) -- (0,1);
      \draw[mymidarrow] (-1,0.3) -- node[midway, below] {$a-b$} (0,0.5);
      \draw(0,1.1) node {$a$};
      \draw(0,-1.1) node {$b$};
      \draw(-1,1.1) node {$b$};
      \draw(-1,-1.1) node {$a$};
    \end{scope}
    \begin{scope}[xshift=1cm,very thick]
      \draw(1.2,0) node[scale=1.2] {$q^{1-x}$};
      \draw[myendarrow,mybegarrow,mymidarrow] (1.5,-1) -- (1.5,1);
      \draw[myendarrow,mybegarrow,mymidarrow] (3,-1) -- (3,1);
      \draw(1.5,1.1) node {$a$};
      \draw(1.5,-1.1) node {$b$};
      \draw(3,1.1) node {$b$};
      \draw(3,-1.1) node {$a$};
      \draw[mymidarrow] (1.5,0.3) -- node[midway, above] {$a-b+1$} (3,0.5);
      \draw[mymidarrow] (1.5,-0.3) -- node[midway, below] {$1$} (3,-0.5);
    \end{scope}
    \begin{scope}[xshift=-1cm, yshift=-3cm,very thick]
      \draw(-0.6,0) node[scale=1.2] {$q^{b-1-x}$};
      \draw(0,1.1) node {$a$};
      \draw(0,-1.1) node {$b$};
      \draw(1,1.1) node {$b$};
      \draw(1,-1.1) node {$a$};
      \draw[myendarrow,mybegarrow,mymidarrow] (0,-1) -- (0,1);
      \draw[myendarrow,mybegarrow,mymidarrow] (1,-1) -- (1,1);
      \draw[mymidarrow] (0,0.3) -- node[midway, above] {$a-1$} (1,0.5);
      \draw[mymidarrow] (1,-0.3) -- node[midway, below] {$b-1$} (0,-0.5);
    \end{scope}
    \begin{scope}[xshift=2.5cm, yshift=-3cm,very thick]
      \draw(-0.5,0) node[scale=1.2] {$q^{b-x}$};
      \draw(0,1.1) node {$a$};
      \draw(0,-1.1) node {$b$};
      \draw(1.5,1.1) node {$b$};
      \draw(1.5,-1.1) node {$a$};
      \draw[mymidarrow] (0,-1) -- (0,1);
      \draw[mymidarrow] (1.5,-1) -- (1.5,-0.3);
      \draw[mymidarrow] (1.5,0.5) -- (1.5,1);
      \draw[mymidarrow] (0,0.3) -- node[midway, above] {$b$} (1.5,0.5);
      \draw[mymidarrow] (1.5,-0.3) -- node[midway, below] {$a$} (0,-0.5);
    \end{scope}

    \draw(-2,0) node[scale=1.6] {$\stackrel{a\ge b}{=}$};
    \draw(6,0) node[scale=1.2] {$\dots$};
    \draw(-4,-3) node[scale=1.2] {$\dots$};
    \draw[->] (0.5,0) -- node[above,midway] {$d_0^+$} (1.5,0);
    \draw[->] (0.5,-3) -- node[above,midway] {$d_{b-1}^+$} (1.5,-3);
    \draw[->] (4.5,0) -- node[above,midway] {$d_1^+$} (5.5,0);
    \draw[->] (-3.5,-3) -- node[above,midway] {$d_{b-2}^+$} (-2.5,-3);
  \end{tikzpicture}
  \caption{The resolution of a crossing. Here $x=b(N-b)$ and $q$ denotes the quantum grading shift. The first term is at homological degree zero.}\label{fig:resolution}
\end{figure}
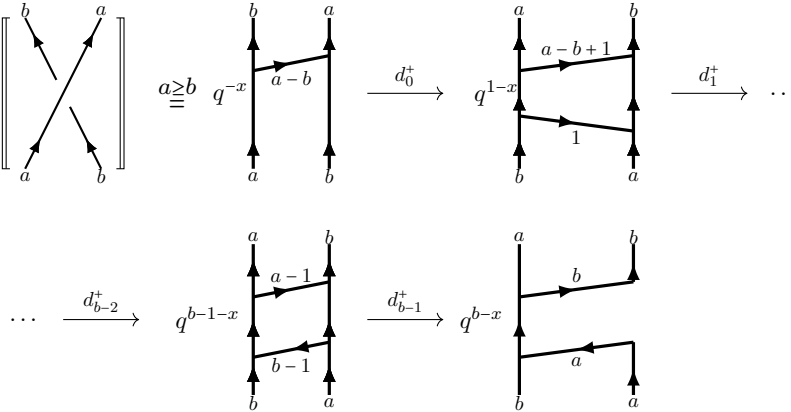
\subsection{$\S_N$-equivariant $\sln$-homology}\label{sub:sym_eq}

In this subsection we follow closely \cite[Section 3]{ETW}, with the exception that we focus on link invariants, as opposed to more general tangle invariants. Let $L$ be a link in $\R^3$. We may assume that $L$
is labelled, that is, components of $L$ carry an integer between $1$ and $N$ (if the link is not labelled, we label all its components by \(1\)).
We use the word `labelled' and not `colored', because this assignment is not meant to distinguish link components as, say, colored signatures, but to stress the thickness of components as e.g. in \cite{Khovanov_Rozanski1,MOY}. 

Let $D$ a planar diagram for the link $L$. We define the cochain complex (the bracket)
$\llbra D\rrbra\in\KFoam$ via the recipe as in \cite[Definition 3.3]{ETW}, namely as a tensor product of local complexes corresponding to crossings, depicted in Figure~\ref{fig:resolution}. More
precisely:

\begin{itemize}
\item An $a$-labelled strand is mapped to the corresponding web at homological degree zero; 
\item A positive crossing with labels $a,b$ with $a\ge b$ is mapped to the complex of Figure~\ref{fig:resolution}. The differential
  $d^+_k$ is specified in Figure~\ref{fig:differential} below;
\item A positive crossing with labels $a,b$ and $a< b$ is mapped to the complex obtained by mirroring the webs and the foams along the vertical axis (and swapping the roles of $a$ and $b$);
\item A negative crossing is mapped to the dual complex to the one in Figure~\ref{fig:resolution}. The duality amounts to inverting
  the order of the differential, and multiplying the quantum and homological degrees by $-1$.
\end{itemize}



We can describe the cochain complex $\llbra D\rrbra$ via a hypercube of resolutions. While this description is not needed for the construction,
it gives us a slightly better control of the sign assignment, which is important for discussing group actions. Note that in the standard
approach, as \cite{ETW,RobertWagner}, the sign assignment is hidden in the Koszul sign convention
for the tensor product of cochain complexes corresponding to local resolutions.

Let $\Cr(D)$ be the set crossings of a labelled diagram
$D$. 
For each crossing $i\in\Cr(D)$, denote by $a_i$ and $b_i$ the labels of strands of \(D\) meeting at \(i\).
We set $c_i=\min(a_i,b_i)$.
We define $C_i=\{0,\dots,c_i\}$ if the  crossing is positive, and $C_i=\{-c_i,\dots,0\}$ if the crossing is negative.
Likewise, we set $SC_i=[0,c_i]$ or $SC_i=[-c_i,0]$ depending on the sign of the crossing.
We consider \(SC_{i}\) as a CW-complex with \(0\)-cells being the integral points and \(1\)-cells being the intervals connecting consecutive integral points.
Set $\Cube(D)=\prod_i C_i$ and $\SCube(D)=\prod_i SC_i$ ($\SCube$ for `solid cube') with the product CW-structure.
Observe that~\(\Cube(D)\) can be identified with the \(0\)-skeleton of \(\SCube(D)\).

\begin{definition}
  \begin{itemize}
    \item For $I\in\Cube(D)$, we say that $I'\in\Cube(D)$ is an \emph{immediate successor} of \(I\), if $I$ and $I'$ agree
  on all coordinates but one, and that coordinate of $I'$ is one larger than that of $I$. 
\item A \emph{sign assignment} $\sas$ is an assignment of $\sas(I,I')\in\F_2$, for any pair $I,I'\in\Cube(D)$ such that $I'$ is an immediate successor of $I$. The sign assignment is subject to the following cochain condition.
  If $I_1$ and $I_2$ are immediate
  successors of $I$, $I_1\neq I_2$, and $I_{12}$ is an immediate successor of both $I_1$ and $I_2$,
  then $\sas(I,I_1)+\sas(I,I_2)+\sas(I_1,I_{12})+\sas(I_2,I_{12})=1\in\F_2$.
\item We can think of the sign assignment $\sas$ as a cellular \(1\)-cochain on \(\SCube\), i.e., \(\sas \in C^1_{\text{cell}}(\SCube;\F_2)\) such that $\partial \sas$ is the \(2\)-cochain with constant value $1 \in \F_{2}$, where \(\partial\) denotes
  the coboundary map in \(C^{\ast}_{\text{cell}}(\SCube;\F_2)\).
  \end{itemize}
\end{definition}


\begin{lemma}\label{lem:coboundary}
  For any diagram \(D\) there exists a sign assignment.
  Any two sign assignments \(\sas\) and \(\sas'\) differ by a coboundary: $\sas-\sas'=\partial \sat$.
  Here, $\sat$ is a cellular \(0\)-cochain on \(\SCube(D)\).
  The cochain \(\sat\) is uniquely determined if one fixes its value on $(0,\dots,0)$.
\end{lemma}
\begin{proof}
  If \(c \in C^{2}_{\text{cell}}(\SCube(D);\F_{2})\) denotes the constant \(2\)-cochain with value \(1\), then \(\partial c = 0\) (because every \(3\)-dimensional cell of \(\SCube(D)\) has an even number of
  \(2\)-dimensional faces).
  Since \(\SCube(D)\) is contractible, it follows that there exists a \(1\)-cochain \(\sas \in C^{1}_{\text{cell}}(\SCube(D);\F_{2})\) such that \(\partial \sas = c\), hence \(\sas\) is a sign assignment.
  If \(\sas\) and \(\sas'\) are two sign assignments, the difference $\sas-\sas'$ is a cellular \(1\)-cocycle. 
  As $\SCube(D)$ is contractible, $\sas-\sas'$ is a coboundary, i.e., $\sas-\sas'=\partial \sat$.
  Suppose $\sas-\sas'=\partial\sat'$ for some other cellular $0$-cochain.
  Then $\sat-\sat'$ is a cellular \(0\)-cocycle, 
  so $\sat-\sat'$ is constant.
\end{proof}

For $I\in \Cube(D)$ we consider the web $D_I$ obtained by taking the $I(i)$-th
resolution in Figure~\ref{fig:resolution} at the $i$-th crossing of the diagram $D$, where $i$ ranges
through all the crossings of $D$. To each such web, we associate its quantum degree:
\[\saq(I)=\sum_{i\in\Cr(D)}q_{i},\]
where $q_{i}=I(i)-c_i(N-c_i)$ if $i$ is a positive crossing (and $c_i=\min(a_i,b_i)$), and $q_i=-I(i)+c_i(N-c_i)$ if $i$ is
a negative crossing; compare Figure~\ref{fig:resolution}. There is also a homological degree
\[\sah(I)=\sum_{i\in\Cr(D)} I(i).\]

We define the cochain complex $\wt{\llbra D\rrbra}$ by requiring that at homological grading $s$, it is a formal sum of webs $D_I$ where $\sah(I)=s$. The differential is 
defined as follows. Suppose $I\in\Cube(D)$ and $I'$ is an immediate successor of $I$. That is, $I'$ differs from $I$
only at a single crossing.
We let $\delta(I,I')$ be the foam given by Figure~\ref{fig:differential} near the relevant crossing, otherwise it is a product foam.
The differential in the complex $\wt{\llbra D\rrbra}$ is given by
$(-1)^{\sas(I,I')}\delta(I,I')$, where $\sas(I,I')$ is a sign assignment.
The grading $\saq(I)$ induces a quantum grading of $\wt{\llbra D\rrbra}$.
By construction, there is an identification $\wt{\llbra D\rrbra}$ and $\llbra D\rrbra$.

Given a link diagram \(D\) and a sign assignment \(\sas\) on \(\Cube(D)\), we will sometimes write \(\llbra D, \sas \rrbra\) to emphasize the dependence of \(\llbra D \rrbra\) on the sign assignment.

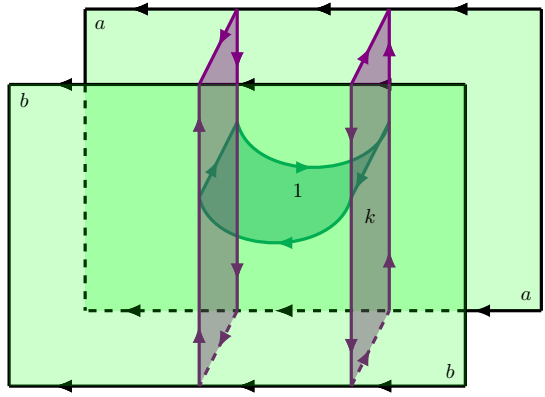
\begin{figure}
  \begin{tikzpicture}[every node/.style={scale=0.7}]
    \fill[green,opacity=0.2] (-3,-2) rectangle (3,2);
    \begin{scope}[very thick]
      \draw[mybegarrow,mymidarrow,myendarrow] (3,2) -- (-3,2);
      \draw[mybegarrow,mymidarrow,myendarrow] (2,1) -- (-4,1);
      \draw[mybegarrow,mymidarrow,myendarrow] (2,-3) -- (-4,-3);
      \draw[myendarrow] (3,-2) -- (2,-2);
      \draw[mymidarrow,myendarrow,dashed] (2,-2) -- (-3,-2);
      \draw(-4,-3) -- (-4,1);
      \draw(2,-3) -- (2,1);
      \draw(3,-2) -- (3,2);
      \draw(-3,2) -- (-3,1);
      \draw[dashed](-3,1) -- (-3,-2);
    \end{scope}
    \begin{scope}[green!60!blue,very thick]
      \fill[green!60!blue,opacity=0.4]  (0.5,-0.5) .. controls ++(-0.2,-0.8) and ++ (0.2,-0.8) .. (-1.5,-0.5) -- (-1,0.5) .. controls ++ (0.2,-0.8) and ++(-0.2,-0.8) .. (1,0.5) -- (0.5,-0.5) .. controls ++(-0.2,-0.8) and ++ (0.2,-0.8) .. (-1.5,-0.5);
      \draw[mymidarrow] (-1.5,-0.5) -- (-1,0.5);
      \draw[mymidarrow] (-1,0.5) .. controls ++ (0.2,-0.8) and ++(-0.2,-0.8) .. (1,0.5);
      \draw[myendarrow] (1,0.5) -- (0.5,-0.5);
      \draw[mymidarrow] (0.5,-0.5) .. controls ++(-0.2,-0.8) and ++ (0.2,-0.8) .. (-1.5,-0.5);
    \end{scope}
    \begin{scope}[blue!50!red,very thick]
      \draw[mybegarrow,myendarrow] (-1.5,-3) -- (-1.5,1);
      \draw[mybegarrow,myendarrow] (-1,2) -- (-1,-2);
      \draw[mybegarrow,myendarrow] (1,-2) -- (1,2);
      \draw[mybegarrow,myendarrow] (0.5,1) -- (0.5,-3);
      \draw[mymidarrow] (-1,2) -- (-1.5,1);
      \draw[mymidarrow] (0.5,1) -- (1,2);
      \draw[dashed,mymidarrow] (-1,-2) -- (-1.5,-3);
      \draw[dashed,mymidarrow] (0.5,-3) -- (1,-2);
      \fill[blue!50!red,opacity=0.4] (-1.5,-3) -- (-1,-2) -- (-1,2) -- (-1.5,1) -- (-1.5,-3);
      \fill[blue!50!red,opacity=0.4] (0.5,-3) -- (1,-2) -- (1,2) -- (0.5,1) -- (0.5,-3);
    \end{scope}
    \fill[green,opacity=0.2] (-4,-3) rectangle (2,1);
    \draw (-2.8,1.8) node {$a$};
    \draw(-3.8,0.8) node {$b$};
    \draw(2.8,-1.8) node {$a$};
    \draw(1.8,-2.8) node {$b$};
    \draw(-0.2,-0.4) node {$1$};
    \draw(0.75,-0.75) node {$k$};
  \end{tikzpicture}
  \caption{The foam that is the differential $d_k^+$ of the complex in Figure~\ref{fig:resolution}. It is decorated by constant
    polynomial equal $1$.}\label{fig:differential}·
\end{figure}

\subsection{Invariance of $\sln$-homology}
We first recall the classical result.
\begin{theorem}\label{thm:invariance_reid}
  Suppose $D$ and $D'$ differ by a single Reidemeister move. Then there is a homotopy equivalence of cochain complexes $\llbra D\rrbra\cong\llbra D'\rrbra$. 
\end{theorem}
\begin{proof}
  We present some aspects of the proof of this well-known fact, as a warm-up before Lemma~\ref{lem:equiv_reid}, which is significantly harder.
  In our exposition we will focus on choosing signs,
  While the proof for non-periodic links need not address the sign assignment problem, Koszul's sign rule being sufficient, we will show how
  to relate sign assignments on $D$ with sign assignments on $D'$
  in a consistent way. The methods we develop, will be used in
  the proof of Lemma~\ref{lem:equiv_reid}.

  Existing proofs of Theorem~\ref{thm:invariance_reid}
  were described in \cite[Theorem 3.5]{ETW}.
  We will follow \cite[Section 7]{MSV}.
  Even if \cite{MSV} does not deal with $\S_{N}$-equivariant foams, as noted in \cite{ETW} the techniques carry over to the case of $\S_N$-equivariant foams.

  Suppose $D\langle\smallloop\rangle$ is the diagram
  obtained from $D$ via a single Reidemeister move creating a positive crossing (the case of a negative crossing requires only minor changes,
  and it is left to the reader). We let $D\langle\smallcircle\rangle$ and $D\langle\otherdown\rangle$ be partial resolutions
  of $D\langle\smallloop\rangle$.
  We can present \(\llbra D \rrbra\) as the following bicomplex
  \begin{equation}\label{eq:r1_chain}
    0\to\llbra D\langle\smallcircle \rangle\rrbra \xrightarrow{d}\llbra D\langle \otherdown \rangle \rrbra\to 0,
  \end{equation}
  where $d$ is given by a family of foams that are products except near the relevant crossing. Near the crossing, $d$ is given by
  the foam of Figure~\ref{fig:differential}.
  The map between $\llbra D\rrbra$ and
  $\llbra D\langle\smallloop\rangle\rrbra$ is given by
  \begin{equation}\label{eq:r1_map}
    \begin{tikzcd}
      0\ar[r] & \llbra D\rrbra\ar[d,"\phi"]\ar[r] & 0 \ar[d] \\
      0\ar[r] & \llbra D\langle\smallcircle\rangle\rrbra\ar[r] & \llbra D\langle\otherdown\rangle\rrbra \ar[r] & 0.
    \end{tikzcd}
  \end{equation}
  Here $\phi$ is given as follows. For any resolution $D_I$, we build the foam $\phi_I$ which is a product foam
  except at the relevant crossing, where it is given locally by Figure~\ref{fig:r1map}.
  In other words, \(\phi_{I}\) is the union of the identity foam \(D_{I} \times [0,1]\) on \(D_{I}\) and the cup foam.
  The resolution $D_I$ induces the resolutions $D_I\langle\smallcircle\rangle$ and $D_I\langle\otherdown\rangle$.
  The map between $D_I=D_I\langle\smalldown\rangle$ and 
  $D_I\langle\smallcircle\rangle$ is given by $(-1)^{\sad(I)}\phi_I$ for some $\sad(I)\in\F_2$. In \cite{MSV} it is proved that $\phi$
  is a chain homotopy equivalence, however, in their approach sign
  assignments are hidden in the Koszul's sign rule.
  
  The rest
  of the proof (of the Reidemeister one case)
  is devoted to chosing appropriate sign assignments on $\Cube(D)$ and~$\Cube(D\langle\smallloop\rangle)$
  in such a way that setting $\sad(I)=0$ for all $I$, we obtain a chain map.

  Recall that $\Cube(D)$ is
  the generalized cube of resolutions for $D$. As the new crossing is positive, $\Cube(D\langle\smallloop\rangle)=\Cube(D)\times\{0,1\}$
  (if the crossing is negative, we have to take the product with $\{-1,0\}$ instead).
  The next lemma covers both cases and will be extensively used in the future.
  \begin{lemma}\label{lem:sas}
    Suppose $\sas$ is a sign assignment for $D$. Assume that $D'$ is a diagram having one more crossing than $D$.
    Identify \(\Cube(D') \cong \Cube(D) \times \{0,\epsilon\}\), where \(\epsilon\) is the sign of the additional crossing.
    There exists a unique sign assignment $\sas'$ for $D'$ satisfying the following conditions.
    \begin{enumerate}[label=(S-\arabic*)]
    \item For $I,I'\in\Cube(D)$, such that $I'$ is an immediate successor of $I$, we have
      \[\sas'((I,0),(I',0))=\sas(I,I');\]
      \label{item:horizontal}
    \item For $I\in\Cube(D)$,
      \[\sas'((I,0),(I,1))=0\]
      if the new crossing is positive,
      and
      \[\sas'((I,-1),(I,0))=0\]
      if the new crossing is negative.
      \label{item:vertical}
    \end{enumerate}

    Moreover, suppose $\sas_1,\sas_2$ are two sign assignments on $D$ and $\sas_1-\sas_2= \partial \sat$, where \(\sat\) is a cellular \(1\)-cochain on \(\SCube(D)\).
    Denote by \(\sas_{1}'\) and \(\sas_{2}'\) extensions of \(\sas_{1}\) and \(\sas_{2}\), respectively, to sign assignments on \(\SCube(D')\).
    Define the cellular \(0\)-cochain \(\sat'\) on \(\SCube(D')\) by the  $\sat'((I,x))=\sat(I)$ for any \((I,x) \in \Cube(D')\).
    Then, $\sas_1'-\sas_2'= \partial \sat'$.
  \end{lemma}
  \begin{proof}
    We prove the lemma only for a positive crossing, the proof for a negative one is completely analogous.
    The only missing datum for $\sas'$ is $\sas'((I,1),(I',1))$ if $I'$ is an immediate successor of $I$.
    We set
    \begin{equation}
      \label{eq:sas_prim}
      \sas'((I,1),(I',1))=1+\sas(I,I').
    \end{equation}
    We verify that $\sas'$ satisfies cocycle condition. Take $I',I'_1,I'_2,I'_{12} \in \Cube(D')$ such that $I'_1\neq I'_2$ are immediate successors of $I'$ and $I'_{12}$ is an immediate successor of $I'_{1}$ and $I'_{2}$.
    There are three cases to consider.
    \begin{itemize}
    \item[(i)] Assume that the last coordinate of all $I',I'_1,I'_2,I'_{12}$ is zero, that is
      $I'=(I,0)$, $I'_1=(I_1,0)$, $I_2'=(I_2,0)$, $I'_{12}=(I_{12},0)$. By item~\ref{item:horizontal}:
      \[\sas'(I',I'_1)+\sas'(I'_1,I'_{12})+\sas'(I',I'_2)+\sas'(I'_2,I'_{12})=
      \sas(I,I_1)+\sas(I_1,I_{12})+\sas(I,I_2)+\sas(I_2,I_{12}).\]
      The latter is $1$ by the cocycle condition for $\sas$;
    \item[(ii)] Assume that the last coordinate of all $I',I'_1,I'_2,I'_{12}$ is $1$, that is, 
      $I'=(I,1)$, $I'_1=(I_1,1)$, $I_2'=(I_2,1)$, $I'_{12}=(I_{12},1)$. By \eqref{eq:sas_prim}:
      \[\sas'(I',I'_1)+\sas'(I'_1,I'_{12})+\sas'(I',I'_2)+\sas'(I'_2,I'_{12})=
      \sas(I,I_1)+\sas(I_1,I_{12})+\sas(I,I_2)+\sas(I_2,I_{12}).\]
      By the cocycle condition for $\sas$, the latter sum is $1$ as desired;
    \item[(iii)] For $I,I_1\in\Cube(D)$ and $I_1$ an immediate successor of $I$, we have $I'=(I,0)$, $I'_1=(I_1,0)$, $I'_2=(I,1)$, $I'_{12}=(I_1,1)$.
      Then,
      \[\sas(I'_2,I'_{12})=\sas(I',I'_1)+1,\ \sas(I',I'_2)=0,\ \sas(I'_1,I'_{12})=0,\]
      where the first equality follows from~\eqref{eq:sas_prim}, the next two follow from item~\ref{item:vertical}  in the statement of the lemma.
    \end{itemize}
    We have verified that the cocycle condition is satisfied for all three
    cases.
    Uniqueness is forced by (iii): if \eqref{eq:sas_prim} is not satisfied for some $I,I'$, then in (iii) the cocycle condition does not hold.

    We prove now the second part.
    Suppose $\sas_1$ and $\sas_2$ are sign assigments on $D$ such that $\sas_1-\sas_2=\partial\sat$. This means, for $I,I'\in\Cube(D)$
    such that $I'$ is an immediate successor of $I$, we have $\sas_1(I,I')-\sas_2(I,I')=\sat(I)-\sat(I')$. Consider $I'_1,I'_2\in\Cube(D\times\{0,1\})$ such that $I'_2$ is an immediate successor of $I'_1$. We have three cases.
    \begin{itemize}
    \item $I'_2=(I_2,0)$, $I_1'=(I_1,0)$ for $I_1,I_2\in\Cube(D)$ and $I_2$ an immediate successor of $I_1$.
      Then
      \begin{equation}\label{eq:sas_proof}
        \sas'_1(I'_1,I'_2)-\sas'_2(I'_1,I'_2)=\sas_1(I_1,I_2)-\sas_2(I_1,I_2)=\sat(I_1)-\sat(I_2)=\sat'(I_1)-\sat'(I_2).
      \end{equation}
      Here, the first equality follows from~\ref{item:horizontal}, the second is a translation of $\sas_1-\sas_2=\partial\sat$ (this is an assumption of the lemma), and the third follows from the definition of $\sat'$.
    \item $I'_2=(I_2,1)$, $I'_1=(I_1,1)$ for $I_1,I_2\in\Cube(D)$ and $I_2$ an immediate successor of $I_1$.
      Equation~\eqref{eq:sas_proof} holds, but the first equality follows from~\eqref{eq:sas_prim} and not~\ref{item:horizontal} as in the previous item. 
    \item $I'_2=(I,1)$, $I'_1=(I,0)$ for some $I\in\Cube(D)$. We have
      \[\sas'_1(I'_1,I'_2)-\sas'_2(I'_1,I'_2)=0=\sat'(I_1')-\sat'(I_2').\]
      Here the first equality follows from~\ref{item:vertical}, the second results from the definition of $\sat'$.
    \end{itemize}
  \end{proof}

  Lemma~\ref{lem:sas} gives us the following corollary.

  \begin{corollary}\label{cor:pos-signs-compatible-sign-assignments}
    If $\sas'$ is the sign assignment from Lemma~\ref{lem:sas},
    then we can choose \(\sad(I) = 0\), for any \(I \in \Cube(D)\), i.e., the \(I\)-th component of the map \(\phi \colon \llbra D,\sas\rrbra \to \llbra D\langle\smallcircle\rangle,\sas'\rrbra\) is given by $\phi_{I} \colon D_{I} \to D_{(I,0)}$ (there is no sign).
  \end{corollary}
  \begin{proof}
    The fact that $\phi_I$ commutes with the differential up to sign follows from the fact that differential in $\llbra D\rrbra$ 
    is given by local modifications
    of the diagram away from the new crossing. To verify the sign, we use the following routine procedure. Let $I_1\in\Cube(D)$
    and $I_2$ be an immediate successor of $I_1$. We let $I_1'=(I_1,0)$ and $I_2'=(I_2,0)$. The map $\phi_{I_1}$ postcomposed
    with the differential is the composition of the foams $\phi_{I_1}$ and $\delta(I_1',I_2')$ with the sign $(-1)^{\sad(I_1)+\sas'(I_1',I_2')}$.
    The differential postcomposed with $\phi_{I_2}$ is the composition of foams $\delta(I_1,I_2)$ and $\phi_{I_2}$
    multiplied by $(-1)^{\sad(I_2)+\sas(I_1,I_2)}$. Now $\sas(I_1,I_2)=\sas'(I_1',I_2')$ by \ref{item:horizontal}.
  \end{proof}
  Corollary~\ref{cor:pos-signs-compatible-sign-assignments} implies that the maps $\phi_I$ glue to the chain map between $\llbra D\rrbra$ and
  $\llbra D\langle\otherdown\rangle\rrbra$.
  In~\cite{MSV} it is proved that the map \(\phi\) is indeed a chain homotopy equivalence.
%
  \begin{figure}
    \begin{tikzpicture}
      \begin{scope}[xshift=-3cm,scale=0.8]
        \draw[thick] (2,2) arc [start angle=180, delta angle=360,x radius=0.5cm, y radius=0.2cm];
        \draw[thick] (1,0) .. controls ++(-1,0.2) and ++(1,-0.3) .. (-1,0.5);
        \draw[thick] (1,2) .. controls ++(-1,-0.2) and ++(1,-0.3) .. (-1,2);
        \draw[ultra thick, ->] (1.5,0.5) -- node [midway,left, scale=0.8] {$\phi$} (1.5,1.5);
      \end{scope}
      \begin{scope}[xshift=3cm,scale=0.8]
        \fill[blue!20!green!20] (2,2) .. controls ++(0.3,-1.5) and ++(-0.3,-1.5) .. (4,2) arc[start angle=0, delta angle=-180, x radius=1cm, y radius=0.3cm];
        \fill[blue!20!green!20] (1,0) .. controls ++(-1,0.2) and ++(1,-0.3) .. (-1,0.5) -- (-1,2) .. controls ++(1,-0.3) and ++(-1,-0.2) .. (1,2) -- (1,0);
        \draw[fill=blue!10,thick] (2,2) arc [start angle=180, delta angle=360,x radius=1cm, y radius=0.3cm];
        \draw (2,2) .. controls ++(0.3,-1.5) and ++(-0.3,-1.5) .. (4,2);
        \draw[thick] (1,0) .. controls ++(-1,0.2) and ++(1,-0.3) .. (-1,0.5);
        \draw[thick] (1,2) .. controls ++(-1,-0.2) and ++(1,-0.3) .. (-1,2);
        \draw (1,0) -- (1,2);
        \draw (-1,0.5) -- (-1,2);
      \end{scope}
    \end{tikzpicture}
    \caption{Reidemeister one move. To the left: the source and the target of the map $\phi$. To the right: the foam realizing this map (it is a product foam everywhere except near the crossing).}\label{fig:r1map}
  \end{figure}
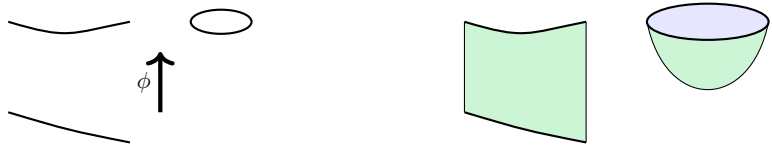

  We will now sketch the proof that the Reidemeister $2a$ move, that is
  
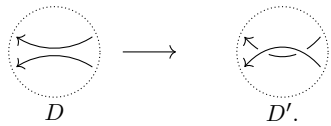
\[
    \begin{tikzpicture}
      \path[name path=A] (4,0.2) .. controls ++ (-0.3,-0.35) and ++ (0.3,-0.35) .. (3,0.2);
      \path[name path=B] (4,-0.2) .. controls ++ (-0.3,0.35) and ++ (0.3,0.35) .. (3,-0.2);
      \path[name intersections={of=A and B,by={C,D}}];
      \draw[->] (4,0.2) .. controls ++ (-0.3,-0.35) and ++ (0.3,-0.35) .. (3,0.2);
      \fill[white] (C) circle (0.08);
      \fill[white] (D) circle (0.08);
      \draw[->] (4,-0.2) .. controls ++ (-0.3,0.35) and ++ (0.3,0.35) .. (3,-0.2);
      \draw[->] (1,0.2) .. controls ++ (-0.3,-0.2) and ++ (0.3,-0.2) .. (0,0.2);
      \draw[->] (1,-0.2) .. controls ++ (-0.3,0.2) and ++ (0.3,0.2) .. (0,-0.2);
      \draw (0.5,-0.8) node [scale=0.8] {$D$};
      \draw (3.5,-0.8) node [scale=0.8] {$D'.$};
      \draw[densely dotted] (0.5,0) circle (0.6);
      \draw[densely dotted] (3.5,0) circle (0.6);
      \draw[->] (1.4,0) -- (2.1,0);
    \end{tikzpicture}
  \]
  does not change the chain homotopy type of $\llbra D\rrbra$.
  The left crossing in $D'$ is referred to as the first new crossing, while the crossing to the right of $D'$ is
  the second new crossing. We have $\Cube(D')=\Cube(D)\times\{0,1\}\times\{-1,0\}$. For a given sign assignment $\sas$ of $D$,
  we extend it to the sign assignment $\sas_1$ on $\Cube(D)\times\{0,1\}$ by Lemma~\ref{lem:sas}, and then to $\Cube(D)\times\{0,1\}\times\{-1,0\}$
  by Lemma~\ref{lem:sas} again. We obtain a sign assignment $\sas'$ on $\Cube(D')$. 
  We record the following observation for future use.
  \begin{lemma}\label{lem:sign_2a}
    The sign assignment $\sas'$ on $\Cube(D)\times\{1\}\times\{-1\}$ agrees with $\sas$.
  \end{lemma}
  \begin{proof}
    Let $I,I'\in\Cube(D)$ with $I'$ an immediate successor of $I$. By construction, via \eqref{eq:sas_prim}, $\sas_1((I,1),(I',1))=1+\sas(I,I')$.
    Applying \eqref{eq:sas_prim} again, we obtain $\sas'((I,1,-1),(I',1,-1))=1+\sas_1((I,1),(I',1))=\sas(I,I')$.
  \end{proof}

  We define the cochain map
  
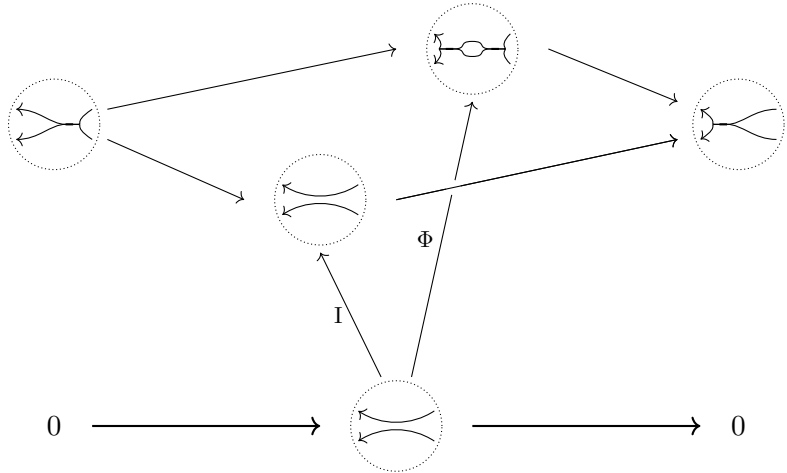
\[
    \begin{tikzpicture}
      \draw[->] (1,0.2) .. controls ++ (-0.3,-0.2) and ++ (0.3,-0.2) .. (0,0.2);
      \draw[->] (1,-0.2) .. controls ++ (-0.3,0.2) and ++ (0.3,0.2) .. (0,-0.2);
      \draw[densely dotted] (0.5,0) circle (0.6);
      \draw[->,thick] (1.5,0) -- (4.5,0); \draw (5,0) node {$0$};
      \draw[->,thick] (-3.5,0) -- (-0.5,0); \draw (-4,0) node {$0$};

      \draw[->] (-3.5,4.2) .. controls ++ (-0.3,-0.2) and ++(0.2,0) .. (-3.75,4) -- (-3.85,4) .. controls ++ (-0.2,0) and ++(0.3,0) .. (-4.5,4.2);
      \draw[->] (-3.5,3.8) .. controls ++ (-0.3,0.2) and ++(0.2,0) .. (-3.75,4) -- (-3.85,4) .. controls ++ (-0.2,0) and ++(0.3,0) .. (-4.5,3.8);
      \draw[thick] (-3.75,4) -- (-3.85,4);
      \draw[densely dotted] (-4,4) circle (0.6);

      \draw[->] (0,3.2) .. controls ++ (-0.3,-0.2) and ++ (0.3,-0.2) .. (-1,3.2);
      \draw[->] (0,2.8) .. controls ++ (-0.3,0.2) and ++ (0.3,0.2) .. (-1,2.8);
      \draw[densely dotted] (-0.5,3) circle (0.6);

      \draw[->] (5.5,4.2) .. controls ++ (-0.3,0) and ++(0.2,0) .. (4.85,4) -- (4.75,4) .. controls ++ (-0.2,0) and ++(0.3,-0.2) .. (4.5,4.2);
      \draw[->] (5.5,3.8) .. controls ++ (-0.3,0) and ++(0.2,0) .. (4.85,4) -- (4.75,4) .. controls ++ (-0.2,0) and ++(0.3,0.2) .. (4.5,3.8);
      \draw[thick] (4.85,4) -- (4.75,4);
      \draw[densely dotted] (5,4) circle (0.6);

      \draw[->] (2,5.2) .. controls ++ (-0.2,-0.2) and ++ (0.2,0) .. (1.85,5) -- (1.75,5) .. controls ++ (-0.2,0) and ++(0.2,0) .. (1.5,5.1) .. controls ++ (-0.2,0) and ++(0.2,0) .. (1.25,5) -- (1.15,5) .. controls ++ (-0.2,0) and ++(0.2,-0.2) .. (1,5.2);
      \draw[->] (2,4.8) .. controls ++ (-0.2,0.2) and ++ (0.2,0) .. (1.85,5) -- (1.75,5) .. controls ++ (-0.2,0) and ++(0.2,0) .. (1.5,4.9) .. controls ++ (-0.2,0) and ++(0.2,0) .. (1.25,5) -- (1.15,5) .. controls ++ (-0.2,0) and ++(0.2,0.2) .. (1,4.8);
      \draw[thick] (1.85,5) -- (1.75,5);
      \draw[thick] (1.25,5) -- (1.15,5);
      \draw[densely dotted] (1.5,5) circle (0.6);
      
      \draw[->] (-3.3,3.8) -- (-1.5,3);
      \draw[->] (-3.3,4.2) -- (0.5,5);
      \draw[name path=U,->] (0.5,3) -- (4.2,3.8);
      \draw[->] (2.5,5) -- (4.2,4.3);

      \draw[->] (0.3,0.65) -- node[left,midway,scale=0.8] {I} (-0.5,2.3);
      \draw[name path=V,->] (0.7,0.65) -- node[left,midway,scale=0.8] {$\Phi$} (1.5,4.3);
      \path[name intersections={of=U and V,by={X}}];
      \fill[white] (X) circle (0.1);
      \draw[->] (0.5,3) -- (4.2,3.8);

    \end{tikzpicture}
  \]
  The bottom diagram is a local cochain complex for $\llbra D\rrbra$, the top diagram is a local cochain complex for $\llbra D'\rrbra$.
  The map $I$ is given by the identity foam with sign $+1$. The map $\Phi$ is the foam depicted in Figure~\ref{fig:try_to_draw_it}
  also with sign $+1$. By Lemma~\ref{lem:sign_2a}, $\Phi$ is a cochain map
  between the complex $\llbra D\rrbra$ and the subcomplex of $D'$
  obtained by a $(-1,1)$-resolution of the crossing created in the Reidemeister 2a move.

  We refer to \cite{MSV} for showing that $I\oplus\Phi$ is a chain
  map, that  the description of the inverse map and that $I\oplus\Phi\colon\llbra D\rrbra\to\llbra D'\rrbra$
  is a cochain homotopy equivalence.

  \smallskip
  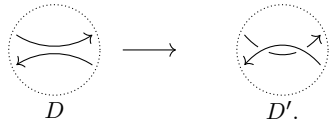
\begin{figure}
    \[
      \begin{tikzpicture}
        \path[name path=A] (4,0.2) .. controls ++ (-0.3,-0.35) and ++ (0.3,-0.35) .. (3,0.2);
        \path[name path=B] (4,-0.2) .. controls ++ (-0.3,0.35) and ++ (0.3,0.35) .. (3,-0.2);
        \path[name intersections={of=A and B,by={C,D}}];
        \draw[->] (3,0.2) .. controls ++ (0.3,-0.35) and ++ (-0.3,-0.35) .. (4,0.2);
        \fill[white] (C) circle (0.08);
        \fill[white] (D) circle (0.08);
        \draw[->] (4,-0.2) .. controls ++ (-0.3,0.35) and ++ (0.3,0.35) .. (3,-0.2);
        \draw[->] (0,0.2) .. controls ++ (0.3,-0.2) and ++ (-0.3,-0.2) .. (1,0.2);
        \draw[->] (1,-0.2) .. controls ++ (-0.3,0.2) and ++ (0.3,0.2) .. (0,-0.2);
        \draw (0.5,-0.8) node [scale=0.8] {$D$};
        \draw (3.5,-0.8) node [scale=0.8] {$D'.$};
        \draw[densely dotted] (0.5,0) circle (0.6);
        \draw[densely dotted] (3.5,0) circle (0.6);
        \draw[->] (1.4,0) -- (2.1,0);
      \end{tikzpicture}
    \]
    \caption{Reidemeister 2b move.}\label{fig:2b}
  \end{figure}
  The description of sign assignments for Reidemeister 2b move, drawn in Figure~\ref{fig:2b}, is the same. For Reidemeister 3, the problem of sign assignments does not appear. In fact, that move preserves the crossings, so the sign assignment on $D$ induces a natural sign
  assignment for $D'$.
\end{proof}
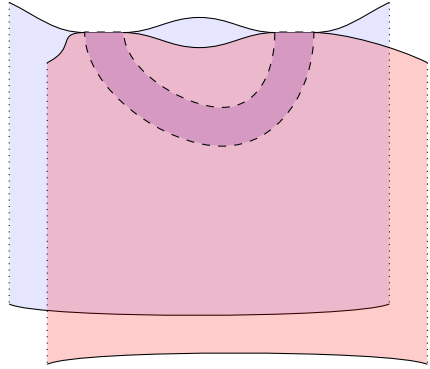
\begin{figure}
  \begin{tikzpicture}
    \fill[blue,opacity=0.1] (2.5,0.4) .. controls ++ (-0.4,-0.2) and ++ (0.4,0) .. (1.5,0) -- (1,0) .. controls ++ (-0.4,0) and ++(0.4,0) .. (0,0.2) .. controls ++ (-0.4,0) and ++(0.4,0) .. (-1,0) -- (-1.5,0) .. controls ++ (-0.4,0) and ++(0.4,-0.2) .. (-2.5,0.4) -- (-2.5,-3.6) .. controls ++ (0.4,-0.2) and ++ (-0.5,-0.2) .. (2.5,-3.6) -- (2.5,0.4);
    \draw (2.5,0.4) .. controls ++ (-0.4,-0.2) and ++ (0.4,0) .. (1.5,0) -- (1,0) .. controls ++ (-0.4,0) and ++(0.4,0) .. (0,0.2) .. controls ++ (-0.4,0) and ++(0.4,0) .. (-1,0) -- (-1.5,0) .. controls ++ (-0.4,0) and ++(0.4,-0.2) .. (-2.5,0.4);
    \draw (-2.5,-3.6) .. controls ++ (0.4,-0.2) and ++ (-0.5,-0.2) .. (2.5,-3.6);
    \draw[dotted] (2.5,0.4) -- (2.5,-3.6);
    \draw[dotted] (-2.5,0.4) -- (-2.5,-3.6);

    \fill[red,opacity=0.2] (3,-0.4) .. controls ++ (-0.4,0.2) and ++ (0.4,0) .. (1.5,0) -- (1,0) .. controls ++ (-0.4,0) and ++(0.4,0) .. (0,-0.2) .. controls ++ (-0.4,0) and ++(0.4,0) .. (-1,0) -- (-1.5,0) .. controls ++ (-0.4,0) and ++(0.4,0.2) .. (-2,-0.4)
    -- (-2,-4.4) .. controls ++ (0.4,0.2) and ++ (-0.5,0.2) .. (3,-4.4) -- cycle;
    \draw (3,-0.4) .. controls ++ (-0.4,0.2) and ++ (0.4,0) .. (1.5,0) -- (1,0) .. controls ++ (-0.4,0) and ++(0.4,0) .. (0,-0.2) .. controls ++ (-0.4,0) and ++(0.4,0) .. (-1,0) -- (-1.5,0) .. controls ++ (-0.4,0) and ++(0.4,0.2) .. (-2,-0.4);
    \draw (-2,-4.4) .. controls ++ (0.4,0.2) and ++ (-0.5,0.2) .. (3,-4.4);
    \draw[dotted] (3,-0.4) -- (3,-4.4);
    \draw[dotted] (-2,-0.4) -- (-2,-4.4);
    \draw[thick,dashed] (1,0) .. controls ++(0,-0.5) and ++(0.5,0) .. (0.3,-1) .. controls ++(-0.5,0) and ++(0,-0.5) .. (-1,0) -- (-1.5,0) .. controls ++(0,-0.75) and ++(-0.75,0) .. (0.3,-1.5) .. controls ++ (0.75,0) and ++ (0,-0.75) .. (1.5,0) -- cycle;
    \fill[blue!40!red!40] (1,0) .. controls ++(0,-0.5) and ++(0.5,0) .. (0.3,-1) .. controls ++(-0.5,0) and ++(0,-0.5) .. (-1,0) -- (-1.5,0) .. controls ++(0,-0.75) and ++(-0.75,0) .. (0.3,-1.5) .. controls ++ (0.75,0) and ++ (0,-0.75) .. (1.5,0) -- cycle;

  \end{tikzpicture}
  \caption{The map $\Phi$ for the Reidemeister $2a$ move in the proof of Theorem~\ref{thm:invariance_reid}. The dashed part is the seam singularity on the foam.}\label{fig:try_to_draw_it}
\end{figure}

The following result, see again \cite[Theorem 3.5]{ETW} and references in the proof therein, is a corollary of Theorem~\ref{thm:invariance_reid}.
\begin{theorem}\label{thm:invariance}
  If $D,D'$ represent the same link, then
  $\llbra D\rrbran\simeq \llbra D'\rrbran$, that is, the complexes in $\SFoam_N$
  corresponding to these two diagrams are chain homotopy equivalent.
\end{theorem}

\begin{definition}[\(\S_{N}\)-equivariant Khovanov--Rozanski homology]
  Let $L$ be a link with a diagram $D$. The homology of the cochain complex $\cF(\llbra D\rrbran)$ is called the $\S_N$-equivariant Khovanov--Rozanski homology of \(L\).
\end{definition}
By Theorem~\ref{thm:invariance}, the Khovanov--Rozanski homology does not depend, up to isomorphism, on the choice of the diagram.

We conclude this subsection with a general remark.
  \begin{remark}\label{rem:ambiguity-definition-of-maps}
    In Corollary~\ref{cor:pos-signs-compatible-sign-assignments} we showed that the choice $\sad(I)=0$ for all $I\in\Cube(D)$
    leads to a consistent sign assignment. However, the choice $\sad(I)\equiv 1$ will also lead to a chain map. This ambiguity
    is resolved by a choice, for example we might require that $\sad((0,\dots,0))=0$. In a more complex setting, like for changing orbits of crossings, we might need to admit
    more general choices of $\sad(I)$.
  \end{remark}

\section{Specializations}\label{sec:modules}
Fix $N>0$. Recall that we denote by $\S_N$ the ring of symmetric polynomials in $N$ variables with complex coefficients. 

\subsection{Algebraic specialization of modules}
Recall that $\cSym_N$ denotes the category of finitely generated, graded projective $\S_N$-modules. The grading is referred to as the \emph{quantum grading}.
As $\S_N$ is naturally
isomorphic to the ring of polynomials in $N$ variables, Quillen--Suslin theorem guarantees that projective $\S_N$-modules are free.
That is to say,
an object of $\cSym_N$ is a direct sum of finite number of copies of $\S_N$, possibly with a grading shift.
We indicate such shift by $\S_N\{q^a\}$ with $a\in\Z$. 
To be more precise, if we denote by \((\S_{N})^{k}\) the \(k\)-graded part of \(\S_{N}\), then \((\S_{N} \{q^{a}\})^{k} = (\S_{N})^{k-a}\).
A degree-\(k\) morphism $\phi\,\colon\,\S_N\{q^a\}\to\S_N\{q^b\}$, is the same as a map \(\S_{N} \to \S_{N}\) of degree \(k+b-a\), hence it is a multiplication by a homogeneous polynomial of degree \(\frac{k+b-a}{2}\) (recall that
the indeterminates have degree two).

Denote by \(\Sym^{N}(\C)\) the \(N\)-th symmetric product of \(\C\), i.e., \(\Sym^{N}(\C) = \C^{N} / S_{N}\), where the symmetric group~\(S_{N}\) acts on \(\C^{N}\) by permuting coordinates.
In other words, elements of \(\Sym^{N}(\C)\) are unordered \(N\)-tuples of complex numbers.
Let $\Sigma = [z_{1},z_{2},\ldots,z_{N}] \in \Sym^{N}(\C)$.
We say that \(\Sigma\) is \emph{generic} if the elements of \(\Sigma\) are pairwise distinct, and \emph{singular} if \(\Sigma = [0,0,\ldots,0]\).
For a polynomial $P\in\S_N$, we denote by $P(\Sigma)$ complex number resulting from the evaluation of $P$ at $\Sigma$.
As $P$ is symmetric, the value $P(\Sigma)$ is well-defined.
The $N$-tuple $\Sigma$ endows $\C$ with a structure of a left $\S_N$-module, via $P\cdot z=P(\Sigma)z$, for $P\in\S_N$ and $z\in\C$.
We will denote by \(\C_{\Sigma}\) the \(\S_{N}\)-module structure on \(\C\) determined by \(\Sigma\).

\begin{definition}[Specialization functor]\label{def:spec_fun}
  The functor
  \[\ev^\Sigma \colon \cSym_N \to \Vect(\C), \quad \ev^{\Sigma}(M) = M \otimes_{\S_{N}} \C_{\Sigma},\]
  is called the \emph{specialization functor}.
  The functor \(\ev^{\Sigma}\) is called generic or singular, according to whether \(\Sigma\) is generic or singular.
\end{definition}
\begin{remark}
  We note that the specialization functor is usually not exact,
  because $\C_{\Sigma}$ is not a flat $\S_N$-module.
\end{remark}
\subsection{Algebraic specialization of cochain complexes}\label{sub:algebraic_spec_chain}

Suppose $C^*$ is a cochain complex of graded, free $\S_N$-modules. 
We form two cochain complexes.
\begin{definition}[Generic and singular specialization of complexes]\label{def:singular_and_generic}\
  \begin{itemize}
  \item The \emph{singular specialization} $C^*_0$ is the complex obtained by applying the singular specialization functor $\ev^\Sigma$ to $C^*$ with $\Sigma=[0,\dots,0]$;
  \item The \emph{generic specialization} $C^*_{gen}$ is the complex obtained by applying the generic specialization functor $\ev^\Sigma$ to $C^*$ with $\Sigma$ generic.
  \end{itemize}
\end{definition}
Both the singular specialization and the generic specialization are cochain complexes whose underlying vector 
spaces carry a grading. 
For instance,
if
\begin{equation}\label{eq:Ci_complex}C^i=\bigoplus_{j=1}^{n_i} \S_N\{q^{a_{ij}}\},\end{equation}
then $C^i_0=C^i_{gen}=\bigoplus_{j=1}^{n_i}\C\{q^{a_{ij}}\}$. We study how does the differential behave with respect to the grading.
Suppose the cochain complex $C^*$ is written as in \eqref{eq:Ci_complex}.
We can split the differential $d^{i}\colon C^{i}\to C^{i+1}$ into components $d_{i,k\ell}\colon \S_N\{q^{a_{ik}}\}\to\S_N\{q^{a_{i+1,\ell}}\}$,
where $1\le k\le n_i$, $1\le \ell\le n_{i-1}$.
The map $d_{i,k\ell}$, having degree \(0\), is the multiplication by a homogeneous polynomial of degree $\frac{a_{i+1,l}-a_{i,k}}{2}$.
The singular evaluation of a homogeneous polynomial of degree $s = \frac{a_{i+1,l}-a_{i,k}}{2}$ can be non-zero only if $a_{i,k}=a_{i+1,\ell}$. That is to say, the differential $d^i_0$ of the complex $C^*_0$ preserves the grading.

The situation of a cochain complex $C^*_{gen}$ is different, because the evaluation can be non-zero if $a_{i+1,\ell}-a_{i,k}\ge 0$.
That is to say, the differential $d^i_{gen}$ does not decrease the grading.
In particular, $(C^i_{gen},d^i_{gen})$ is a \emph{filtered} cochain complex.

For later use we record the following statement.
\begin{proposition}\label{prop:abstract_lee_gornik}
  There exists a spectral sequence, whose first page is $H^*(C^*_0)$ and whose $E^\infty$-page is $H^*(C^*_{gen})$.
\end{proposition}
\begin{proof}
  The differentials $d^i\colon C^i\to C^{i+1}$ can be decomposed as a sum $d^{i0}+d^{i1}+\dots$,
  where $d^{is}$ is given by a matrix of homogeneous polynomials of degree $s$.
  After performing a generic specialization, $d^{is}$ becomes the map $d^{is}_{gen}$ increasing the grading by $2s$. That is, $d^{i}_{gen}=d^{i0}_{gen}+d^{i1}_{gen}+\dots$.
  The graded part of $d^{i}_{gen}$ is equal to $d^{i0}_{gen}$.

  Specialization of $d^{is}$ with all variables zero gives the zero map, unless $s=0$. That is, $d^{i}_0=d^{i0}_0$.
  The non-zero map $d^{i0}_0$ is equal to
  $d^{i0}_{gen}$, because a degree zero polynomial is necessarily constant.
  Therefore, the graded part of $d^{i}_{gen}$ is equal to the differential $d^{i}_0$.

  Summarizing, $(C^*_{gen},d^i_{gen})$ is a filtered cochain complex, whose graded part is $d^i_0$. A classical argument  shows the existence of the spectral sequence.
\end{proof}

\subsection{Geometric specialization}\label{sub:specsigma}
We have already discussed the algebraic specialization functor $\ev^\Sigma$.
We can perform the specialization earlier, on the level of the category $\SFoam_N$.
The resulting category, $\SigFoam_N$, will be shown to have a much simpler structure.

\begin{definition}
  Fix \(\Sigma \in \Sym^{N}(\C)\).
  The evaluation $\langle F\rangle_\Sigma$ of a closed foam $F$ is a complex number obtained by the following recipe.
  Recall the evaluation functor \(\cF \colon \SFoam_{N} \to \cSym_{N}\), from Definition~\ref{def:evaluation_2}.
  Since \(\cF(F)\) is a symmetric polynomial in \(N\) variables, we define \(\langle F \rangle_{\Sigma} := (\ev^{\Sigma} \circ \cF)(F)\), i.e., we evaluate the symmetric polynomial~\(\cF(F)\) on \(\Sigma\).
\end{definition}

We can now define the category $\SigFoam_N$.

\begin{definition}[The category $\SigFoam_N$]\label{def:sigfoam}
  The category $\SigFoam_N$ is the category with the same objects as $\SFoam^*_N$, that is, the objects are formal direct sums of webs with formal grading shifts.
  Morphisms in \(\SigFoam_{N}\) are given by
  \[\Hom_{\SigFoam_N}(V,W):=\Hom_{\SFoam^*_N}(V,W)/\cI_\Sigma(V).\]
\end{definition}
Here, $\cI_\Sigma$ is defined by analogy to the construction of Subsection~\ref{sub:catego}. For $G\in\Hom_{\SFoam_N^*}(V,\emptyset)$,
we set $\phi_{G,\Sigma}\colon\wt{F}(V)\to\C$ via $\phi_{G,\Sigma}(F)=\langle G\circ F\rangle_\Sigma$ and we set $\cI_\Sigma(V)=\bigcap\ker\phi_{G,\Sigma}$. Note that $\Hom_{\SigFoam_N}(V,W)$ has a structure of a $\C$-vector space.

The functor $\cF_\Sigma$ descends to a functor from the category $\SigFoam_N$ to the category of vector spaces.
Namely, we set
\[\cF_\Sigma(V)=\Hom_{\SigFoam_N}(\emptyset,V).\]
We will use $\cF_\Sigma$ to define a perturbed $\sln$-homology and to describe the Lee-Gornik spectral sequence. Before that,
we discuss gradings. For this we distinguish two special cases.

We let $\OFoam_N$ be the singular specialization of \(\SFoam_{N}\), i.e., $\OFoam_{N} := \SigFoam_N$ for $\Sigma=[0,\dots,0]$.
We let $\Sigma' \in \Sym^{N}(\C)$ be generic.
We write $\SigPFoam_N$ for the $\SigFoam_N$ category relative to $\Sigma'$.


The category $\OFoam_N$ is graded. Indeed, each foam $F$ of positive degree evaluates to $0$ under $\cF_\Sigma$, hence
all the morphisms in $\OFoam_N$ are given by foams of degree $0$ between $q^kW$ and $q^\ell W'$ with $k=\ell$.

The degree of a foam does not behave well under specializing to other variables than zero. 
A foam $F$ from $W$ to $W'$ induces a morphism $q^kW\to q^\ell W'$ in $\SFoam_N$, where $k-\ell=\deg F$.
Now after specializing to $\SigPFoam_N$, the foam $F$ gives a linear map from $q^kW$ to $q^\ell W'$. The map can be non-trivial
even if $k<\ell$. In particular, there might be non-trivial morphisms
in $\SigPFoam_N$, between $q^kW$ and $q^\ell W'$ as long as $k \leq \ell$. In other words, $\SigPFoam_N$ is a \emph{filtered} category.

\subsection{Geometric versus algebraic specialization}
The next goal is to relate geometric and algebraic specializations. Consider the category $\SFoam_N$.
There is a functor $\cF$ from $\SFoam_N$ to the category $\cSym_N$.
For $\Sigma \in \Sym^{N}(\C)$, there is a projection functor from $\SFoam_N$ to $\SigFoam_N$.
Indeed, $\SigFoam_N$ and $\SFoam_N$ are both quotient categories of $\SFoam^*_N$.
Furthermore, the projection functor \(\SFoam_{N}^{\ast} \to \SigFoam_{N}\) factors through the projection \(\SFoam_{N}^{\ast} \to \SFoam_{N}\).
Indeed, to pass from $\SFoam^*_N$ to $\SFoam_N$ we mod out by foams for which $\langle F\rangle$ is zero,
to pass from $\SFoam^*_N$ to $\SigFoam_N$ we mod out by foams for which $\langle F\rangle_{\Sigma}$ is zero.
Consequently, we have the following diagram of functors.

\[
  \begin{tikzcd}
    \SFoam_N \ar[r,"\cF"]\ar[d] & \cSym_N\ar[d,"\ev^\Sigma"] \\ \SigFoam_N\ar[r,"\cF_\Sigma"] & \Vect_\C.
  \end{tikzcd}
\]
Here, $\Vect_\C$ is the category of filtered vector spaces over $\C$.
\begin{proposition}\label{prop:diagram_commutes}
  The above diagram of functors commutes.
\end{proposition}
\begin{proof}
  This is the statement of \cite[Proposition 4.1]{RobertWagner}.
\end{proof}
We now define the two main cohomology spaces (over $\C$) associated with $L$.
\begin{definition}
  The \emph{Khovanov--Rozansky $\sln$-homology of $L$} (at homological grading $k$ and quantum grading
  $r$) is defined $\KR^{k,r} := H^k(\ev^\Sigma\circ\cF(\llbra D\rrbra))$ with quantum grading $r$, where $D$ is a diagram representing $L$ and \(\Sigma\) is singular.

  The cohomology space $\Lee_{\Sigma}^{k}(L)$, called the \emph{Lee $\sln$-homology} of $L$ at homological grading $k$
  is the space $H^k(\cF_\Sigma(\llbra D\rrbran))$, for $\Sigma$ generic.
\end{definition}

We remark that while $\Lee_N^{k}(L)$ is not graded, it has a $\Z$-filtration.
It makes sense to speak of the associated graded vector space $\Gr^r\Lee_N^{k}(L)$.

The following result, generalizing the statement of Eu Son Lee \cite{Lee} was first proved by Gornik \cite[Theorem 1]{Gornik}.
We refer to \cite{LobbLewark} for a detailed study of this spectral sequence.
\begin{theorem}[Lee-Gornik spectral sequence]\label{thm:leegornik}
  Let $D$ be a link diagram and fix a generic \(\Sigma \in \Sym^{N}(\C)\). 
  There is a spectral sequence whose first page is $\KR_N^{k,r}(L)$ abuting to $\Lee_{\Sigma}^{k}(L)$.
\end{theorem}
\begin{proof}
  Consider the cochain complex $C^*=\cF(\llbra D\rrbran)$ over $\S_N$. The cochain complexes $C^*_0=\cF_{0}(\llbra D\rrbran)$ and
  $C^*_{gen}=\cF_{\Sigma}(\llbra D\rrbran)$ are the specializations of $C^*$ in the sense of Subsection~\ref{sub:algebraic_spec_chain}. 
  The statement follows from Proposition~\ref{prop:abstract_lee_gornik}.
\end{proof}

\subsection{Mirrors}
Suppose $L$ is a link and $L'$ is its mirror.
It is well-known that Khovanov homologies of $L$ and of $L'$ are related. We now describe
an analogous result for $\sln$-homology. The statement should be well-known to the experts, however, we could not find a precise statement
in the literature. To fix the terminology, we will assume that if $D$ is a diagram for $L$, the diagram $D'$ for $L'$
is obtained by changing all crossings of $D$.

\begin{proposition}\label{prop:mirror}
  We have an isomorphism $\KR^{k,r}_N(L)\cong\KR^{-k,-r}_N(L')$.
\end{proposition}
\begin{proof}
  Let $D$ be a diagram representing $L$. Assume it has $n$ crossings,
  which we enumerate $\Cr(D)=\{1,\dots,n\}$.
  Consider the generalized cube of resolutions $\Cube(D)$ for $L$ with a given sign assignment; see Subsection~\ref{sub:sym_eq}. 
  To each vertex $I$ we associate
  $\cF_0(D_I)$, where $D_I$ is the corresponding resolution of $D$. The differentials are induced from
  differentials in 
  Figure~\ref{fig:differential} via the functor $\cF_0$.

If $D'$ is the diagram
  of the mirror, there is a correspondence of crossings of $D$ and $D'$.
  For $I\in\Cube(D)$, $I=(i_1,\dots,i_n)$, denote by $I'$ the `dual' resolution $(-i_1,\dots,-i_n)\in\Cube(D')$.  The webs $D'_{I'}$ and $D_I$ are isomorphic by construction, but the associated homological and quantum gradings
  are inverted. This is due the fact that the resolution of the negative crossing is dual to the resolution of the positive
  crossing, see Subsection~\ref{sub:sym_eq}. 
  That is, there exists a map $\iota\colon\cF_0(\llbra D\rrbra)\to\cF_0(\llbra D'\rrbra)$ of $\S_N$-modules given by
  the identity $\iota(\cF_0(D_I))=\cF_0(D'_{I'})$: we say `identity' because the resolutions $D_I$ and $D'_{I'}$ are the same.
  In particular, $\iota$ is canonically defined.

  The differentials in the mirror complex are inverted. That is, if a differential from $\cF_0(D_{I_1})$ to $\cF_0(D_{I_2})$
  is given a matrix in some basis, the differential from $\cF_0(D'_{I'_2})$ to $\cF_0(D'_{I'_1})$ is given by the transpose matrix.

  All the above means that, with a fixed choice of bases of $\cF_0(D_I)$, we have an identification of cochain complexes
  \[\cF_0(\llbra D'\rrbra)=\Hom_{\C}(\cF_0\llbra D\rrbra,\C)\]
  with underlying gradings reversed. As $\C$ is a field, we obtain
  \[H^{k,r}(\llbra D'\rrbra)\cong H^{-k,-r}(\llbra D\rrbra).\]
  This is precisely the statement of the theorem.
\end{proof}
\begin{remark}
  The above proof does not immediately generalize to show the relations between $\sln$-homology over $\S_N$ of a link with its mirror.
  While it is true that \(\cF(\llbra D'\rrbra)\) is the dual cochain complex of \(\cF(\llbra D\rrbra)\), the Universal Coefficient Spectral Sequence shows that in general it is not true that \(H^{k,r}(\cF(\llbra D\rrbra))
  \not\cong H^{-k,-r}(\cF(\llbra D\rrbra))\).
\end{remark}


\subsection{Computation of Lee-Gornik homology}\label{sub:properties_of_lee}
Computation of Lee-Gornik homology relies on the same principle as computation of Bar-Natan's perturbation of Khovanov
homology \cite{Bar_Morrison}. A throughout discussion of the $\sln$ case is done in \cite{RoseWedrich}.

Let $f$ be a facet of a foam with label $a$. We consider the algebra $\cA_f$ of all possible decorations of $f$ (recall that a decoration assigns to $f$ a symmetric polynomial in $a$ variables) modulo all decorations that make $F$ be a zero map in $\SigPFoam_N$. The algebra structure is given by addition and multiplication of polynomials
associated to the face. This is called
the \emph{algebra of decorations}.
\begin{theorem}[see \expandafter{\cite[Lemma 4.2]{RoseWedrich}, \cite[Lemma 2.28]{ETW}}]\label{thm:decoration_face}
  The algebra of decorations $\cA_f$ is a direct sum of one dimensional summands indexed by subsets $A\subset\Sigma$ of cardinality $a$.
  There is a generator of each summand, $\one_A$, which is an idempotent in $\cA_f$.
\end{theorem}

Suppose now $W$ is a web and $F$ is the identity foam from $W$ to $W$. Consider the algebra $\cA_F$ of all possible decorations of $F$
modulo decoration evaluating to zero under $\cF_\Sigma$. Formally,
$\cA_F$ is a quotient of the
direct sum of algebras $\cA_f$ associated to faces of $F$.
\begin{theorem}[see \expandafter{\cite[Lemma 3.10]{ETW}}]\label{thm:decoration_foam}
  The algebra $\cA_F$ is the direct sum of one-dimensional summands. The summands are in bijection with colorings of all facets
  by subsets of $\Sigma$ as in Theorem~\ref{thm:decoration_face} satisfying the admissibility condition of
  Definition~\ref{def:coloring_foam}.
\end{theorem}

Recall, after~\cite[Section 3.2]{ETW} that if $\cC$ is a category, then the \emph{Karoubi envelope} of $\cC$ is the category
obtained by formally splitting all idempotents of $\cC$. In more detail, the category $\Kar(\cC)$ has as objects pairs $(O,e)$,
where $O$ is an object in $\cC$ and $e\colon O\to O$ is an idempotent. A morphism between $(O,e)$ and $(O',e')$ is a map
$f\in\Mor_{\cC}(O,O')$ such that $f\circ e=e'\circ f$.

Consider the category $\SigPFoam_N$. Let $W$ be a web and let $F$ be the identity foam. A decoration of the identity foam by
subsets of $\Sigma'$ as in Theorem~\ref{thm:decoration_foam} induces a decoration of $W$ by elements of $\Sigma'$ and vice versa.
In particular, a decoration of $W$ by elements of $\Sigma'$ corresponds to an idempotent in $\SigPFoam_N$. A decorated web can be regarded as an object in the Karoubi envelope of $\SigPFoam_N$; compare \cite[Section 4.2]{RoseWedrich}.
\begin{definition}[Category $\KO$]\label{def:KO}
The category $\KO$ is the full subcategory of the Karoubi envelope of $\SigPFoam_N$  containing all objects of the form \((W,F_{W})\), where \(W\) is an \(N\)-web and \(F_{W}\) is an identity foam colored by
subsets of \(\Sigma'\).
\end{definition}
\begin{example}[see \expandafter{\cite[Corollary 3.19]{RoseWedrich}}]
  Let $W$ be a web.
  In $\KO$ we have the following relation:
  \[W=\sum_{\cD} (W,\cD),\]
  where the sum is over all admissible decorations $\cD$ of $W$.
\end{example}

The purpose of introducing the category $\KO$ is the following result.

\begin{theorem}[see \expandafter{\cite[Lemma 3.13]{ETW}, and \cite[Lemma 5.9]{RobertWagner}}]\label{thm:trivial_diff}
  Let $D$ be the diagram of a link and let $\llbra D\rrbra_{\Sigma'}$ be the complex associated to $D$ in $\SigPFoam_N$ (i.e. the bracket $\llbra D\rrbra$ after geometric specialization in $\SigPFoam_N$). In $\KO$ the
  complex is isomorphic to the complex with trivial differentials. We have locally:
  \begin{equation}\label{eq:local_Sigma}
    \raisebox{-2em}{
    \begin{tikzpicture}
      \draw[->] (0.3,-0.5) -- (-0.3,0.5);
      \fill[white] (0,0) circle (0.1);
      \draw[->] (-0.3,-0.5) -- (0.3,0.5); 
      \draw(0.3,-0.6) node[scale=0.7]{\ensuremath{b}};
      \draw(-0.3,-0.6) node[scale=0.7]{\ensuremath{a}};
      \draw(-0.4,-0.5) -- (-0.48,-0.5) -- (-0.48,0.5) -- (-0.4,0.5);
      \draw(-0.44,-0.5) -- (-0.44,0.5);
      \draw(0.4,-0.5) -- (0.48,-0.5) -- (0.48,0.5) -- (0.4,0.5);
      \draw(0.44,-0.5) -- (0.44,0.5);
      \draw(0.6,0.52) node [scale=0.7]{\ensuremath{\Sigma}};
  \end{tikzpicture}}
    \cong
    \bigoplus_{\substack{A,B\subset\Sigma\\ |A|=a\\|B|=b}}
    \raisebox{-2em}{
    \begin{tikzpicture}
      \draw[->] (0.3,-0.5) -- (-0.3,0.5);
      \fill[white] (0,0) circle (0.1);
      \draw[->] (-0.3,-0.5) -- (0.3,0.5); 
      \draw(0.3,-0.6) node[scale=0.7]{\ensuremath{B}};
      \draw(-0.3,-0.6) node[scale=0.7]{\ensuremath{A}};
  \end{tikzpicture}}
  \cong
  \bigoplus_{k=0,\dots,b}
  \bigoplus_{\substack{A,B\subset\Sigma\\ |B\setminus A|=k}}t^k
  \raisebox{-2em}{\begin{tikzpicture}
      \begin{scope}[very thick]

    \draw[mymidarrow,purple] (0.5,-0.3) -- node[midway, above,scale=0.6] {$B\setminus A$} (-0.5,-0.15);
    \draw[mymidarrow,purple] (-0.5,0.15) -- node[midway, above,scale=0.6] {$A\setminus B$} (0.5,0.3);
      \draw[myendarrow,mybegarrow,mymidarrow] (0.5,-0.5) -- (0.5,0.5);
      \draw[myendarrow,mybegarrow,mymidarrow] (-0.5,-0.5) -- (-0.5,0.5);
      \draw(0.5,-0.6) node[scale=0.7]{\ensuremath{B}};
      \draw(-0.5,-0.6) node[scale=0.7]{\ensuremath{A}};
      \draw(0.5,0.6) node[scale=0.7]{\ensuremath{A}};
      \draw(-0.5,0.6) node[scale=0.7]{\ensuremath{B}};
  \end{scope}
  \end{tikzpicture}}
  \end{equation}
\end{theorem}

Note that $\cF_{\Sigma'}$ applied to $\llbra D\rrbra_{\Sigma'}$
yields the $\sln$ Lee homology by Proposition~\ref{prop:diagram_commutes}. Therefore, Theorem~\ref{thm:trivial_diff} allows us to compute
Lee homology for labelled links. In fact,
the group $\Lee_N(L)$ was computed by Gornik \cite[Theorem 2]{Gornik}.
\begin{proposition}\label{prop:computingLee}
  Let $L$ be a link. Then $\Lee_N(L)$ is isomorphic to the vector space over $\C$ on $N^{\#L}$ generators.
  Moreover, to each map $\psi\colon \{\textrm{components of L}\}\to\{1,\dots,N\}$, we can assign a class $\sal_\psi\in\Lee_N(L)$
  of homological degree
  \[\deg(\sal_\psi)=\sum_{\substack{a,b=1,\dots,n\\ a\neq b}} \lk(\psi^{-1}(a),\psi^{-1}(b)).\]
  The classes $\sal_\psi$ generate $\Lee_N(L)$.
\end{proposition}
We remark that the isomorphism of $\Lee_N(L)$ with $\C^{N^{\#L}}$ is an immediate consequence of the first equation in
\eqref{eq:local_Sigma}. Computation of homological gradings requires non-trivial arguments.

\section{$\sln$-homology for periodic links}\label{sec:periodic}
In this section we let $G=\Z_m$ be the group acting on $\R^2\times\R$ by rotation about the axis $(0,0)\times \R$.
Assume $L$ is a link in $\R^3$ which is preserved by $\Z_{m}$ and disjoint from the rotation axis. Let $D$ be a periodic link diagram for $L$.

\subsection{Group actions on $\llbra D\rrbra$}
Before proceeding to the main part of this section, let us make some preparations.
For any finite group \(G\), let \(BG\) denote the category with a single object \(\ast\) and \(\Hom_{BG}(\ast,\ast) = G\).
Suppose $\cB$ is an additive category, functors \(X \colon BG \to \cB\) correspond bijectively to objects of \(\cB\) equipped with an action of \(G\).
We refer to such functors as $G$-objects.
We will denote by \(\cB[G] := \operatorname{Fun}(BG,\cB)\) the category of functors \(X \colon BG \to \cB\).
Equivalently, \(\cB[G]\) is the category of \(G\)-objects in \(\cB\).
Morphisms in \(\cB[G]\) are natural transformation of functors \(X \colon BG \to \cB\), which translate into \(G\)-equivariant maps of \(G\)-objects.
In particular, taking \(\cB = \cSym_{N}\) and \(G = \Z_{m}\), we obtain the category \(\cSym_{N}[\Z_{m}]\) of \(\S_{N}[\Z_{m}]\)-modules, which are free as \(\S_{N}\)-modules.
Similarly, we can consider the category \(\SFoam_{N}[\Z_{m}]\) of \(N\)-foams equipped with an action of \(\Z_{m}\).
In the same manner, we obtain categories
$\Kom{\SFoam_{N}}[\Z_{m}]$
and
$\Kom{\cSym_{N}}[\Z_{m}]$.

To construct the $\Z_{m}$-equivariant $\sln$-homology of a periodic link, we need to prove three things, namely:
\begin{itemize}
\item Existence of the action of $\Z_{m}$ on $\llbra D\rrbra$ (Proposition~\ref{prop:group_action_on_D});
\item Equivariance of the evaluation functor $\cF$, implying the existence of a $\Z_{m}$-action on $\cF(\llbra D\rrbra)$ (Proposition~\ref{prop:functoriality});
\item Independence of the action on the choice of a periodic diagram (Theorem~\ref{thm:group_action}).
\end{itemize}
  
We begin with the first part.
\begin{proposition}\label{prop:group_action_on_D}
  Suppose $D$ is a periodic link diagram. Then, there is an action of $\Z_{m}$ on $\llbra D\rrbra$ induced by rotating resolution diagrams of \(D\).
  In other words, \(\llbra D \rrbra\) becomes an object in \(\Kom{\SFoam_{N}}[\Z_{m}]\).
\end{proposition}
\begin{proof}
  The proof resembles the argument of~\cite{Politarczyk-Khovanov}.
  Fix a generator \(g \in \Z_{m}\) and let \(\rho_{g} \colon D \to D\) be a cobordism realizing the rotation by \(g\).
  We choose specific \(\rho_{g}\), which is the trace of a path of rotations of \(\R^{3}\), where the rotation angle increases linearly from~\(0\) to~\(\frac{2 \pi}{m}\).
  To be more precise, we consider the map \(r
  \colon \R^{3} \times I \to \R\) with \(r(-,t)\) being the rotation of the plane by \(\frac{2 \pi t}{m}\). Then \(\rho_{g} = r(D \times I)\).
  As $\Z_{m}$ acts on $D$ and permutes cyclically crossings of $D$, there is an induced action of $\Z_{m}$
  on $\Cube(D)$, which we denote by $(g,I)\mapsto gI$.
  The action of \(\Z_{m}\) on \(\Cube(D)\) extends to an action on \(\SCube(D)\).
  For a given diagram, $gD_I=D_{gI}$, where $g$ acts on $D_I$ by rotation.

  Consider a sign assignment $\sas$ on $D$. 
  For a fixed generator \(g \in \Z_{m}\), we define the sign assignment $g\sas$ via the following formula.
  \begin{equation}\label{eq:gsas_def}
    g\sas(gI,gI'):=\sas(I,I').
  \end{equation}
  The sign assignment $g\sas$ need not be equal to $\sas$, but we know that 
  \begin{equation}\label{eq:gsas}
    g\sas - \sas = \partial \sat
  \end{equation}
  for some $0$-cochain $\sat$ on \(\SCube(D)\).
  By Lemma~\ref{lem:sas}, upon adding a constant cochain,
  if necessary, we may and will assume that $\sat$ vanishes on the element $(0,0,\dots,0)$ in \(\Cube(D)\). 
  Define \(\cG_{g} := \llbra \rho_{g}, \sat \rrbra\).
  To be more precise, for any \(I\) in \(\Cube(D)\), the \(I\)-th component of the map \(\cG_{g}\) is defined as
  \[\cG_{g,I} = (-1)^{\sat(I)} \rho_{g,I}.\]
  
  We argue that \(\cG_{g}\) is a cochain map, i.e., if $d$ is the differential on \(\llbra D \rrbra\), then $d \cG_{g}= \cG_{g} d$.
  To see this, take resolutions $I,I'\in\Cube(D)$ such that $I'$ is an immediate successor of $I$.
  Let $\delta_{I,I'}$ be the foam giving the component of the differential from $I$ to $I'$ as in Figure~\ref{fig:differential} (up to sign determined by $\sas(I,I')$). 
  We have the following diagram.
  \[
    \begin{tikzcd}[column sep=2cm]
      D_I\ar[r,"(-1)^{\sat(I)} \rho_{g,I}"]\ar[d,"(-1)^{\sas(I,I')}\delta_{I,I'}"'] & D_{gI}\ar[d,"(-1)^{\sas(gI,gI')}g\delta_{I,I'}"]\\
      D_{I'}\ar[r,"(-1)^{\sat(I')}\rho_{g,I'}"] & D_{gI'}.
    \end{tikzcd}
  \]
  The vertical maps are the components of the differential in $\llbra D\rrbra$ from $D_I$ to $D_{I'}$ (the left vertical map) and from $D_{gI}$ to $D_{gI'}$ (the right vertical map). The horizontal maps
  are given by~$\cG_{g,I}$ and~$\cG_{g,I'}$, respectively.

  The foams $\rho_{g,I'}\circ \delta_{I,I'}$ and $g\delta_{I,I'}\circ\rho_{g,I} = \delta_{gI,gI'} \circ \rho_{g,I}$ are isotopic.
  By the definition of $\sat$, the diagram commutes, hence $\phi$ and $g\phi$ have the same contribution to the differential.
  That is, $g$ commutes with the differential, so $\cG_g$ is a chain map in $\KFoam$.

  It remains to prove that  \((\cG_{g})^{m} = \operatorname{id}\).
  Note that $\cG_{g^2}(D_I)=(-1)^{\sat(I)+\sat(gI)} \rho_{g,gI} \circ \rho_{g,I}$ and by induction
  \[\cG_{g^m}(D_I)=(-1)^{\sat(I)+\dots+\sat(g^{m-1}I)} \rho_{g,g^{m-1}I} \circ \cdots \circ \rho_{g,gI} \circ \rho_{g,I} = (-1)^{\sat(I)+\dots+\sat(g^{m-1}I)} \operatorname{id}_{D_{I}}.\]
  Set $\wt{\sat}(I)=\sat(I)+\dots+\sat(g^{m-1}I)$.
  By \eqref{eq:gsas}, we have $\partial\wt{\sat}=0$.
  Therefore, \(\wt{\sat}\) is a \(0\)-cocycle, hence, $\wt{\sat}$ is constant.
  As $(0,0,\dots,0)$ is fixed by the group action, the value of $\wt{\sat}$ at $(0,\dots,0)$ must be zero.
  Hence, $\wt{\sat}$ is a zero cochain.
  This shows that $(\cG_{g})^{m} = \cG_{g^m} = \operatorname{id}$, that is, $\cG$ defines a $\Z_m$ action on $\llbra D\rrbra$.
\end{proof}

\begin{remark}\label{remark:sign-ambiguity-gp-action}
  As pointed out in Lemma~\ref{lem:sas}, the choice of the \(0\)-cochain \(\sat\) on \(\SCube(D)\) satisfying \(g \sas - \sas = \partial \sat\) is not unique.
  By Remark~\ref{rem:ambiguity-definition-of-maps}, choosing different \(0\)-cochain \(\sat'\), gives
  \[\llbra \rho_{g} , \sat \rrbra = - \llbra \rho_{g} , \sat' \rrbra.\]
  For odd $m$, setting $\sat(0,\dots,0)=1$ implies that $\wt{\sat}(I)=1$, which means that $(\cG_g)^m$ is
  not the identity, so the choice $\sat(0,\dots,0)=0$ is the only possible.
  On the other hand, for even $m$, there are two choices potentially leading to different group actions.
  Indeed, if \(m\) is even, multiplying \(\mathcal{F}(\cG_{g})\) by a sign, exchanges the \((\pm1)\)-eigenspaces of \(\mathcal{F}(\cG_{g})\).
  Consequently, the condition that fixes the value of \(\sat\) on \((0,0,\ldots,0) \in \Cube(D)\) is necessary to obtain a well-defined group action.
  This issue will reappear in the proof of Lemma~\ref{lem:equivariant_skein_relation}.
\end{remark}

The second step of the construction, is to show that the evaluation functor \(\cF \colon \SFoam_{N} \to \cSym_{N}\) descends to a \(\Z_{m}\)-equivariant evaluation functor
\[\cF_{\Z_{m}} \colon \SFoam_{N}[\Z_{m}] \to \cSym_{N}[\Z_{m}].\]

\begin{proposition}\label{prop:functoriality}
  The evaluation functor $\cF$ from Definition~\ref{def:evaluation_2} extends to a $\Z_{m}$-equivariant evaluation functors
  \begin{align*}
    \cF_{\Z_{m}} &\colon \SFoam_{N}[\Z_{m}] \to \cSym_{N}[\Z_{m}], \\
    \cF_{\Z_{m}} &\colon \Kom{\SFoam_{N}}[\Z_{m}] \to \Kom{\cSym_{N}}[\Z_{m}].
  \end{align*}
\end{proposition}
\begin{proof}
  Observe that if we are given an additive functor \(F \colon \cB \to \cC\) between additive categories \(\cB\) and \(\cC\), there is an induced functor
  \[F \circ (-) \colon \cB[G] \to \cC[G], \quad (X \colon BG \to \cB) \mapsto (F \circ X \colon BG \to \cC).\]
  Taking \(\cB = \SFoam_{N}\) and \(\cC = \cSym_{N}\) and \(F = \cF\) the evaluation functor, we obtain the \(\Z_{m}\)-equivariant evaluation functor
  \[\cF_{\Z_{m}} = \cF \circ (-) \colon \SFoam_{N}[\Z_{m}] \to \cSym_{N}[\Z_{m}].\]
  Similarly, taking \(\cB = \Kom{\SFoam_{N}[\Z_{m}]} = \Kom{\SFoam_{N}}[\Z_{m}]\) and \(\cC = \Kom{\cSym_{N}[\Z_{m}]} = \Kom{\cSym_{N}}[\Z_{m}]\), we obtain the second functor.
\end{proof}

The most difficult part is the invariance of our construction under equivariant isotopies of periodic links.
Recall from~\cite{Politarczyk-Khovanov} that every \(\Z_{m}\)-equivariant isotopy of periodic links can be realized as a sequence of \emph{equivariant Reidemeister moves}, i.e., equivariant analogs of ordinary
Reidemeister moves.
If \(L\) is an \(m\)-periodic link and \(D_{1}\), \(D_{2}\) are two \(m\)-periodic diagrams representing \(L\), we say that \(D_{1}\) and \(D_{2}\) are \emph{\(\Z_{m}\)-equivalent periodic diagrams}, if they are equivariantly
isotopic, i.e., they can be connected by a sequence of equivariant Reidemeister moves.

\begin{remark}\label{rem:inequivalent-periodic-diagrams}
  Observe that given an \(m\)-periodic link and two \(m\)-periodic diagrams of \(L\), the given periodic diagrams may not be equivalent in general.
  Indeed, consider the torus link \(T(2,4)\).
  Taking two different representations of \(T(2,4)\) as a braid closure, we obtain two \(2\)-periodic diagrams of \(T(2,4)\), such that for one of the diagrams the symmetry exchanges components of the link and
  for the other diagram it preserves components.
  Obviously, these \(2\)-periodic diagrams of \(T(2,4)\) cannot be equivalent.
\end{remark}

\begin{theorem}\label{thm:group_action}
  If $D$ and $D'$ are \(\Z_{m}\)-equivalent \(m\)-periodic link diagrams, then there is a chain homotopy equivalence between $\llbra D\rrbra$ and $\llbra D'\rrbra$ in the category $\Kom{\SFoam_N}$
  and an induced quasi-isomorphism between~$\cF(\llbra D\rrbra)$ and~$\cF(\llbra D'\rrbra)$ in~\(\Kom{\cSym_{N}}[\Z_{m}]\).
\end{theorem}
\begin{proof}
  Any two periodic diagrams $D,D'$ of the same periodic link $L$ are related by a sequence of equivariant Reidemeister moves; 
  see \cite[Proposition 2.6]{Politarczyk-Khovanov}.
  Therefore, to prove Theorem~\ref{thm:group_action}, we need to prove the following key lemma.
 \begin{lemma}\label{lem:equiv_reid}
   Suppose $D'$ is a diagram of $L$ obtained from $D$ by a single equivariant Reidemeister move.
   Then, the equivariant Reidemeister move induces a chain homotopy equivalence $\Phi\colon\llbra D\rrbra\to\llbra D'\rrbra$ in \(\Kom{\SFoam_{N}}\) and an induced quasi-isomorphism $\cF(\Phi)\colon \cF (\llbra
   D\rrbra)\to \cF (\llbra D'\rrbra) $
   in the category \(\Kom{\cSym_{N}}[\Z_{m}]\).
 \end{lemma}
 We postpone the proof of Lemma~\ref{lem:equiv_reid} to Subsection~\ref{sub:equiv_reid}. We will now deduce
 Theorem~\ref{thm:group_action} from Lemma~\ref{lem:equiv_reid}.
 Suppose $D$ and $D'$ are two link diagrams of the same periodic link. As they are connected by a sequence of equivariant
 Reidemeister moves, there is a map $\Phi\colon \llbra D\rrbra \to \llbra D'\rrbra$ obtained as a composition
 of individual maps provided by Lemma~\ref{lem:equiv_reid}.
\end{proof}

\subsection{Equivariant $\sln$-homology}
Suppose $L$ is an $m$-periodic link and choose a periodic link diagram $D$ for it. The cochain complex $C^*(D) :=\cF(\llbra D\rrbra)$ is a graded cochain complex of~$\S_N[\Z_{m}]$-modules.
By Theorem~\ref{thm:group_action}, the $\S_N[\Z_{m}]$-quasi-isomorphism class of $C^*(D)$ depends only on the equivalence class of~\(D\).
Denote by $\KR^*_{\S_N[\Z_{m}]}(L)$ the cohomology of $C^*(D)$ as a $\S_N[\Z_{m}]$-module.

\begin{proposition}
  The $\S_N[\Z_{m}]$-module structure on $\KR^*_{\S_N[\Z_{m}]}(L)$ descends via the evaluation functor to a well-defined $\C[\Z_{m}]$-module structure on $\KR_N^*(L)$ and $\Lee_N^*(L)$.
  The Lee--Gornik spectral sequence exists in the category of finitely generated $\C[\Z_{m}]$-modules.
\end{proposition}
\begin{proof}
  Suppose $D$ and $D'$ are two equivalent $m$-periodic link diagrams of $L$.
  We know that $D$ and $D'$ are related by a sequence of equivariant Reidemeister moves, in particular $D$ and $D'$
  are related by a sequence of ordinary Reidemeister moves.
  That is, there exists a chain homotopy equivalence $h\colon C^*(D)\to C^*(D')$ in the category \(\Kom{\Sym_{N}[\Z_{m}]}\).
  By Theorem~\ref{thm:group_action}, the map induced by \(h\) on cohomology of $C^*(D)$ and $C^*(D)$, is a $\S_N[\Z_{m}]$-linear isomorphism.

  Choose $\Sigma \in \Sym^{N}(\C)$ and consider the cochain complexes $C^*_\Sigma(D)=\ev^\Sigma(C^*(D))$ and $C^*_\Sigma(D')=\ev^\Sigma(C^*(D'))$.
  The map $h_\Sigma=\ev^\Sigma(h)$ is a chain homotopy equivalence in the category of complexes of $\C$-modules. 
  In particular $h_\Sigma$ induces an isomorphism on cohomology of $C^*_\Sigma(D)$ and $C^*_\Sigma(D')$.

  Since the evaluation $\ev^\Sigma$ commutes with the $\Z_{m}$-action, the map $h_\Sigma$ is $\Z_{m}$-equivariant.
  Consequently, the map induced by \(h_{\Sigma}\) on cohomology is also \(\Z_{m}\)-equivariant.
Since \(\C[\Z_{m}]\) is a semi-simple algebra, we can use Schur's Lemma to conclude that every cochain complex in $\C[\Z_m]$ splits.
In particular, it follows that \(h_{\Sigma}\) is a chain homotopy equivalence in $\C[\Z_m]$.

  In particular, choosing \(\Sigma\) to be singular or generic, we recover that $\KR^*_N(L)$ and $\Lee_N^*(L)$ have a well-defined structure of $\C[\Z_{m}]$-modules.

  We now want to prove that the spectral sequence from $\KR^*_N(L)$ to $\Lee_N^*(L)$ exists in the category of $\C[\Z_m]$-modules.
  Observe that the evaluation map commutes with the $\Z_m$-action.
  Following the proof of Proposition~\ref{prop:abstract_lee_gornik},
  we decompose $d^i_{gen}=d^i_0+d^i_1+\dots$. The decomposition is $\Z_{m}$-equivariant, because the terms $d^i_j$ come from
  the specialization of degree $j$ terms in the differential in $\cF(\llbra D\rrbra)$, and the $\Z_{m}$-action preserves the quantum degree.
  Consequently, the spectral sequence can be defined over the ring $\C[\Z_{m}]$.
\end{proof}

We now pass to the definition of \(\Z_{m}\)-equivariant Khovanov--Rozansky homology.

\begin{definition}
  Let $L$ be an $m$-periodic link. The \emph{\(\Z_{m}\)-equivariant Khovanov--Rozansky $\sln$-homology} $\EKR^{k,r}_N$
  is the group $\KR^{k,r}_N$ with the $\C[\Z_{m}]$-module structure coming from the action of \(\Z_{m}\) on $\KR^*_N(L)$.
  Likewise, the \emph{\(\Z_{m}\)-equivariant Lee $\sln$-homology} $\ELee^{k}_N$ is the group $\Lee^k$ with the $\C[\Z_{m}]$-structure
  coming from the action of $\Z_{m}$ on $\Lee^*_N(L)$.
\end{definition}

Next result relies on the convention we adopted that the mirror link diagram is obtained by switching all the crossings (`the mirror is parallel to the plane'). The other convention (`the mirror is perpendicular to the plane') leads to a similar result, but the action of $\Z_m$ is
inverted.

\begin{proposition}\label{prop:e_mirror}
  Suppose $L$ is an $m$-periodic link with periodic diagram $D$. Let $L'$ be the mirror. Then, for any $k,r$ there is a map
  of $\C[\Z_m]$-modules
  \[\EKR^{k,r}_N(L)\cong \EKR^{-k,-r}_N(L').\]
\end{proposition}
\begin{proof}
  The isomorphism at the level of vector spaces over $\C$ was established in Proposition~\ref{prop:mirror}. We need to argue that this isomorphism commutes with the group action.
  We use the notation of the proof of Proposition~\ref{prop:mirror}.
  The group $\Z_{m}$ acts on $\Cube(D)$ and on $\Cube(D')$ in the same manner, that is, $(gI)'=gI'$.
  Moreover, mirroring a resolution commutes with rotating, that is, $gD_I=D_{gI}$ and $gD'_{I'}=D'_{gI'}$.
  With this identification, it is easy to see that the map $\iota\colon\cF_0(\llbra D\rrbra)\to\cF_0(\llbra D'\rrbra)$ given by $\iota(\cF_0(D_I))=\cF_0(D'_{I'})$ commutes with the group action.

  Choose a basis of $\cF_0(\llbra D\rrbra)$ as $\C[\Z_{m}]$-module.
  Then, $\cF_0(\llbra D'\rrbra)$ has the same basis.
  Using this basis, we can write a $\C[\Z_m]$-isomorphism of modules.
  \begin{equation}\label{eq:oclf}
    \cF_0(\llbra D'\rrbra)\xrightarrow{\cong}\Hom_{\C[\Z_{m}]}(\cF_0(\llbra D'\rrbra),\C[\Z_{m}]).
  \end{equation}
  This isomorphism maps a basis to the dual basis.

  Note that with a choice of a basis, the differential on $\cF_0(\llbra D'\rrbra)$ is the transpose of the differential
  on $\cF_0(\llbra D\rrbra)$. Therefore, the identity \eqref{eq:oclf} is in fact an isomorphism of cochain complexes.
  The group ring $\C[\Z_{m}]$ is semisimple, so $\Hom_{\C[\Z_{m}]}(\cdot,\C[\Z_{m}])$ is exact. It follows that 
  $H^*(\cF_0(\llbra D'\rrbra))=\Hom(H^*(\cF_0(\llbra D\rrbra),\C[\Z_{m}])$. The latter group is (non-canonically)
  isomorphic to $H^*(\cF_0(\llbra D\rrbra))$.
\end{proof}


\subsection{Decomposition of $\sln$-homology}
\label{sec:decomp-sln-homol}

This subsection is based on \cite{Politarczyk-Jones,Politarczyk-Khovanov}, 
where an analogous decomposition is constructed for Khovanov homology (that is, for $N=2$). The present case does not pose significantly more
problems. We present a decomposition of $\EKR_N^{k,r}$ and $\ELee_N^k$ into simple $\C[\Z_{m}]$-modules.
This leads to a third grading of $\EKR_N^{k,r}$ and the second grading of $\ELee_N^k$.

Fix a generator \(g \in \Z_{m}\).
The group algebra $\C[\Z_{m}]$ is semisimple, and we have a decomposition
\[\C[\Z_{m}] = \bigoplus_{j=0}^{m-1} \C_{\xi_{m}^{j}},\]
where \(\C_{\xi_{m}^{j}}\) denotes the \(\xi_{m}^{j}\)-eigenspace of \(\C[\Z_{m}]\), for \(\xi_{m} = \exp \left( \frac{2 \pi \sqrt{-1}}{m} \right)\) and \(j=0,1,\ldots,m-1\).
The above decomposition is obtained by writing
\begin{equation}
  \label{eq:idempotent-sum}
  1 = \sum_{j=0}^{m-1}e_{j},
\end{equation}
where \(e_{0},e_{1}\ldots,e_{m-1}\) are pairwise orthogonal idempotents, i.e., \(e_{j} e_{k} = \delta_{kl} e_{j}\), for \(0 \leq j,k \leq m -1\).
Furthermore, for any \(j=0,1,\ldots,m-1\), we have \(g \cdot e_{j} = \xi_{m}^{j} e_{j}\).

In particular, we can present any $\C[\Z_{m}]$-module $M$ as a direct sum over different representations of $\C[\Z_{m}]$.
More precisely,
\begin{equation}
  \label{eq:eigenspace-decomposition-G-modules}
  M = \bigoplus_{i=0}^{m-1} M_{\xi_{m}^{i}}, \quad \text{where } M_{\xi_{m}^{i}} := e_{i} M
\end{equation}
is the \(\xi_{m}^{i}\)-eigenspace of \(M\), for \(i=0,1,\ldots,m-1\).
Since the idempotents \(e_{0},e_{1},\ldots,e_{m-1}\) are pairwise orthogonal, it follows easily that for any finitely generated \(\C[\Z_{m}]\)-modules \(M\) and \(N\), we have
\begin{equation}
  \label{eq:schurs-lemma}
  \Hom_{\C[\Z_{m}]}(M_{\xi_{m}^{j}},N_{\xi_{m}^{k}}) = 0,
\end{equation}
unless \(j=k\).

\begin{remark}
  Observe that since we work over complex symmetric polynomials, the formula~\eqref{eq:idempotent-sum} holds also in \(\S_{N}[\Z_{m}]\).
Consequently, the group ring \(\S_{N}[\Z_{m}]\) admits a decomposition into eigenspaces of the action of the chosen generator~\(g\):
\[\S_{N}[\Z_{m}] = \bigoplus_{j=0}^{m-1} \S_{N,\xi_{m}^{j}}, \quad S_{N,\xi_{m}^{j}} := e_{j} \S_{N}[\Z_{m}],\]
where \(\S_{N,\xi_{m}^{j}}\) are ideals consisting of \(\xi_{m}^{j}\)-eigenvectors of the action of \(g\), for \(j=0,1,\ldots,m-1\).
Consequently, any finitely generated \(\S_{N}[\Z_{m}]\)-module \(M\) admits a decomposition into eigenspaces of \(g\)
\begin{equation}\tag{\ref{eq:eigenspace-decomposition-G-modules}'}M = \bigoplus_{j=0}^{m-1} M_{\xi_{m}^{j}}, \quad M_{\xi_{m}^{j}} := e_{j} M.\end{equation}
Furthermore, in analogy to~\eqref{eq:schurs-lemma}, we deduce that if \(M,N\) are finitely generated \(\S_{N}[\Z_{m}]\)-modules, then
\begin{equation}\tag{\ref{eq:schurs-lemma}'}\Hom_{\S_{N}[\Z_{m}]}(M_{\xi_{m}^{j}}, N_{\xi_{m}^{k}}) = 0,\end{equation}
unless \(j=k\).
  Consequently, if \(D\) is an \(m\)-periodic link diagram, the \(\S_{N}\)-equivariant Khovanov-Rozansky homology of \(D\) admits a decomposition into eigenspaces of~\(g\)
  \[H^{k,r}\left( \cF(\llbra D \rrbra) \right) = \bigoplus_{j=0}^{m-1} H^{k,r}_{\xi_{m}^{j}}(\cF \left( \llbra D \rrbra \right)).\] In fact, there is a decomposition on the level of the cochain complex,
  \[\cF \left( \llbra D \rrbra ) \right) = \bigoplus_{j=0}^{m-1} \cF_{\xi_{m}^{j}} \left( \llbra D \rrbra \right)\]
  and, for \(j=0,1,\ldots,m-1\), \(H^{\ast,\ast}\left( \cF_{\xi_{m}^{j}} \left( \llbra D \rrbra \right) \right) = H^{k,r}_{\xi_{m}^{j}}(\cF \left( \llbra D \rrbra \right))\).
\end{remark}
From decomposition~(\ref{eq:eigenspace-decomposition-G-modules}'), we create the following splitting
\[M = \bigoplus_{d \mid m} M_{d},\]
where, for any \(d\) dividing \(m\), we define
\[M_{d} = \bigoplus_{\stackrel{0 \leq i < m}{\gcd(i,m)=m/d}} M_{\xi_{m}^{i}} = \bigoplus_{\stackrel{0 \leq i < d}{\gcd(i,d)=1}} M_{\xi_{d}^{i}}.\]
Accordingly, we write
\begin{equation}\label{eq:third_grading}
  \EKR_N^{k,r}(L)=\bigoplus_{d|m}\EKR_N^{k,r,d}(L),
\end{equation}
where, for every \(d\) dividing \(m\), we define
\begin{equation}
  \label{eq:def-equivariant-homology}
  \EKR_{N}^{\ast,\ast,d}(L) := \Hom_{\C[\Z_{m}]}(\C[\Z_{m}]_{d},\EKR_{N}^{\ast,\ast}(L)) \cong H^{\ast,\ast}\left( \Hom_{\C[\Z_{m}]} (\C[\Z_{m}]_{d},C_{0}^{\ast,\ast}(L)) \right).
\end{equation}
Observe that in~\cite{Politarczyk-Khovanov}, the decomposition of the equivariant Khovanov homology functors were defined as certain Ext functors.
However, \(\C[\Z_{m}]\) is a semisimple algebra, hence \(\Ext^{i}_{\C[\Z_{m}]}(M,N) = 0\), for \(i>0\) and any \(\C[\Z_{m}]\)-modules \(M, N\).
It follows that the \(\Hom\) functor is exact with respect to both variables.

We now perform an analogous operation on Lee homology. The group $\Lee^k_N(L)$ depends only on the linking numbers
of components of $L$; compare Proposition~\ref{prop:computingLee}. 
The module $\ELee^k_N$ depends on how the symmetry group acts on the components of $L$.
We have the decomposition.
\[\ELee_N^{k}(L)=\bigoplus_{d|m}\ELee_N^{k,d}(L).\]
To study $\ELee_N^{k,d}(L)$, recall from Proposition~\ref{prop:computingLee} that $\Lee_N(L)$ was generated by classes $\sal_\psi$
for any coloring $\psi$ of components of $L$ by elements from $\{1,\dots,N\}$. 
The $\Z_{m}$-action on $S^3$ preserving $L$, acts on the components of $L$. 
In particular, there is an action of $\Z_{m}$ on the set of all colorings of components of $L$.
We denote this action $(g,\psi)\mapsto g\psi$, $g \in \Z_{m}$.
An \emph{order} of a coloring $\psi$, $\sao(\psi)$, is the minimal
number $k\ge 1$, such that $g^k\psi=\psi$ for all $g\in \Z_{m}$. It is clear from the definition of the decomposition
of $\ELee_N$, that
$\sal_\psi\in\ELee_N^{d,\sao(\psi)}$, where $d=\deg(\sal_\psi)$. In particular, if $\Z_{m}$ fixes components of $L$,
we have $\sao(\psi)=1$ for all $\psi$. This leads to the following observation, which we record for future use.
\begin{lemma}\label{lem:trivialLee}
  Suppose the group $\Z_{m}$ acts trivially on the components of $L$, that is, rotation maps components to themselves. Then
  $\Lee_N^{k,d}(L)$ is trivial unless $d=1$.
\end{lemma}

\subsection{Proof of Lemma~\ref{lem:equiv_reid}}\label{sub:equiv_reid}
For the reader's convenience we recall the statement of Lemma~\ref{lem:equiv_reid}.
 \begin{myLemma}
   Suppose $D'$ is a diagram of $L$ obtained from $D$ by a single equivariant Reidemeister move.
   Then, the equivariant Reidemeister move induces a chain homotopy equivalence $\Phi\colon\llbra D\rrbra\to\llbra D'\rrbra$ in \(\Kom{\SFoam_{N}}\) and an induced quasi-isomorphism $\cF(\Phi)\colon \cF (\llbra
   D\rrbra)\to \cF (\llbra D'\rrbra) $
   in the category \(\Kom{\cSym_{N}}[\Z_{m}]\).
 \end{myLemma}
 Throughout the proof, we let $g$ be a fixed generator of $\Z_{m}$ acting on $\R^2$ (and hence, on resolutions) by the rotation by $2\pi /m$.
   Without loss of generality we may and will assume that $D'$ has no fewer crossings than $D$.
   The map $\Phi$ is given by a family of foams $\Phi_{I,J}$ for $(I,J)\in \cA\subset\Cube(D)\times \Cube(D')$ and signs
   $\sad(I,J)$, so that the component of $\Phi$ from $D_I$ to $D'_J$ is $(-1)^{\sad(I,J)}\Phi_{I,J}$. 
   The foams $\Phi_{I,J}$ are constructed as concatenations of foams corresponding to individual (non-equivariant) Reidemeister
   moves, whose composition
   is the equivariant Reidemeister move we are considering.

   Two things are to be proved.
   \begin{itemize}
     \item \emph{(geometry).} The group $\Z_{m}$ acts on $\Cube(D)$ and on $\Cube(D')$ via permutations of crossings. We need to show
       that we can construct foams $\Phi_{I,J}$ in such a way $\Phi_{gI,gJ}$ is isotopic to $g\Phi_{I,J}$ as a foam from
       $gD_I$ to $gD'_J$;
     \item \emph{(algebra).} We have to show that a sign assignment on $D$ induces a compatible sign assignment $\sas'$
       on $D'$ such that the maps \(\Phi_{I,J}\), for \((I,J) \in \cA\), together with appropriate signs combine into a
       \(\Z_{m}\)-equivariant chain map.
       To be more precise, we will show that there is a sign assignment \(\sas'\) on \(\SCube(D')\) such that the following diagram is commutative
       \begin{equation}\label{eq:algebra_part}\begin{tikzcd}[column sep=2cm]
	 D_{I} \ar[r, "{(-1)^{\sat(I)} \rho_{g,I}}"] \ar[d, "(-1)^{\sad(I,J)}\Phi_{I,J}"'] & D_{gI} \ar[d, "(-1)^{\sad(gI,gJ)}\Phi_{gI,gJ}"] \\
         D'_{J} \ar[r, "{(-1)^{\sat'(J)} \rho'_{g,J}}"] & D'_{gJ},
       \end{tikzcd}\end{equation}
       where $\sat$ is the $0$-cochain on $\SCube(D)$ defined by the condition $\partial \sat=\sas-g\sas$, $\sat(0,\dots,0)=0$ and $\sat'$
       is defined as an analogous $0$-cochain on $\SCube(D')$.
   \end{itemize}
   To emphasize, we included the signs $\sad(I,J)$ in~\eqref{eq:algebra_part}. However, we will show that $\sad(I,J)=0$
   is a compatible choice, exactly as in the non-equivariant setting.
    The plan of the proof is as follows.
   \begin{itemize}
     \item We give full details of the case of a positive Reidemeister~1 move and $\Z_{m}=\Z_2$;
     \item We show the algebra part of a positive Reidemeister~1 move and $m>2;$
     \item We discuss the algebra part of a Reidemeiser $2a$ move and $\Z_{m}=\Z_2$;
   \end{itemize}
   The geometry part is similar in all moves, so the first item of the list can be easily generalized to all other cases. 
   As for the algebra part, the cases of Reidemeister~1 and Reidemeister~2 moves differ. 
   The case of a Reidemeister~2b move, or a negative Reidemeister~1 move are direct generalizations
   of the cases we discuss. The case of Reidemeister 3 move is the easiest,
   because the number of crossings does not change. The sign assignment on $D$ gives a direct sign assignment of $D'$. Therefore, even
   if the invariance of $\llbra D\rrbra$ on Reidemeister~3 move has the most complex proof, see \cite{MSV}, it is the easiest to deal with in the equivariant case.

   \bigskip
   \textbf{Step 1. Positive Reidemeister~1 move and $\Z_{m}=\Z_2$.}
   We will use the following notation. The diagram $D'$ is the diagram $D$ with two Reidemeister 1 moves applied to it.
   We denote it by $D'=D\langle\smallloop\smallloop\!\rangle$.
   We write $D=D\langle\smalldown\,\smalldown\rangle$ and the diagram $D$ with one of the two Reidemeister moves applied is denoted by $D\langle\smalldown\,\smallloop\rangle$, respectively $D\langle\smallloop\,\smalldown\rangle$.


   The cochain complexes for $D\langle\smalldown\,\smallloop\rangle$ and $D\langle\smallloop\,\smalldown\rangle$, respectively, are
   \begin{align*}
     \llbra D\langle\smalldown\,\smallloop\rangle\rrbra &= \big\{0 \to \llbra D\langle\smallcircle\,\smalldown\rangle\rrbra\xrightarrow{d}
                                              \llbra D\langle\otherdown\,\smalldown\rangle\rrbra\to 0\big\},\\
     \llbra D\langle\smallloop\,\smalldown\rangle\rrbra &= \big\{0 \to \llbra D\langle\smalldown\,\smallcircle\rangle\rrbra\xrightarrow{d}
                                              \llbra D\langle\smalldown\,\otherdown\rangle\rrbra\to 0\big\}.
   \end{align*}
   The differentials in the above complexes are given by the family of foams that are products except at the relevant crossings: near the crossing where the Reidemeister move occurs, the foams are as in
   Figure~\ref{fig:resolution}.
   Let us identify
   \begin{equation}
     \label{eq:identification-of-cubes-equiiv-reidmeister-1}
     \Cube(D \langle \smallloop \smallloop \rangle) \cong \Cube(D \langle \smalldown \smalldown \rangle) \times \{0,1\}^{2}
   \end{equation}
   in such a way that
   \begin{align*}
     \Cube(D \langle \smalldown \smallloop \rangle) &= \bigcup_{x=0,1} \Cube(D \langle \smalldown \smalldown \rangle) \times \{(0,x)\} \subset \Cube(D \langle \smallloop \smallloop \rangle), \\
     \Cube(D \langle \smallloop \smalldown \rangle) &= \bigcup_{x=0,1} \Cube(D \langle \smalldown \smalldown \rangle) \times \{(x,0)\} \subset \Cube(D \langle \smallloop \smallloop \rangle).
   \end{align*}

   The cochain complex $\llbra D\langle\smallloop\smallloop\rangle \rrbra$ can be written as a bicomplex of partial resolutions.
   \[
     \begin{tikzcd}
       && \llbra D\langle\otherdown\smallcircle\rangle\rrbra\ar[rd] & \\
       0\ar[r] & \llbra D\langle\smallcircle\smallcircle\rangle\rrbra\ar[ru]\ar[rd] && 
       \llbra D\langle\otherdown\otherdown\rangle\rrbra\ar[r] & 0 \\
       && \llbra D\langle\smallcircle\otherdown\rangle\rrbra\ar[ru] & 
     \end{tikzcd}
   \]
   The chain homotopy equivalence between the above bicomplex and $\llbra D\langle\smalldown\smalldown\rangle \rrbra$
   is given by explicit maps as in the proof of Theorem~\ref{thm:invariance_reid}. More precisely,
   consider the maps
   \begin{align*}
     \phi^1&\colon\llbra D\langle\smalldown\smalldown\rangle \rrbra\to \llbra D\langle\smalldown\smallloop\rangle \rrbra
	   &\phi^2&\colon\llbra D\langle\smalldown\smallloop\rangle \rrbra\to \llbra D\langle\smallloop\smallloop\rangle \rrbra\\ 
     \phi^3&\colon\llbra D\langle\smalldown\smalldown\rangle \rrbra\to \llbra D\langle\smallloop\smalldown\rangle \rrbra 
	   &\phi^4&\colon\llbra D\langle\smallloop\smalldown\rangle \rrbra\to \llbra D\langle\smallloop\smallloop\rangle \rrbra 
   \end{align*}
   given by foams realizing the corresponding (non-equivariant) Reidemeister~1 move. 
   These maps lead to the following diagram in~\(\KFoam\)
\begin{equation}\label{eq:if_theres_a_diagram_it_commutes}
  \begin{tikzcd}
  & \llbra D\langle\smalldown\smalldown\rangle \rrbra \ar[ld,"\phi^1"] \ar[rd,"\phi^3"] \arrow[out=120,in=60,loop,<->,squiggly,blue,"{\cG_{g}}"]\\
    \llbra D\langle\smalldown\smallloop\rangle \rrbra \ar[squiggly,blue,<->,rr,"{\cG_{g}}"] \ar[rd,"\phi^2"]& &
    \llbra D\langle\smallloop\smalldown\rangle \rrbra \ar[ld,"\phi^4"]\\
  & \llbra D\langle\smallloop\smallloop\rangle \rrbra \arrow[out=120,in=60,loop below,<->,distance=3em,squiggly,blue,"{\cG_{g}}"],
  \end{tikzcd}
\end{equation}
   where 
   the squiggly arrows denote the action of the generator $g \in \Z_2$.

   Let us deal with the geometric part first.
   
   \begin{lemma}\label{lem:if_theres_a_diagram_it_commutes}
     For any \(I \in \Cube(D)\), the following diagram is commutative in \(\SFoam_{N}\).
     \begin{equation}\label{eq:geometric_commutation}
       \begin{tikzcd}
         D\langle\smalldown\smalldown\rangle_{I} \ar[d, "{\phi^{1}_{I}}"] \ar[r,squiggly,blue,"{\rho_{g,I}}"] & D\langle\smalldown\smalldown\rangle_{gI} \ar[d, "{\phi^{3}_{I}}"] \\
         D\langle\smalldown\smallloop\rangle_{(I,0)} \ar[d, "{\phi^{2}_{(I,0)}}"] \ar[r,squiggly,blue,"{\rho_{g,(I,0)}}"] & D\langle\smallloop\smalldown\rangle_{(gI,0)} \ar[d, "{\phi^{4}_{(I,0)}}"] \\
         D\langle\smallloop\smallloop\rangle_{(I,0,0)} \ar[r,squiggly,blue,"{\rho_{g,(I,0,0)}}"] & D\langle\smallloop\smallloop\rangle_{(gI,0,0)}.
       \end{tikzcd}
     \end{equation}
     In the above diagram, by \(\rho_{g,I}\) and \(\rho_{g,(I,0,0)}\) we denote, as in the proof of Proposition~\ref{prop:group_action_on_D}, the foams realizing the rotation by \(g\) of the appropriate
     resolution diagrams.
     We are also using the identification~\eqref{eq:identification-of-cubes-equiiv-reidmeister-1} of cubes of resolutions
     of the relevant diagrams.
   \end{lemma}
   \begin{proof}[Proof of Lemma~\ref{lem:if_theres_a_diagram_it_commutes}]
     The statement follows from the construction of the foams $\rho_g$. Recall that these are traces of the rotation by $2\pi /m$.
     That is to say, for given $I\in \Cube(D)$, the composition $\rho_{g,(I,0,0)}\circ \phi^1_I$ takes the foam $\phi^1_I$ (which is the
     product foam except near the crossing), and then it twists its upper part by $2\pi /m$. On the other hand,
     the composition $\phi^3_I\circ\rho_g$ take the foam $\phi^1_I$ and twists its lower part. It is clear that these two
     foams are isotopic rel boundary. This shows commutativity of the top square. The commutativity of the bottom square is proved
     in a similar way.
   \end{proof}

To complete the proof of Lemma~\ref{lem:equiv_reid}, we need to deal with the algebraic part.
Fix a sign assignment $\sas$ on $D$ and let $g$ be the generator of $\Z_2$.
By Lemma~\ref{lem:coboundary}, $g\sas=\sas+\partial\sat$
for some \(0\)-cochain $\sat$ on \(\SCube(D)\).
As explained in~Remark~\ref{remark:sign-ambiguity-gp-action}, the requirement $\sat((0,\dots,0))=0$, which we henceforth impose, determines $\sat$ uniquely.
Let $\sas_1$ be the sign assignment on $D\langle\smallloop\smalldown\rangle$ induced from $D$ via Lemma~\ref{lem:sas}.
Write $\sas_2=g\sas_1$ for the sign assignment on $D\langle\smalldown\smallloop\rangle=gD\langle\smallloop\smalldown\rangle$.

On $D\langle\smallloop\smallloop\rangle$, we can construct a sign assignment by starting
from $\sas_1$ on $D\langle\smallloop\smalldown\rangle$ and extending it by Lemma~\ref{lem:sas}. Call this sign assignment
$\sas_3$. There is another option, namely we start with $\sas_2=g\sas_1$ on
$D\langle\smalldown\smallloop\rangle=gD\langle\smallloop\smalldown\rangle$ and use Lemma~\ref{lem:sas} in this case, to
obtain a sign assignment on $D\langle\smallloop\smallloop\rangle$. Call the latter sign assignment $\sas_4$.
We know that there exists a \(0\)-cochain $\sat'$ on $\SCube(D\langle\smallloop\smallloop\rangle)$ such that $\sas_3(I,I')-\sas_4(I,I')=\sat'(I)-\sat'(I')$, for any \(I,I' \in \Cube(D\langle\smallloop\smallloop\rangle)\) such that \(I'\) is an immediate successor of \(I\).
\begin{lemma}\label{lem:sat_on_equiv}\
  \begin{itemize}
    \item[(a)] We have $g\sas_3=\sas_4$.
    \item[(b)] if $\sat$ is a \(0\)-cochain on $\SCube(D)$ such that $\sas-g\sas=\partial\sat$, then the \(0\)-cochain \(\sat'\) on \(\SCube(D\langle\smallloop\smallloop\rangle)\) defined by
      \[\sat'((I,x,y))=xy+\sat(I)\in\F_2,\]
      satisfies \(\sas_{3} - \sas_{4} = \partial \sat'\).
  \end{itemize}
\end{lemma}
\begin{proof}[Proof of Lemma~\ref{lem:sat_on_equiv}]
  Let $I'_1,I'_2\in\Cube(D\langle\smallloop\smallloop\rangle)$ be such that $I_2'$ is an immediate successor of $I_1'$. Write $I_k'=(I_k,x_k,y_k)$,
  $k=1,2$ with $I_k\in\Cube(D)$, $x_k,y_k\in\{0,1\}$. We have 
  \begin{equation}\label{eq:gik}
    gI_k'=(gI_k,y_k,x_k),
  \end{equation}
  because the group action switches
  the last two crossings.
  By construction of $\sas_3$ via Lemma~\ref{lem:sas}, we have:
  \[
    \sas_3((I_1,x_1,y_1),(I_2,x_2,y_2))=\begin{cases} 0 & \textrm{if } y_2=y_1+1\\
      \sas_1((I_1,x_1),(I_2,x_2))+y_1 & \textrm{if } y_1=y_2.
    \end{cases}
  \]
    Analogously
  \[
    \sas_4((I_1,x_1,y_1),(I_2,x_2,y_2))=\begin{cases} 0 & \textrm{if } x_2=x_1+1\\
      \sas_2((I_1,y_1),(I_2,y_2))+x_1 & \textrm{if } x_1=x_2.
    \end{cases}
  \]
  Repeating this procedure, we obtain
  \begin{equation}\label{eq:sas_3}
    \sas_3((I_1,x_1,y_1),(I_2,x_2,y_2))=
    \begin{cases}
      0 & y_2=y_1+1\\
      y_1 & y_1=y_2\textrm{ and } x_2=x_1+1\\
      \sas(I_1,I_2) + x_1 + y_1 & y_1=y_2 \textrm{ and } x_1=x_2. 
  \end{cases}\end{equation}
  and
  \begin{equation}\label{eq:sas_4}
    \sas_4((I_1,x_1,y_1),(I_2,x_2,y_2))=
    \begin{cases}
      0 & x_2=x_1+1\\
      x_1 & x_1=x_2\textrm{ and } y_2=y_1+1\\
      g\sas(I_1,I_2) + x_1 + y_1 & y_1=y_2 \textrm{ and } x_1=x_2. 
    \end{cases}
  \end{equation}
  Equation~\eqref{eq:gik} together with \eqref{eq:sas_3} and~\eqref{eq:sas_4}
  implies that $g\sas_3=\sas_4$, proving (a). Next, observe that
  \begin{equation}\label{eq:sas_diff}
    \begin{split}
    \sas_3((I_1,x_1,y_1),(I_2,x_2,y_2))-\sas_4((I_1,x_1,y_1),(I_2,x_2,y_2))=\\
    =\begin{cases}
      x_1& x_1=x_2 \textrm{ and } y_2=y_1+1\\
      y_1& y_1=y_2 \textrm{ and } x_2=x_1+1\\
      \sas(I_1,I_2)-g\sas(I_1,I_2) & x_1=x_2,\ y_1=y_2.
    \end{cases}
  \end{split}
  \end{equation}
  Notice that $I_2'$ is an immediate successor of $I_1'$, hence we do not consider the case when simultaneously $y_2=y_1+1$ and $x_2=x_1+1$ in
  \eqref{eq:sas_diff}.

  On the other hand, with the definition of $\sat'(I,x,y)=\sat(I)+xy$ we have:
  \begin{equation}\label{eq:sat_diff}
    \sat'(I_1,x_1,y_1)-\sat'(I_2,x_2,y_2)=x_1y_1+x_2y_2+\sat'(I_1)-\sat'(I_2).
  \end{equation}
  There are two cases.
  \begin{itemize}
    \item If $I_1=I_2$, then \eqref{eq:sas_diff} is $1$ if and only if $(x_1,y_1)=(1,1)$ or $(x_2,y_2)=(1,1)$. 
      These two inequalities cannot simultaneously hold, for otherwise $(I_1,x_1,y_1)=(I_2,x_2,y_2)$. But then, the right-hand side of \eqref{eq:sat_diff} is $1$.
      If none of the pair $(x_1,y_1)$ or $(x_2,y_2)$ is $(1,1)$, the right-hand sides of both \eqref{eq:sas_diff} and~\eqref{eq:sat_diff}
      are $1$.
    \item If $I_1\neq I_2$, then necessarily $x_1=x_2$ and $y_1=y_2$, hence $\sat'(I_1,x_1,y_1)-\sat'(I_2,x_2,y_2)=\sat(I_1)-\sat(I_2)$.
      On the other hand,
      using \eqref{eq:sas_3} and~\eqref{eq:sas_4}, we obtain that 
      \[\sas_3((I_1,x_1,y_1),(I_2,x_2,y_2))-\sas_4((I_1,x_1,y_1),(I_2,x_2,y_2))=\sas(I_1,I_2)-g\sas(I_1,I_2).\]
      The latter is equal to $\sat(I_1)-\sat(I_2)$ by the definition of $\sat$.
  \end{itemize}
\end{proof}
Continuing the proof of the algebra part in Step~1, we consider the commutative diagrams \eqref{eq:geometric_commutation}
for all $I\in\Cr(D)$. As $g\sas_3=\sas_4$, we obtain the following diagram in $\Kom{\SFoam_N}$.
\begin{equation}\label{eq:algebraic_commutation}
  \begin{tikzcd}[column sep=2cm]
    \llbra D\langle\smalldown\smalldown\rangle, \sas \rrbra \ar[d,"\phi^1"]\ar[r,blue,squiggly,"{\llbra \rho_{g}, \sat \rrbra}" above] 
    & \llbra D\langle\smalldown\smalldown\rangle, \sas \rrbra \ar[d, "\phi^{3}"] \\
    \llbra D\langle\smalldown\smallloop\rangle, \sas_{1} \rrbra \ar[d, "\phi^{2}"] & 
    \llbra D\langle\smallloop\smalldown\rangle, \sas_{2} \rrbra \ar[d, "\phi^{4}"]  \\
    \llbra D\langle\smallloop\smallloop\rangle, \sas_{3} \rrbra \ar[r,blue,squiggly,"{\llbra \rho_{g}, \sat' \rrbra}" above]
    & \llbra D\langle\smallloop\smallloop\rangle, \sas_{4} \rrbra.
  \end{tikzcd}
\end{equation}
We claim that the diagram is commutative. First, commutativity up to sign is a direct consequence of Lemma~\ref{lem:if_theres_a_diagram_it_commutes}. In fact, for any individual $I\in\Cr(D)$, the diagram \eqref{eq:algebraic_commutation} becomes the diagram \eqref{eq:geometric_commutation}
up to sign. To check the sign, note that the sign accumulated by going from $D_I$ first through $\phi^2\circ\phi^1$ and then
through $(\rho_g,\sat')$ is $(-1)^{\sat'((I,0,0))}$. Going first through $\llbra \rho_g,\sat \rrbra$ and then by $\phi^4\circ\phi^3$
gives the sign of $(-1)^{\sat(I)}$. Now, by construction, $\sat'((I,0,0))=\sat(I)$. That is, \eqref{eq:algebraic_commutation} commutes.

Commutativity of \eqref{eq:algebraic_commutation} together with the following result, will conclude the proof of Step~1.
\begin{lemma}\label{lem:equal_as_maps}
  The compositions $\phi^4\circ\phi^3$ and $\phi^2\circ\phi^1$ are equal as maps in $\Kom{\Foam_N}$. 
\end{lemma}
\begin{proof}
  Choose $I\in\Cr(D)$. The map $\phi^4_I\circ\phi^3_I$ is given by composing foams realizing the Reidemeister move for the first crossing
  and then for the second crossing, i.e. $\smalldown\smalldown\to\smallloop\smalldown\to\smallloop\smallloop$. The map
  $\phi^2_I\circ\phi^1_I$ is given by composing foams realizing the Reidemeister move for the second crossing and then for the first crossing,
  i.e. $\smalldown\smalldown\to\smalldown\smallloop\to\smalldown\smalldown$. All the foams $\phi^1_I,\dots,\phi^4_I$ are products
  except at relevant crossings. Hence, $\phi^4_I\circ\phi^3_I$ and $\phi^2_I\circ\phi^1_I$ are isotopic rel boundary as foams.
  By Lemma~\ref{lem:invariant}, these two compositions give the same element as a morphism in $\Foam_N$.
\end{proof}
Let $\Phi$ denote the composition $\phi^4\circ\phi^3=\phi^2\circ\phi^1$.
Then $\Phi$ is induced by
a composition of individual, non-equivariant Reidemeister moves.
In particular, $\Phi$ is a chain homotopy equivalence.
The horizontal maps in \eqref{eq:algebraic_commutation} are group actions on $\llbra D\rrbra$ and $\llbra D\langle\smalldown\smalldown\rangle\rrbra$.
Commutativity of \eqref{eq:algebraic_commutation}
means that $\Phi$ commutes with the group action. This proves the first part of Lemma~\ref{lem:equiv_reid} for Step~1.

To prove the second part, we apply the evaluation functor $\cF$ from the category $\Kom{\SFoam_N}$ to the abelian
category $\Kom{\cSym_N}$. The map $\cF(\Phi)\colon \cF(\llbra D\rrbra)\to \cF(\llbra D'\rrbra)$ is a chain homotopy equivalence.
In particular, it is a quasi-isomorphism in $\Kom{\cSym_N}$.

By Proposition~\ref{prop:functoriality}, $\cF(\llbra D\rrbra)$ and $\cF(\llbra D'\rrbra)$ 
carry an action of $\Z_{m}$ and, by virtue of Lemma~\ref{lem:equal_as_maps}, $\cF(\Phi)$ commutes with the group action. A $\Z_{m}$-equivariant
quasi-isomorphism is a quasi-isomorphism in the $\Kom{\cSym_N[\Z_{m}]}$. The proof of Step~1 is finished.

\medskip
\textbf{Step 2. Positive Reidemeister one move, $\Z_m$ action for general $m$.} 
This step is notationally more cumbersome than the previous one, but the idea is the same.
Let $D$ be a periodic link diagram and $D'$ be the link obtained by applying an equivariant Reidemeister one move.
Once more, identify \(\Cube(D') \cong \Cube(D) \times \{0,1\}^{m}\) in such a way that for any \(I \in \Cube(D)\) and any \(x_{1},x_{2},\ldots,x_{m} \in \{0,1\}\), we have
\[g(I,x_{1},x_{2},\ldots,x_{m}) = (gI,x_{2},x_{2},\ldots,x_{m},x_1).\]
Define the map $\phi^A_I$ is the composition of maps $\phi_I^m\circ\phi_I^{m-1}\circ\dots\circ\phi_I^1$, where $\phi_I^i$ is the foam of Figure~\ref{fig:r1map} realizing
the $i$-th Reidemeister move. Define also, $\phi_I^B=g\phi_I=\phi_I^{m-1}\circ\phi_I^{m-1}\circ\dots\circ\phi_I^1\circ\phi_I^m$.
For $I\in\Cr(D)$ and a generator $g$ of $\Z_m$, we consider the following diagram.
\begin{equation}\label{eq:geometric_commutation_m}
  \begin{tikzcd}[column sep=3cm]
    D\langle\smalldown\dots\smalldown\rangle_{I} \ar[d, "{\phi^{A}_{I}}"] \ar[r,squiggly,blue,"{\rho_{g,I}}"] & D\langle\smalldown\dots\smalldown\rangle_{gI} \ar[d, "{\phi^{B}_{I}}"] \\
    D\langle\smallloop\dots\smallloop\rangle_{(I,0,\dots,0)} \ar[r,squiggly,blue,"{\rho_{g,(I,0,\dots,0)}}"] & D\langle\smallloop\dots\smallloop\rangle_{(gI,0,\dots,0)},
  \end{tikzcd}
\end{equation}
which is a generalization of the diagram \eqref{eq:geometric_commutation}. 
The same arguments as in the proof of Lemma~\ref{lem:if_theres_a_diagram_it_commutes} show that \eqref{eq:geometric_commutation_m} commutes. This is the geometry part of Step~2.

We pass to the algebra part.
Choose $\sas$ to be the sign assignment for $D$ and let $\sat$ be such that $\sas-g\sas=\partial\sat$.
Let $\sas'$ be the sign assignment on $\Cube(D')$ obtained via successive application of Lemma~\ref{lem:sas}.
\begin{lemma}\label{lem:reid1_sas}
  Suppose $I'_1,I'_2\in\Cube(D')$ and $I'_2$ is an immediate successor of $I'_1$. Write $I'_1=(I_1,x_1,\dots,x_m)$,
  $I'_2=(I_2,y_1,\dots,y_m)$ for $I_1,I_2\in\Cube(D)$. If $x_k\neq y_k$ for some $k$, then
  \[\sas'(I'_1,I'_2)=x_{k+1}+\dots+x_m.\]
  If $x_k=y_k$ for all $k$, then $\sas'(I'_1,I'_2)=x_1+\dots+x_m+\sas(I_1,I_2)$.
\end{lemma}
\begin{proof}
  Define $\sas'_\ell$ as the sign assignment on $D'_\ell$, the diagram obtained by successively applying the first $\ell$ Reidemeister moves.
  Suppose $x_k\neq y_k$ for some $k$.
  Then $I_1=I_2$, $x_i=y_i$ for $i\neq k$. By~\ref{item:vertical}:
  \[\sas'_k(I_{1},x_1,\dots,x_k)-\sas'_k(I_2,y_1,\dots,y_k)=0.\]
  Applying inductively for $j=k+1,\dots,m$ either~\ref{item:horizontal} (if $x_j=0$) or \eqref{eq:sas_prim} (if $x_j=1$),
  we obtain that
  \[\sas'_j(I_1,x_1,\dots,x_j)-\sas'_j(I_2,y_1,\dots,y_j)=x_{k+1}+\dots+x_j.\]
  Eventually, $\sas'(I'_1,I'_2)=x_{k+1}+\dots+x_m$.

  The case $x_1=y_1,\dots,x_m=y_m$ has an analogous proof.
\end{proof}
As a corollary, we prove the following fact, generalizing Lemma~\ref{lem:sat_on_equiv}.
\begin{lemma}\label{lem:sat_general_m}
  Suppose $\sas-g\sas=\partial\sat$.
  Consider the \(0\)-cochain on \(\SCube(D')\) defined by $\sat'(I,x_1,\dots,x_m)=\sat(I)+x_1(x_2+\dots+x_m)$.
  Then, for $I'_1,I'_2\in\Cube(D')$ and $I'_2$ an immediate successor of $I'_1$, we have
  \begin{equation}\label{eq:sas_to_sat}
  \sas'(I_1',I_2')-\sas'(gI_1',gI_2')=\sat'(I_1')-\sat(I_2').
\end{equation}
\end{lemma}
\begin{proof}
  Suppose $I_1=I_2$ and $x_i=y_i$ except that $x_k=0$ and $y_k=1$. By Lemma~\ref{lem:reid1_sas}, we have
  $\sas'(I_1',I_2')=x_{k+1}+\dots+x_m$. On the other hand, $gI'_1=(gI_1,x_2,\dots,x_m,x_1)$, that is
  \[\sas'(gI_1',gI_2')=x_{k+1}+\dots+x_m+x_1.\]
  Therefore,
  \[\sas'(I_1',I_2')-\sas'(gI_1',gI_2')=\begin{cases} x_1 & k>1 \\ x_2+\dots+x_m & k=1.\end{cases}\]

  On the other hand, if $k>1$, then
  \[\sat'(I_1')-\sat'(I_2')=x_1(x_2+\dots+x_m)-x_1(x_2+\dots+x_m+1)=x_1.\]
  If $k=1$, then
  \[\sat'(I_1')-\sat'(I_2')=x_1(x_2+\dots+x_m)-(x_1+1)(x_2+\dots+x_m)=x_2+\dots+x_m.\]
  If $I_1\neq I_2$, then $x_k=y_k$ for all $k$, and so
  \[\sas'(I_1',I_2')-\sas'(gI_1',gI_2')=\sas(I_1,I_2)-\sas(gI_1,gI_2).\]
  Equation~\eqref{eq:sas_to_sat} follows promptly, since in that case $\sat'(I_1')-\sat'(I_2')=\sat(I_1)-\sat(I_2)$.
\end{proof}
The remainder of the proof in case $m>2$ follows the proof in the case $m=2$. Namely, the methods of the proof of Lemma~\ref{lem:equal_as_maps} generalize
to show that $\phi^A_I$ and $\phi^B_I$ induce the same map 
\[\Phi\colon\llbra D\langle\smalldown,\dots,\smalldown\rangle\rrbra\xrightarrow{\simeq}\llbra D\langle\smallloop,\dots,\smallloop\rangle\rrbra.\]
An analog of the diagram \eqref{eq:algebraic_commutation} is 
\[
  \begin{tikzcd}[column sep=2cm]
    \llbra D\langle\smalldown,\dots,\smalldown\rangle, \sas \rrbra \ar[d,"\Phi"]\ar[r,blue,squiggly,"{\llbra \rho_{g}, \sat \rrbra}" above] 
    & \llbra D\langle\smalldown,\dots,\smalldown\rangle, \sas \rrbra \ar[d, "\Phi"] \\
    \llbra D\langle\smallloop,\dots,\smallloop, \sas'\rangle \rrbra \ar[r,blue,squiggly,"{\llbra \rho_{g}, \sat' \rrbra}" above]
    & \llbra D\langle\smallloop,\dots,\smallloop\rangle, g\sas' \rrbra.
  \end{tikzcd}
\]
The same argument as above implies that this diagram commutes. That is, $\Phi$ commutes with the group action. In particular, $\cF(\Phi)$
induces a $\Z_{m}$-equivariant chain homotopy equivalence, that is, a quasi-isomorphism in the category $\Kom{\cSym_N[\Z_{m}]}$.

\medskip
\textbf{Step 3. Reidemeister 2a move, $\Z_2$ action.}
We let $D'$ be the diagram after applying an equivariant Reidemeister 2a move.
Let us identify
\[\Cube(D') \cong \Cube(D) \times \{0,1\} \times \{-1,0\} \times \{0,1\} \times \{-1,0\}.\]
Given the resolution $I\in\Cube(D)$, denote
\[I'_{1} = (I,0,0,0,0), \quad I'_{2} = (I,1,-1,0,0), \quad I'_{3} = (I,0,0,1,-1), \quad I'_{4} = (I,1,-1,1,-1).\]
There are foams \(\phi_{I,k} \colon D_{I} \to D'_{I'_{k}}\), for \(k=1,2,3,4\), such that
the \(I\)-th component of the map $\Phi\colon\llbra D\rrbra\to\llbra D'\rrbra$ induced by the equivariant Reidemeister move is $\phi_I:=\phi_{I,1}+\phi_{I,2}+\phi_{I,3}+\phi_{I,4}$. Namely:
\begin{itemize}
  \item $\phi_{I,1}$ is the identity foam;
  \item $\phi_{I,2}$ is the foam from Figure~\ref{fig:try_to_draw_it} on the first place where the Reidemeister move is applied, followed by the identity foam;
  \item $\phi_{I,3}$ is the identity foam followed by the foam from Figure~\ref{fig:try_to_draw_it} for the second Reidemeister move;
  \item $\phi_{I,4}$ is the foam from Figure~\ref{fig:try_to_draw_it} for the first Reidemeister 2a move followed by the foam
    from Figure~\ref{fig:try_to_draw_it} for the second move.
\end{itemize}
Recall that $g$ is the generator of $\Z_2$. The argument as in Lemma~\ref{lem:equal_as_maps} shows the geometry part,
that is,
\begin{equation}\label{eq:gphi}
  g\phi_{I,1}=\phi_{gI,1},\ g\phi_{I,2}=\phi_{gI,3},\ g\phi_{I,3}=\phi_{gI,2},\ g\phi_{I,4}=\phi_{gI,4}.
\end{equation}
This, in turn, implies that $g\phi_I=\phi_{gI}$.
Let, as usual $\Phi\colon\llbra D\rrbra\to\llbra D'\rrbra$ be the map obtained by compositions of foams $\phi_I$. As the action
of $\Z_{m}$ on $\llbra D\rrbra$ and $\llbra D'\rrbra$ involves signs, we cannot immediately conclude that $\Phi$ commutes with $\Z_{m}$-action,
for the moment, we know that $\Phi$ commutes with the $\Z_{m}$-action up to sign. This is the end of the geometry part of the proof.

\smallskip
The sign assignment $\sas$ on $D$ induces the sign assignment $\sas'$ on $D'$ by adding crossing consecutively and applying Lemma~\ref{lem:sas}: first
we add crossings corresponding to $x_1$, then to $x_2$ and so on. We have the following variant of Lemma~\ref{lem:sat_general_m}.
\begin{lemma}\label{lem:sat_reid2}
  Suppose $\sas-g\sas=\partial\sat$.
  Consider the \(0\)-cochain on \(\SCube(D')\) defined by $\sat'(I,x_1,\dots,x_4)=\sat(I)+(x_1+x_2)(x_3+x_4)$.
  Then, $\sas'-g\sas'=\partial\sat'$.
\end{lemma}
\begin{proof}
  Choose $I_1',I_2'\in\Cube(D')$ such that $I_2'$ is an immediate successor of $I_1'$.
  Write $I'_s= (I_s,x_{1s},x_{2s},x_{3s},x_{4s})$. By Lemma~\ref{lem:reid1_sas},
  we have that
  \[\sas'(I_1',I_2')=\begin{cases} x_{j+1,1}+\dots+x_{41} & x_{j1}\neq x_{j2}\\ \sas(I_1,I_2) & x_{j1}=x_{j2} \textrm{ for all $j$}.\end{cases}\]
  Note that if $I'=(I,x_1,\dots,x_4))\in\Cube(D')$, then $gI'=(gI,x_3,x_4,x_1,x_2)$. Thus
  \[\sas'(I_1',I_2')-\sas'(gI_1',gI_2')=\begin{cases}
      x_{31}+x_{41} & x_{11}\neq x_{12}\textrm{ or }x_{21}\neq x_{22}\\
      x_{11}+x_{21} & x_{31}\neq x_{32}\textrm{ or }x_{41}\neq x_{42}\\
      \sat(I_1)-\sat(I_2) & x_{j1}=x_{j2} \textrm{ for all $j$}.
  \end{cases}\]
  The proof is concluded in the same way as in Lemma~\ref{lem:sat_general_m}.
\end{proof}
We conclude now Step~3 of the proof of Lemma~\ref{lem:equiv_reid}. Consider the diagram
\[
  \begin{tikzcd}[column sep=2cm]
    \llbra D, \sas \rrbra \ar[d,"\Phi"]\ar[r,blue,squiggly,"{\llbra \rho_{g}, \sat \rrbra}" above] 
    & \llbra D, \sas \rrbra \ar[d, "\Phi"] \\
    \llbra D, \sas' \rrbra \ar[r,blue,squiggly,"{\llbra \rho_{g}, \sat' \rrbra}" above]
    & \llbra D,g\sas' \rrbra.
  \end{tikzcd}
\]
Above, we showed that this diagram commutes up to sign. Now, given Lemma~\ref{lem:sat_reid2}, we conclude
that the diagram commutes. This shows that $\Phi$ is $\Z_{m}$-equivariant. We know from Theorem~\ref{thm:invariance_reid}
that $\Phi$ is a chain homotopy equivalence. As in Steps~1 and~2, we conclude that $\cF(\Phi)$ is a
quasi-isomorphism in $\Kom{\cSym_N[\Z_{m}]}$ category.

The case of Reidemeister~2a move with a general cyclic group $\Z_{m}$, as well as the case of Reidemeister~2b move is analogous. Reidemeister~3 move
induces a natural bijection between crossings of $D$ and $D'$. This bijection induces a tautological
correspondence  between sign
assignments for $D$ and $D'$, therefore only the `geometry' part is needed, but that part follows from the same argument as
Step~1. We leave the details to the reader.

\section{The skein spectral sequence}
\label{sec:skein-spectr-sequ}
If there are given three links $L_+$, $L_-$, and $L_0$, differing at a single crossing, which is positive for $L_+$, negative for $L_-$
and it is $0$-resolved for $L_0$, then there is a long exact sequence, called the skein exact sequence in $\sln$-homology;
see~\cite[Proposition 7.6]{Rasmussen_some}. To some extent, the skein exact sequence controls the behavior
of the $\sln$-homology under a single crossing change.

If a link is periodic, a single crossing change does not preserve periodicity. To keep the symmetry, we need to perform
a crossing change on \emph{an orbit} of crossings. 
An analog of a skein long exact sequence is a spectral sequence, called the skein spectral sequence.
A non-equivariant version of this spectral sequence was already considered by Khovanov \cite{khovanov_categorification_2000}. 
The equivariant construction we provide
is a generalization of \cite{Politarczyk-Jones,Politarczyk-Khovanov}.

\subsection{Review of $\Ind$ and $\Res$ functors}
Before we construct the spectral sequence, we review some basic facts about the $\Ind$ and $\Res$ functors.

Recall that if \(G\) is a finite group, we denote by \(BG\) the category with a single object \(\ast\) and \(\Hom_{BG}(\ast,\ast) = G\).
If \(\cB\) is an additive category, we denote by \(\cB[G] = \operatorname{Fun}(BG,\cB)\) the category of \(G\)-objects in \(\cB\).
If \(H \subset G\) is a subgroup, there exists a canonical inclusion of categories \(BH \subset BG\), which gives the restriction functor \(\operatorname{Res}^{G}_{H} \colon \cB[G] \to \cB[H]\).
The induction functor
\[\Ind^{G}_{H} \colon \cB[H] \to \cB[G]\]
is the biadjoint to the restriction functor, i.e., for any \(C \in \cB[H]\) and any \(D \in \cB[G]\) there are bijections which are natural in both \(C\) and \(D\)
\begin{align}
  \Hom_{\cB[G]}(\Ind^{G}_{H}(C),D) &\cong \Hom_{\cB[H]}(C,\Res^{G}_{H}(D)), \label{eq:res-ind-duality}\\
  \Hom_{\cB[G]}(C,\Ind^{G}_{H}(D)) &\cong \Hom_{\cB[H]}(\Res^{G}_{H}(C),D). \nonumber
\end{align}
Let \(g_{1},g_{2},\ldots,g_{k}\) be coset representatives of \(G/H\), then \(\Ind^{G}_{H}(C)\) can be identified with the direct sum
\begin{equation}\label{eq:ind_as_a_coproduct}
  \Ind^{G}_{H}(C) = \bigoplus_{i=1}^{k} g_{i} C,
\end{equation}
and for any \(g \in G\), we can write uniquely \(g = g_{i} \cdot h\), where \(h \in H\), then
\[g \cdot (-) \colon g_{j} C \to g_{k} C, \quad x \mapsto (h'g_{j}^{-1}h g_{j}) \cdot x,\]
where \(g_{k} = g_{i} \cdot g_{j} \cdot h'\), where \(h' \in H\) and \(g_{k}\) represents the coset of \(g_{i} \cdot g_{j}\).

\subsection{Construction of the spectral sequence}\label{sub:spectral}
Our construction begins with a construction for general link diagrams. Later on, we will restrict our
attention to periodic link diagrams.
Let \(D\) be a 
labelled 
link diagram, i.e., components of \(D\) are labelled by \(c \in \{1,2,\ldots,N\}\). Recall that $\Cr(D)$ is the set of crossings.

We define the extended cube of resolutions $\ExCube(D)$. Recall from Subsection~\ref{sub:sym_eq}
that to each crossing $i\in\Cr(D)$, we associate the set $C_i$ depending on the sign of the crossing and labels of the strands meeting at \(i\). We let $\wh{C_i}=C_i\cup\{\ast\}$.
The extended cube of resolutions $\ExCube(D)$ is the product of the $\wh{C_i}$.

For $\wh{I}\in\ExCube(D)$, we can define a partial resolution diagram $D_{\wh{I}}$: if $\wh{I}$ at the $i$-th crossing
is $\wh{I}(i)$, we choose the $\wh{I}(i)$-th resolution, as for the standard case $D_I$. If $\wh{I}(i)=\ast$, 
we do not resolve the crossing at all.
The resolutions are depicted in Figure~\ref{fig:smoothing-pos-crossing}, see also the skein relation in Figure~\ref{fig:resolution}.

If $\wh{I}\in\ExCube$, we define $\supp\wh{I}$ to be the set of those crossings $i\in\Cr(D)$
for which $\wh{I}(i)\neq\ast$. 
If $\wh{I},\wh{J}\in\ExCube$ are such that $\supp\wh{I}\cap\supp\wh{J}=\emptyset$ we define $\wh{I}\vee \wh{J}$ to be the resolution
such that
\[(\wh{I}\vee\wh{J})(i) = \begin{cases} \wh{I}(i)& i\in\supp\wh{I}\\ \wh{J}(i)& i\in\supp\wh{J}\\ \ast &\textrm{ otherwise.}\end{cases}\]

With a partial resolution $\wh{I} \in \ExCube(D)$ with support \(X\), we can associate the bracket $\llbra D_{\wh{I}}\rrbra$.
It is a cochain complex generated
by $D_{I}$ such that $I$ and $\wh{I}$ coincide on $X$.
The differential is given by foams of Figure~\ref{fig:differential} with the sign assignment \(\sas_{I}\) inherited from the sign assignment \(\sas\) on $D$.
The set of crossings of $D_{\wh{I}}$ is a subset of $\Cr(D)$. For any $J\in\Cube(D_{\wh{I}})$, we write $\wh{J}\in\ExCube(D)$
obtained by adding value $\wh{J}(x)=\ast$ for each $x\in X$. We write $\wh{I}\vee J$ for $\wh{I}\vee\wh{J}$.
We define also
\[\deg\wh{I}=\sum_{i\in\supp \wh{I}}I(i).\]

For a fixed subset $X\subset\Cr(D)$, we let 
\[A(X)=\{\wh{I}\in\ExCube(D)\colon\supp\wh{I}=X\},\ \ A_k(X)=\{\wh{I}\in A(X)\colon \deg\wh{I}=k\}. \]

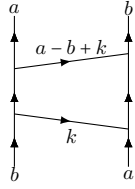
\begin{figure}
  \centering
  \begin{tikzpicture}[every node/.style={scale=0.7}]
    \draw[myendarrow,mybegarrow,mymidarrow] (1.5,-1) -- (1.5,1);
    \draw[myendarrow,mybegarrow,mymidarrow] (3,-1) -- (3,1);
    \draw(1.5,1.1) node {$a$};
    \draw(1.5,-1.1) node {$b$};
    \draw(3,1.1) node {$b$};
    \draw(3,-1.1) node {$a$};
    \draw[mymidarrow] (1.5,0.3) -- node[midway, above] {$a-b+k$} (3,0.5);
    \draw[mymidarrow] (1.5,-0.3) -- node[midway, below] {$k$} (3,-0.5);
  \end{tikzpicture}
  \caption{The \(k\)-smoothing of a positive crossing, for \(0 \leq k \leq b \leq a\). In order to obtain the \(k\)-smoothing of a negative crossing, reflect the above picture about the vertical line and switch labels.}
  \label{fig:smoothing-pos-crossing}
\end{figure}
Let $X\subset \Cr(D)$ be a subset of positive crossings (if $X$ is subset of negative crossings, the discussion is parallel). 
Set $Y=\Cr(D)\setminus X$. We let $\Cube(X),\Cube(Y)$ be the cubes of resolutions for $X$ and $Y$ (i.e. $\Cube(X)=\prod_{i\in X} C_i$,
$\Cube(Y)=\prod_{i\in Y} C_i$). We have $\Cube(D)=\Cube(X)\times\Cube(Y)$. Moreover, for $I\in\Cube(D)$, we let $I_X,I_Y$
be the projections. 

Consider the bicomplex obtained
by treating separately crossings of $D$ in $X$ and not in $X$. More precisely, suppose $I\in\Cube(D)$.
We let $\wh{I}$ be the element in $\ExCube(D)$ obtained by taking $I_X$ and extending $I_X$ by $\ast$ over all crossings in $Y$.
There is a partial resolution $D_{\wh{I}}$, whose set of crossings is $Y$. That is to say, $I_Y$ gives a resolution
of $D_{\wh{I}}$. Clearly, $(D_{\wh{I}})_{I_Y}=D_I$.
We can regard $\llbra D\rrbra$ as the following bicomplex.
\begin{equation}\label{eq:bicomplex}
  0\to \bigoplus_{\wh{I}\in A_0(X)} \llbra D_{\wh{I}}\rrbra\{q^{-b(N-b)|X|}\}
    \xrightarrow{\pm d_0}
    \bigoplus_{\wh{I}\in A_1(X)}\llbra D_{\wh{I}}\rrbra\{q^{1-b(N-b)|X|}\}
      \xrightarrow{\pm d_1}\dots.
\end{equation}
Here $q^{\bullet}$ is the grading shift. The differentials $d_i$ are computed as follows.
Take $I,J\in\Cube(D)$ such that $J$ is an immediate successor of $I$. We have two
cases
\begin{itemize}
  \item if $I_X=J_X$, then the component of the differential on $\llbra D\rrbra$ from $I$ to $J$ contributes
    to the differential on $\llbra D_{\wh{I}}\rrbra$ (going from $(D_{\wh{I}})_{I_Y}$ to $(D_{\wh{I}})_{J_Y}$ with
    sign $(-1)^{\sas(I,J)}$. We call this component the \emph{internal differential}
    or \emph{horizontal differential};
  \item if $I_Y=J_Y$, then we set $s=\deg\wh{I}$. The component of the differential on $\llbra D\rrbra$ contributes to
    the differential $d_s$ going from $\llbra D_{\wh{I}}\rrbra$ to $\llbra D_{\wh{J}}\rrbra$, more precisely
    going from $(D_{\wh{I}})_{I_Y}$ to $(D_{\wh{J}})_{I_Y}$.
    The sign is given by $\sas(I,J)$.  
    We refer to this component as the \emph{external} or
    \emph{vertical differential}.
\end{itemize}
The sum of the vertical differential and the horizontal differential is exactly the differential on $\llbra D\rrbra$.
Hence, we have a result, which we record for future use.

\begin{lemma}\label{lem:bicomplex}
  The total complex of the chain complex \eqref{eq:bicomplex} is equal to $\llbra D\rrbra$.
\end{lemma}

A general principle is that a bicomplex leads to a spectral sequence. To apply this principle, we apply the functor $\cF$
to \eqref{eq:bicomplex}, so that we work in an abelian category.
To make the notation more concise, we define the triply graded bicomplex
\[M(D,X)^{k,\ell,h}=\bigoplus_{\wh{I}\in A_k(X)}\cF(\llbra D_{\wh{I}}\rrbra)\{q^{k-|X|b(N-b)}\}_{\ell,h}.\]
Here, the subscript on the right-hand side should be understood as taking the part at homological grading $\ell$
and quantum grading $q$.
If $X$ is a subset of negative crossings, we define
\[M(D,X)^{k,\ell,h}=\bigoplus_{\wh{I}\in A_k(X)}\cF(\llbra D_{\wh{I}}\rrbra)\{q^{-k+|X|b(N-b)}\}_{\ell,h}\]
and $k$ takes negative values (we omit the details).

As noted above,
in the bicomplex $M(D,X)^{\bullet,\ell,h}$, there is an internal (horizontal) differential, and the external (vertical) differential
going from $M(D,X)^{\bullet,\ell,h}$ to $M(D,X)^{\bullet+1,\ell,h}$.
\begin{lemma}\label{lem:bicomplex2}
  The cohomology of the total complex $\Tot^{r,h} M(D,X)=\bigoplus_{k+\ell=r} M(D,X)^{k,\ell,h}$ is the $\S_N$-valued
  Khovanov--Rozansky homology of the link.
\end{lemma}
\begin{proof}
If we count only the horizontal differentials, up to grading shift, the cohomology of the complex $M(D,X)^{k,\ell,h}$ is equal to the sum of Khovanov--Rozansky
homologies of webs $D_{\wh{I}}$ where $\wh{I}\in A_k(X)$, that is of webs obtained as partial resolutions of the link diagram.
Lemma~\ref{lem:bicomplex} implies the existence of a spectral sequence, whose $E^1$ page is the cohomology of the complex $M(D,X)^{k,\ell,h}$.
The spectral sequence abutes to the Khovanov-Rozansky homology of $D$.
\end{proof}

\smallskip
Suppose now $D$ is a $\Z_m$-periodic link diagram.
We will be mainly interested in the case when \(X\) is an orbit of crossings.
In this case, the group \(\Z_{m}\) acts on \(\Cr(D)\) preserving \(X\) and there is an induced action on \(A_{k}(X)\), for any \(k \in\Z\).
For \(\wh{I} \in A_{k}(X)\) let \(\Iso(\wh{I}) = \{g \in \Z_{m} \colon \wh{I} \circ g = \wh{I}\}\) be the isotropy group of \(\wh{I}\).
For any \(d|m\) let
\begin{equation}\label{eq:aks}
  A_{k}^{d}(X) = \{\wh{I} \in A_{k}(X) \colon \Iso(\wh{I}) = \Z_{d}\}
\end{equation}
and denote by \(\overline{A}_{k}^{d}(X)\) the quotient of \(A_{k}^{d}(X)\) by the action of \(\Z_{m}\).
Observe that for any \(\wh{I} \in A_{k}^{d}(X)\), the diagram \(D_{\wh{I}}\) is \(d\)-periodic.
Furthermore, for any \(\wh{I} \in A(X)\), and any \(g \in \Z_{m}\), the group action on \(\llbra D \rrbra\) gives a map \(\cG_{g,\wh{I}} \colon \llbra D_{\wh{I}}, \sas_{\wh{I}} \rrbra \to \llbra D_{g \wh{I}},
\sas_{g \wh{I}} \rrbra\), where \(\sas_{\wh{I}}\) and \(\sas_{g \wh{I}}\) denote restrictions of the sign assignment \(\sas\) on \(\Cube(D)\) to \(\Cube(D_{\wh{I}})\) and \(\Cube(D_{g \wh{I}})\), respectively.


In order to generalize Lemma~\ref{lem:bicomplex} to the equivariant setting,
suppose $X$ is a $\Z_m$-invariant set of crossings, such that either all crossings in $X$ are positive
or all crossings in $X$ are negative.
For any \(\wh{I} \in A_{k}^{d}(X)\), the partial resolution \(D_{\wh{I}}\) is a \(d\)-periodic diagram.
Consider the natural \(\Z_{d}\)-action on \(\llbra D_{\wh{I}} \rrbra\) and \(\mathcal{F}(\llbra D_{\wh{I}} \rrbra)\) as defined in Proposition~\ref{prop:group_action_on_D}.
The equivariant version of the bicomplex $M(D,X)^{k,\ell,h}$ is defined by
\begin{equation}\label{eq:EM_def}
  \EM(D,X)^{k,l,\bullet} = 
  \begin{cases}
    \bigoplus_{d \mid k} \bigoplus_{\wh{I} \in \overline{A}_{k}^{d}(X)} \Ind_{\Z_{d}}^{\Z_{m}} \left( \cF(\llbra D_{\wh{I}}\rrbra) \{q^{-|X|b(N-b)+k} \} \otimes \C_{s(m,d,\wh{I})} \right) & \textrm{ positive},
  \\
    \bigoplus_{d \mid k} \bigoplus_{\wh{I} \in \overline{A}_{k}^{d}(X)} \Ind_{\Z_{d}}^{\Z_{m}} \left( \cF(\llbra D_{\wh{I}}\rrbra) \{q^{|X|b(N-b)-k}\} \otimes \C_{s(m,d,\wh{I})} \right) & \textrm{ negative}.
\end{cases}
\end{equation}
Here, the tensor product is over the ring $\C[\Z_d]$, where the $\cSym_N[\Z_d]$-module $\cF(\llbra D_{\wh{}}\rrbra)$ 
can be regarded as a right $\C[\Z_d]$-module (via the standard action of $\C$ on $\cSym_N$). The space
$\C_j$ a one dimensional complex vector space and the $\Z_d$-action is either trivial (if $j=0$) or
it is the sign action (the generator of $\Z_d$ acts on $\C$ by multiplication by $-1$), if $j=1$. For the latter, $d$ must necessarily be even.
The choice of representation depends on the number $s(m,d,\wh{I})\in\F_2$:
\begin{equation}
  \label{eq:twisting}
  s(m,d,\wh{I}) = \sat(I_{0}) + \sat(gI_{0}) + \cdots + \sat(g^{m/d-1}I_{0})
\end{equation}
with $g$ being the generator of $\Z_m$ acting on the plane by a rotation by the angle $\frac{2\pi}{m}$,
and $I_0=\wh{I}\vee J_0$ for $J_0=(0,\dots,0)\in\Cube(D_{\wh{I}})$
and $\sat$ a $0$-cochain on $\SCube(D)$ satisfying $g\sas-\sas=\partial \sat$, $\sat((0,\dots,0))=0$.
\begin{lemma}\label{lem:EM_is_M}
  There is an isomorphism $\EM(D,X)\cong M(D,X)$ as complexes of $\S_N$-modules.
\end{lemma}
\begin{proof}
  If we forget the group action, we need not consider tensoring by $\C_{s(m,d,\wh{I})}$. In the definition
  of $M(D,X)$ we sum over $\wh{I}\in A_k(X)$. In the definition of $\EM(D,X)$, we sum over the quotient
  $\overline{A}_k^d(X)$, but we compensate this by using the $\Ind$ functor. The statement follows
  from \eqref{eq:ind_as_a_coproduct}, we omit tedious but straightforward details.
\end{proof}


The next result generalizes Lemma~\ref{lem:bicomplex} as well as
an analogous statement of \cite[Lemma 2.10 and Section 4]{Politarczyk-Khovanov}.
\begin{lemma}\label{lem:equivariant_skein_relation}
  There is an isomorphism of the total complex of $\EM(D,X)$ and $\cF(\llbra D\rrbra)$ as $\S_N[\Z_m]$-modules.
\end{lemma}
\begin{proof}
  Given Lemma~\ref{lem:EM_is_M}, we need to show that the isomorphism of $\EM(D,X)$ and $\cF(\llbra D\rrbra)$ as $\S_N$-modules
  is $\Z_m$-equivariant.

  Recall that we let $g$ be the generator of \(\Z_{m}\) acting on the plane by rotation by the angle~\(\frac{2 \pi}{m}\).
  Fix a sign assignment $\sas$ on $D$, and let $\sat$ be the $0$-cochain satisfying $\partial \sat=g\sas-\sas$, $\sat((0,\dots,0))=0$.

  Suppose that \(d\) is a divisor of \(m\) and set \(h = g^{m/d}\) to be a generator of \(\Z_{d} \subset \Z_{m}\).
  Fix \(\wh{I} \in \overline{A}_{k}^{d}(X)\) and consider the partial resolution \(D_{\wh{I}}\).
  We define $\sas_{\wh{I}}$ to be the sign assignment on $\Cube(D_{\wh{I}})$, defined as $\sas_{\wh{I}}(J,J')=\sas(\wh{I}\vee J,\wh{I}\vee J')$.
  
  Since \(D_{\wh{I}}\) is a \(d\)-periodic diagram, we can define an action of \(\Z_{d}\) on \(\llbra D_{\wh{I}} \rrbra\) using the recipe from Proposition~\ref{prop:group_action_on_D}. In particular, we let $\sat_{\wh{I}}$ be the $0$-cochain on $\Cube(D_{\wh{I}})$
  such that $h\sas_{\wh{I}}-\sas_{\wh{I}}=\sat_{\wh{I}}$ and $\sat_{\wh{I}}(0,\dots,0)=0$.
  Denote by \(\cH_{h} \colon \llbra D_{\wh{I}} \rrbra \to \llbra D_{\wh{I}} \rrbra\) the map corresponding to the action of a generator \(h\) (i.e. rotation by the angle \(\frac{2\pi}{d}\)).

  There are now two choices of actions of $h$ on $\llbra D_{\wh{I}}\rrbra$. The external one is $(\cG_g)^{m/d}$, where $\cG_g$
  is the action constructed in Proposition~\ref{prop:group_action_on_D} for $\llbra D\rrbra$.
  The other map is $\cH_h$. These maps are equal up to a sign choice, in fact, they are obtained from
  the same sets of foams. They might differ by the sign. To complete the proof
  of Lemma~\ref{lem:equivariant_skein_relation}, we need to compare $\sat_{\wh{I}}(J)$ and $\sat(\wh{I}\vee\wh{J})$ for
  $J\in\Cube(D_{\wh{I}})$. Note that by Remark~\ref{remark:sign-ambiguity-gp-action}, the difference between the two
  does not depend on $J$ itself. In fact, any two $0$-cochains $\sat_{1},\sat_2$ such that
  $h\sas_{\wh{I}}-\sas_{\wh{I}}=\partial \sat_1= \partial\sat_2$, so $J\mapsto\sat_{\wh{I}}(J)$ and $J\mapsto\sat(\wh{I}\vee J)$
  are either equal, or differ by an overall sign. 

  To check this sign, let $J_0=(0,\dots,0)\in\Cube(D_{\wh{I}})$, 
  set $I_0=J_0\vee\wh{I}$.
  Suppose that \(\rho_{g} \colon D_{I_{0}} \to D_{g I_{0}}\) is the foam realizing the rotation of \(D_{I_{0}}\) by \(g \in \Z_{}\), i.e., the \(I_{0}\)-th component of \(\cG_{g}\)
  is equal to \((-1)^{\sat(I_{0})} \rho_{g}\).
  Let \(\rho_{h} = \rho_{g^{m/d-1}I_{0}} \circ \cdots \circ \rho_{gI_{0}} \circ \rho_{I_{0}}\) and \(\sat_{h}(I_0) = \sat(I_{0}) + \sat(gI_{0}) + \cdots + \sat(g^{m/d-1}I_{0})\), then the \(I_{0}\)-th component of \(\mathcal{H}_{h}\) is equal to
  \((-1)^{\sat_{h}(I_0)} \rho_{h}\).
  By convention from the proof of Proposition~\ref{prop:group_action_on_D},
  $\sat_{\wh{I}}(0,\dots,0)=0$. That is,
  the \(J_{0}\)-th component of \(\mathcal{H}_{h}\) is equal to \(\rho_{h}\).
  Consequently, $s(m,d,\wh{I})=\sat_h$. We conclude by \eqref{eq:twisting}.
\end{proof}

Using Lemma~\ref{lem:equivariant_skein_relation}, we obtain the skein spectral sequence for equivariant Khovanov-Rozansky homology.
\begin{proposition}[Skein spectral sequence]\label{prop:skein-spectral-sequence}
  Let \(D\) be an \(m=p^{\ell}\)-periodic labelled link diagram, where \(p\) is an odd prime and \(\ell \geq 1\).
  Let \(X \subset \Cr(D)\) be an orbit of crossings between an \(a\)-{labelled} overstrand and a \(b\)-{labelled} understrand, where \(a \geq b\)
  If \(0 \leq u \leq \ell\) and \(X\) consists of positive crossings, we obtain, for any \(1 \leq s \leq |X|b\) a spectral sequence with
  \begin{equation}\label{eq:to_be_E_1_page}
  E_{1}^{k,l,\bullet}(p^{\ell-u}) = \bigoplus_{p^{s} \mid k} \bigoplus_{{\wh{I}} \in \overline{A}_{k}^{s}(D,X) }  \EKR_{N}^{\bullet,\bullet,\kappa(u,s)}(D_{{\wh{I}}})^{\lambda(u,s)}t^{k} q^{-|X|b(N-b)-k},\end{equation}
  with \(0 \leq k \leq p^{\ell} b\) and
  \[\kappa(u,s) =
    \begin{cases}
      1, & u \geq s, \\
      p^{s-u}, & \text{otherwise},
    \end{cases}
    \quad
    \lambda(u,s) =
    \begin{cases}
      \phi(p^{\ell-u}), & u \geq s, \\
      p^{\ell-s}, & u < s,
    \end{cases}
  \]
  converging to \(\KR_{N}^{\bullet,\bullet,p^{\ell-u}}(D)\).
  Similarly, if \(X\) consists of negative crossings, we obtain a spectral sequence with
  \[E_{1}^{k,l,\bullet}(p^{\ell-u}) = \bigoplus_{p^{s} \mid k} \bigoplus_{{\wh{I}} \in \overline{A}_{p^{\ell}+k}^{s}(D,X) }  \EKR_{N}^{\bullet,\bullet,\kappa(u,s)}(D_{{\wh{I}}})^{\lambda(u,s)}t^{k} q^{|X|b(N-b)-k},\]
  where \(-p^{\ell}b \leq k \leq 0\).
\end{proposition}
\begin{remark}
  It is clear that the same proof works for $\ELee_N$. 
\end{remark}
\begin{proof}[Proof of Proposition~\ref{prop:skein-spectral-sequence}]
  We will give only the proof in the positive case.
  Note that the total complex of $\EM(D,X)$ is the complex of $\cF(\llbra D\rrbra)$ by Lemma~\ref{lem:equivariant_skein_relation}.
  By \(\EM_{0}(D,X)\) we will denote the singular specialization of \(\EM(D,X)\).
  Fix \(0 \leq u \leq \ell\) and consider the bicomplex derived from \(\EM(D,X)\):
    \[\EM^{k,l,\ast}(D,X,p^{\ell-u}) := \Hom_{\C[\Z_{p^{\ell}}]}(\C[\Z_{p^{\ell}}]_{p^{\ell-u}}, \EM_{0}^{k,l,\ast}(D,X)).\]
  On considering separately the internal (vertical) and the external (horizontal) differentials in $\EM(D,X,p^{\ell-u})$, we obtain a spectral sequence
  of $\C[\Z_m]$-modules converging to \(\EKR_{N}^{\ast,\ast,p^{\ell-u}}(D)\), whose $E_1$-page is given by
  \begin{align*}
    E_{1}^{k,l,\ast}(p^{\ell-u}) &= H^{k,\ast} (\EM_{0}^{\ast,l,\ast}(D,X,p^{\ell-u}), d_{\operatorname{vert}}) \\
    &\cong \Hom_{\C[\Z_{p^{\ell}}]} (\C[\Z_{p^{\ell}}]_{p^{\ell-u}}, H^{k,\ast}(\EM_{0}^{\ast,l,\ast}(D,X),d_{\operatorname{vert}})).
  \end{align*}
  i.e., we take the vertical homology of \(\EM(D,X,p^{\ell-u})\).
  The aim of the proof is to show that this page is isomorphic to~\eqref{eq:to_be_E_1_page}.

  Before further studying the \(E_{1}\)-page of the spectral sequence, let us consider the decomposition of the group algebra.
  Consider the group algebra \(\C[\Z_{p^{\ell}}]\) and recall from Section~\ref{sec:decomp-sln-homol} that 
  \[\C[\Z_{p^{\ell}}]_{p^{\ell-u}} = \bigoplus_{\stackrel{0 \leq i < p^{\ell}}{\gcd(i,p^{\ell}) = p^{u}}} \C_{\xi_{p^{\ell}}^{i}}.\] 
  Observe that for any \(0 \leq s \leq \ell\) we have
  \[\Res^{\Z_{p^{\ell}}}_{\Z_{p^{s}}}(\C_{\xi_{p^{\ell-u}}^{j}}) = \C_{(\xi_{p^{\ell-u}}^{j})^{p^{\ell-s}}} = 
    \begin{cases}
      \C_{1}, & s \leq u, \\
      \C_{\xi_{p^{s-u}}^{j}}, & s > u.
    \end{cases}
  \]
  Consequently,
  \begin{equation}
    \label{eq:restriction-functor}
    \Res^{\Z_{p^{\ell}}}_{\Z_{p^{s}}}(\C[\Z_{p^{\ell}}]_{p^{\ell-u}}) =
    \begin{cases}
      \C_{1}^{\phi(p^{\ell-u})}, & s \leq u, \\
      \C[\Z_{p^{s}}]_{p^{s-u}}^{p^{\ell-s}}, & s > u.
    \end{cases}
  \end{equation}
  
  Using the definition of \(\EM_{0}(D,X)\), we obtain
  \begin{align*}
    E^{k,l,\bullet}_{1}(p^{\ell-u}) &= \Hom_{\C[\Z_{p^{\ell}}]} (\C[\Z_{p^{\ell}}]_{p^{\ell-u}}, H^{k,\ast}(\EM_{0}^{\ast,l,\ast}(D,X),d_{\operatorname{vert}})) \\
    &\cong  \bigoplus_{p^{s} \mid k} \bigoplus_{{\wh{I}} \in \overline{A}_{k}^{s}(D,X)} \Hom_{\C[\Z_{p^{\ell}}]} \left( \C[\Z_{p^{\ell}}]_{p^{\ell-u}}, \Ind_{\Z_{p^{s}}}^{\Z_{p^{\ell}}} \EKR (D_{{\wh{I}}})t^{k}q^{-|X|b(N-b)+k}\right).
  \end{align*}
  We consider an individual summand of the right-hand side:
  \begin{align*}
    &\Hom_{\C[\Z_{p^{\ell}}]} \left( \C[\Z_{p^{\ell}}]_{p^{\ell-u}}, \Ind_{\Z_{p^{s}}}^{\Z_{p^{\ell}}} \EKR (D_{{\wh{I}}})t^{k}q^{-|X|b(N-b)+k}\right) \\
      &\stackrel{\eqref{eq:res-ind-duality}}{\cong} \Hom_{\C[\Z_{p^{s}}]}\left( \Res^{\Z_{p^{\ell}}}_{\Z_{p^{s}}} (\C[\Z_{p^{\ell}}]_{p^{\ell-u}}), \EKR(D_{{\wh{I}}}) t^{k} q^{-|X|b(N-b)+k} \right) \\
      &\stackrel{\eqref{eq:restriction-functor}}{=}
      \begin{cases}
        \Hom_{\C[\Z_{p^{\ell}}]}(\C_{1}^{\phi(p^{\ell-u})}, \EKR(D_{{\wh{I}}}) t^{k} q^{-|X|b(N-b)+k}) , & s \leq u, \\
        \Hom_{\C[\Z_{p^{\ell}}]}(\C[\Z_{p^{s}}]_{p^{s-u}}^{p^{\ell-s}}, \EKR(D_{{\wh{I}}}) t^{k} q^{-|X|b(N-b)+k}) & s > u
      \end{cases}
    \\
    &=
      \begin{cases}
        \EKR^{\ast,\ast,1}(D_{{\wh{I}}})^{\phi(p^{\ell-u})} t^{k} q^{-|X|b(N-b)+k}, & s \leq u \\
        \EKR^{\ast,\ast,p^{s-u}}(D_{{\wh{I}}})^{p^{\ell-s}} t^{k} q^{-|X|b(N-b)+k}, & s > u
      \end{cases} \\
    &= \EKR^{\ast,\ast,\kappa(u,s)}(D_{{\wh{I}}} )^{\lambda(u,s)} t^{k} q^{-|X|b(N-b)+k}.
  \end{align*}
  Hence, the proposition follows.
\end{proof}

\section{Polynomial invariants}\label{sec:polynomial}
\subsection{Poincar\'e polynomials of $\sln$ and Lee homology}
We first recall a standard construction.
\begin{definition}
  Let $L$ be a link. The Poincar\'e polynomial $\KRP_N(K)$ of $\sln$-homology is defined as:
  \[
    \KRP_{N}(L)=\sum_{k,r}t^kq^r\dim_{\C}\KR_N^{k,r}(L).
  \]
  The \emph{$\LeeP_N$ polynomial} is defined as
  \[\LeeP_N(L)=\sum_{k,r}\dim_\C\Gr^r\Lee_N^k(L)t^kq^r,\]
  where $\Gr^r$ is the $r$-th graded part of the filtered $\Lee_N$ homology.
\end{definition}
For an $m$-periodic link, we can refine this definition; generalizing the approach of \cite{Politarczyk-Jones}. 

\begin{definition}
  Let $L$ be an $m$-periodic link and $d|m$. The \emph{equivariant Poincar\'e polynomial},
  for $\sln$-homology is defined as:
\begin{equation}\label{eq:equivariant_poincare}
    \KRP_{N,d}(L)=\sum_{k,r}t^kq^r\dim_{\C_d}\EKR_N^{k,r,d}(L).
\end{equation}
  The \emph{equivariant Lee polynomial} is:
  \[\LeeP_{N,d}(L)=\sum_{k,r}\dim_{\C_d}\Gr^r\ELee_N^{k,d}(L)t^kq^r,\]
\end{definition}
We have the following basic identity.
\begin{equation}\label{eq:PN_and_PNd}
  \KRP_N(L)=\sum_{d|m} \phi(d)\KRP_{N,d},
\end{equation}
where $\phi(d) = \# \{1 \leq i \leq d \colon \gcd(i,d)=1\}$ is the Euler's totient function.

The Lee polynomial can be computed from~Proposition~\ref{prop:computingLee}. 
We give the precise formula for a knot.
We refer to \cite[Proposition 2.6]{Lewark_spectral}. Other references include \cite{Gornik, Lobb_slice,Lobb_gornik,RoseWedrich,Wu_quantum}.
\begin{lemma}\label{lem:anyknot}
  For any knot $K$ we have $\LeeP_N(K)=q^{s_N(K)}(q^{-N+1}+q^{-N+3}+\dots+q^{N-1})$,
  where $s_N(K)$ is the Lewark's $s_N$-invariant; see \cite{Lewark_spectral}.
\end{lemma}

The following statement is an immediate consequence of Lemma~\ref{lem:trivialLee}.
\begin{lemma}\label{lem:eleep}
  If the action of the symmetry group on the set of components of $L$ is trivial, then $\LeeP_{N,d}$ is equal to $\LeeP_N$ if
  $d=1$, and to zero otherwise.
\end{lemma}

The next result relates the polynomials $\KRP$ and $\LeeP$. Its proof is the same as in the Khovanov case, see \cite[Proposition 2.17]{BP}.
See also \cite[Theorem 5.1]{Machine} and~\cite[Proposition 5.2]{Lewark_spectral}.
\begin{proposition}\label{prop:lee_spectral}
  For any link $L$ there are polynomials $R_1,\dots$ with non-negative coefficients such that
  \[\KRP_N(L)=\LeeP_N(L)+(1+tq^{2N})R_1+(1+tq^{4N})R_2+\dots.\]
  Moreover, if $L$ is $m$ periodic and $d|m$, then
  \[\KRP_{N,d}(L)=\LeeP_{N,d}(L)+(1+tq^{2N})R^d_1+(1+tq^{4N})R^d_2+\dots\]
  for polynomials $R^d_1,\dots$ with non-negative coefficients.
\end{proposition}
\begin{proof}
  The statement follows from the same argument as in \cite[Proposition 2.17]{BP}. The 
  differential on the $k$-th page of the Lee-Gornik spectral sequence (Theorem~\ref{thm:leegornik}) has
  bidegree $(1,2Nk)$; compare \cite[Section 3.2]{Lewark_sl3_calculation}. 
The factor $2$ before $Nk$ comes from the grading of $X_i$-variables in $\S_N$.

  The polynomials $R_k$ and $R^d_k$ are Poincar\'e
  polynomials associates to the image of the differential on the $E_k$-page, compare \cite[Proposition 2.16]{BP}.
\end{proof}
\subsection{The Reshetikhin--Turaev $\RT_N$ polynomials}\label{sub:RT_poly}
Recall \cite{HOMFLY,PT} that the HOMFLYPT polynomial of an oriented link in $\R^3$ is the polynomial $X(a,b)$ uniquely defined by
its value on the unknot and the skein relation
\begin{equation}\label{eq:hom_skein}aX_{L_+}(a,b)-a^{-1}X_{L_-}(a,b)=bX_{L_0}(a,b),\end{equation}
where $L_+$, $L_0$ and $L_-$ are links differing at a single crossing.

For $N \ge 0$ we consider the following specialization of the HOMFLYPT polynomial.
\begin{equation}\label{eq:RT_N}\RT_{N}(q)=X(q^{N},q-q^{-1})\end{equation}
with the normalization that
\[\RT_{N}(\textrm{unknot})=\frac{q^{N}-q^{-N}}{q-q^{-1}}.\]
The polynomial $\RT_{0}$ (with a slightly different normalization) is the Alexander polynomial, $\RT_{1}\equiv 1$, and $\RT_2$ is the Jones polynomial.
For $N>2$ we obtain a polynomial, which we refer to as the $\sln$-polynomial of $L$.
\begin{remark}
  The polynomials $\RT_N$ were introduced originally by Reshetikhin and Turaev \cite{ResheitikhinTuraev,Turaev_YB} via representation theory.
  We refer to \cite{MOY} for relations with webs.
\end{remark}

It was shown by \cite{Khovanov_Rozanski1,Khovanov_Rozanski2} 
that $\sln$-homology categorifies the $\sln$ polynomial.
For later use, we recall the statement.
\begin{lemma}
  Let $L$ be a link and $\KR_N^{k,r}(L)$ its $\sln$-homology. Then
  \[\RT_{N}(L)=\sum_{k,r}(-1)^kq^r\dim\KR_N^{k,r}(L)=\KRP_N|_{t=-1}.\]
\end{lemma}

The skein relation for the HOMFLYPT polynomial specializes to the skein relation for the $\RT_N$ polynomial.
\begin{equation}\label{eq:skein}
  q^{N} \raisebox{-1em}{\begin{tikzpicture} \draw[->] (0.5,-0.5) -- (-0.5,0.5);\fill[white] (0,0) circle (0.2); \draw[->] (-0.5,-0.5) -- (0.5,0.5); \end{tikzpicture}}
  -
  q^{-N} \raisebox{-1em}{\begin{tikzpicture} \draw[->] (-0.5,-0.5) -- (0.5,0.5); \fill[white] (0,0) circle (0.2); \draw[->] (0.5,-0.5) -- (-0.5,0.5);\end{tikzpicture}}
  =
  (q-q^{-1})
  \raisebox{-1em}{\begin{tikzpicture}
    \draw[->] (-0.5,-0.5) .. controls ++ (0.2,0.2) and ++(0.2,-0.2) .. (-0.5,0.5);
    \draw[->] (0.2,-0.5) .. controls ++ (-0.2,0.2) and ++(-0.2,-0.2) .. (0.2,0.5);
\end{tikzpicture}}
\end{equation}

\subsection{Difference polynomials}
Assume now that $m=p^\ell$ is a prime power and let \(D\) be an \(m\)-periodic diagram of an $m$-periodic link $L$. The
$\sln$ homology of $L$ decomposes as in \eqref{eq:third_grading}. We set
\[\RT_{N,j}=\KRP_{N,p^j}|_{t=-1},\]
where $\KRP_{N,d}$ is as in \eqref{eq:equivariant_poincare}.

For future use, we note the following corollary of Proposition~\ref{prop:e_mirror}.
\begin{corollary}\label{cor:p_mirror}
  If $L$ is a $p^n$ periodic link and $L'$ its mirror, then $\RT_{N,j}(L)(q)=\RT_{N,j}(L')(q^{-1})$.
\end{corollary}
\begin{proof}
  By Proposition~\ref{prop:e_mirror}, there is an isomorphism of $\C$-vector spaces $\EKR^{k,r,m}(L)=\EKR^{-k,-r,m}(L')$. The statement
  follows by definition of $\RT_{N,j}(L)$.
\end{proof}

We define the \emph{difference $\sln$-polynomials via}
\[
  \DP_{N,j}(D) =
  \begin{cases}
    \RT_{N,p^j}(D) - \RT_{N,p^{j+1}}(D) & 0 \leq j< \ell \\
    \RT_{N,p^{\ell}}(D) & j=\ell.
  \end{cases}
\]
Suppose $\Lpm$, $\Lmm$ and $\Lom$ are three $m=p^\ell$-periodic links for  for the periodic resolutions, that is, $\Lpm$ differs from $\Lmm$ and from $\Lom$
on an orbit of crossings.
\begin{proposition}[compare \expandafter{\cite[Theorem 4.7]{Politarczyk-Jones}}]\label{prop:jones}
  We have the following relations between $\DP_{N,j}(D)$ polynomials.
  \begin{enumerate}
  \item For \(j=0\) the following formula holds
    \[q^{m N} \DP_{N,0}(\Lpm) - q^{-m N} \DP_{N,0}(\Lmm) = (q^{-m} - q^{m}) \DDJ_{N,0}(\Lom).\]
  \item For any \(0 \leq j < \ell\), the following congruence is satisfied
    \[q^{m N} \DP_{N,\ell-j}(\Lpm) - q^{-m N} \DP_{N,\ell-j}(\Lmm) \equiv (q^{-m} - q^{m}) \DDJ_{N,\ell-j}(\Lom) \pmod{q^{p^{j}} - q^{-p^{j}}}.\]
  \end{enumerate}
\end{proposition}
\begin{proof}
  We argue as in the proof of~\cite[Theorem 3.6]{Politarczyk-Jones}.
    Suppose that \(\{E_{r}^{\ast,\ast},d_{r}\}_{r \geq 1}\) is a spectral sequence of graded finite-dimensional \(\C\)-vector spaces which converges to a double graded \(\C\)-vector space \(H^{\ast,\ast}\).
    Suppose, furthermore, that the spectral sequence collapses at a finite stage.
    For any \(r \geq 1\) define the \emph{Poincar\'e polynomials} of the page \(E_{r}^{\ast,\ast}\),
    \[P(E_{r}^{\ast,\ast}) = \sum_{i,j} t^{i+j} \operatorname{qdim}_{\C} E_{r}^{i,j}.\]
    Recall that for a graded \(\C\)-vector space \(V^{\ast}\), we define
    \[\operatorname{qdim}_{\C} = \sum_{i} q^{i} \dim_{\C} V^{i}.\]
    From~\cite[Exercise 1.7]{mcclearyUserGuideSpectral2001}, we deduce that for any \(r \geq 1\),
    \begin{equation}
      \label{eq:equality-Poincare-polynomials-spec-seq}
      P(E_{r}^{\ast,\ast})(-1,q) = P(E_{\infty}^{\ast,\ast})(-1,q) = \sum_{i,j}(-1)^{i} \operatorname{qdim}_{\C} H^{i,\ast}.
    \end{equation}

  Fix a \(p^{\ell}\)-periodic diagram \(D\).
    We  apply the formula~\eqref{eq:equality-Poincare-polynomials-spec-seq} to spectral sequences constructed in Proposition~\ref{prop:skein-spectral-sequence}.
    We obtain
    \[P(E_{1}^{\ast,\ast}(p^{\ell-u}))(-1,q) = P(E_{\infty}^{\ast,\ast})(-1,q) = \RT_{N,\ell-u}(D).\]
    The description of the \(E_{1}^{\ast,\ast}(p^{\ell-u})\) given in Proposition~\ref{prop:skein-spectral-sequence} implies that \(P(E_{1}^{\ast,\ast})(-1,q)\) is a linear combination of polynomials \(\RT_{N,j}(D_{\wh{I}})\), where
    \(\wh{I}\in A_{k}(X)\) and appropriate \(j\).
    Consequently,
    \begin{equation}
      \label{eq:difference-polynomials-from-spec-seq}
      \begin{split}
	\DP_{N,\ell-u}(D)& = \RT_{N,\ell-u}(D) - \RT_{N,\ell-u+1}(D) =\\
			 &=P(E_{1}^{\ast,\ast}(p^{\ell-u}))(-1,q) - P(E_{1}^{\ast,\ast}(p^{\ell-u+1}))(-1,q).
      \end{split}
    \end{equation}
  
  We can now apply formula~\eqref{eq:difference-polynomials-from-spec-seq} to \(\DP_{N,\ell-u}(\Lpm)\) and \(\DP_{N,\ell-u}(\Lpm)\) obtaining
  \begin{align*}
    \DP_{N,\ell-u}(\Lpm) &= \sum_{k=0}^{p^{\ell}} \sum_{s=u}^{\ell} \sum_{{\wh{I}} \in A_{k}^{s}(\Lpm,X)} (-1)^{k} q^{-p^{\ell}(N-1)-k} \DP_{N,s-u}(D_{{\wh{I}}}), \\
    \DP_{N,\ell-u}(\Lmm) &= \sum_{k=-p^{\ell}}^{0} \sum_{s=u}^{\ell} \sum_{{\wh{I}} \in A_{p^{\ell}+k}^{s}(\Lmm,X)} (-1)^{k} q^{p^{\ell}(N-1)+p^{\ell}-k} \DP_{N,s-u}(D_{{\wh{I}}}).
  \end{align*}
  Since, \(A_{k}(\Lpm,X) = A_{p^{\ell}-k}(\Lmm,X)\), we obtain
  \[q^{p^{\ell} N} \DP_{N,\ell-u}(\Lpm) - q^{-p^{\ell} N} \DP_{N,\ell-u}(\Lmm) = \sum_{k=0}^{p^{\ell}} \sum_{s=u}^{\ell} \sum_{{\wh{I}} \in A_{k}^{s}(\Lpm,X)} (-1)^{k} (q^{p^{\ell}-k} - q^{-p^{\ell}+k}) \DP_{N,s-u}(D_{{\wh{I}}}).\]
  Observe that \(A_{k}^{s}(\Lpm,X)\) is empty unless \(p^{s}\) divides \(k\).
  Consequently,
  \begin{align*}
    &q^{p^{\ell} N} \DP_{N,\ell-u}(\Lpm) - q^{-p^{\ell} N} \DP_{N,\ell-u}(\Lmm) - (q^{p^{\ell}} - q^{-p^{\ell}}) \DP_{N,\ell-u}(\Lom) = \\
    &\sum_{k=1}^{p^{\ell}-1} \sum_{s=u}^{\ell} \sum_{{\wh{I}} \in A_{k}^{s}(\Lpm,X)} (-1)^{k} (q^{p^{\ell}-k} - q^{-p^{\ell}+k}) \DP_{N,s-u}(D_{{\wh{I}}}).
  \end{align*}
  For \(u = \ell\), the right-hand side of the above equality is zero, i.e.
  \[q^{p^{\ell} N} \DP_{N,0}(\Lpm) - q^{-p^{\ell} N} \DP_{N,0}(\Lmm) = (q^{p^{\ell}} - q^{-p^{\ell}}) \DP_{N,0}(\Lom),\]
  as desired.

  Suppose now that \(0 \leq u < \ell\).
  For any \(u \leq s \leq \ell\) and \(k\) divisible by \(p^{s}\), if \({\wh{I}} \in A_{k}^{s}(\Lpm,X)\), we have
  \[(q^{p^{\ell}-k} - q^{-p^{\ell}+k}) \DP_{N,s-u}(D_{{\wh{I}}}) \equiv 0 \pmod{q^{p^{u}}-q^{p^{u}}}.\]
  Consequently,
  \[q^{p^{\ell} N} \DP_{N,0}(\Lpm) - q^{-p^{\ell} N} \DP_{N,0}(\Lmm) \equiv (q^{p^{\ell}} - q^{-p^{\ell}}) \DP_{N,0}(\Lom) \pmod{q^{p^{u}}-q^{p^{-u}}}.\]
\end{proof}

\begin{corollary}
  Let \(I_{p^{\ell}}\) be the ideal in the polynomial ring \(\Z[q^{\pm1}]\) generated by polynomials
  \[(q^{p^{\ell}}-q^{p^{-\ell}}), p \cdot (q^{p^{\ell}-1}-q^{-p^{\ell}+1}), \ldots, p^{\ell-1}(q^{p}-q^{-p}), p^{\ell}.\]
  If \(L\) is a $p^\ell$-periodic link, then
  \[\RT_{N}(L)(q) \equiv \RT_{N}(L)(q^{-1}) \pmod{I_{p^{\ell}}}.\]
\end{corollary}
\begin{proof}
  The proof is essentially a repetition of an analogous result in \cite{Politarczyk-Jones}. For the reader's convenience,
  we present a sketch.

  A $p^\ell$-periodic link can be transformed to its mirror by a sequence of equivariant crossing changes.
  Each such crossing change, by virtue of Proposition~\ref{prop:jones}, changes the difference polynomials $\DP_{N,j}$
  by a multiple of $q^{p^{\ell-j}}-q^{-p^{\ell-j}}$. From the definition of difference polynomials, 
  we recover that an equivariant crossing change changes the polynomial $\RT_{N,j}$ by a multiple of $q^{p^{\ell-j}}-q^{-p^{\ell-j}}$.
  The relation \eqref{eq:PN_and_PNd}, evaluated at $t=-1$ shows, that under an equivariant crossing change changes
  the polynomial $\RT_N$ by an element belonging to $I_{p^\ell}$. This means, that the polynomial $\RT_N$ of a link and of its mirror
  differ by an element from $I_{p^\ell}$. We conclude by Corollary~\ref{cor:p_mirror}.
\end{proof}

\subsection{Periodicity criterion}\label{sub:periodicity_criterion}
The following result ports the periodicity criterion of \cite{BP} to the case of $\sln$-homology.

\begin{theorem}\label{thm:periodicity}
  Suppose $L$ is an $m=p^\ell$ periodic knot with $p$-prime.
  Then, there exist polynomials $\cP_0,\cP_1,\dots$ such that
  \[\KRP_N=\cP_0+\sum_{j=1}^\ell (p^j-p^{j-1}) \cP_j.\]
  Here $\cP_0,\dots$ are Laurent polynomials in $t,q$ such that
  \begin{enumerate}[label=(P-\arabic*)]
    \item the Laurent polynomial $\cP_{0}$ can be presented as $\cP_0=q^{s_N(L)}(q^{1-N}+q^{3-N}+\dots+q^{N-1}) + \sum_{j=1}^\infty (1+tq^{Nj})\mathcal{S}_{0j}(t,q)$, while the Laurent polynomials $\cP_k$, $k>0$ can be presented as $\cP_k=\sum_{j=1}^{\infty}(1+tq^{Nj})\mathcal{S}_{kj}(t,q)$.\label{item:present}
    \item The Laurent polynomials $\mathcal{S}_{kj}$, $k\ge 0$, from item~\ref{item:present} have non-negative coefficients; \label{item:non-negative}
    \item the polynomials $\cP_k$, $k\ge 0$ satisfy the following congruence relation: $\cP_{k}(-1,q) - \cP_{k+1}(-1,q) \equiv \cP_{k}(-1,q^{-1}) - \cP_{k+1}(-1,q^{-1})\pmod{q^{p^{\ell-k}} - q^{-p^{\ell-k}}}$.\label{item:congruence}
  \end{enumerate}
  Here, $s_N(L)$ is the Lewark's $s$-invariant; see Lemma~\ref{lem:anyknot}.
\end{theorem}
\begin{proof}
  The proof follows the same line as the proof of the main result in \cite{BP}. For integral $k,r$, we have $\KR^{k,r}_N(L)=\EKR^{k,r}_N(L)$ as
  vector spaces. The latter admits the decomposition as in \eqref{eq:third_grading}:
  \[\EKR_N^{k,r}(L)=\bigoplus_{d|m}\EKR_N^{k,r,d}(L).\]
  As $m=p^\ell$, writing $\cP_j=P_{N,p^j}$ as the Poincar\'e polynomial of $\EKR_N^{\bullet,\bullet,p^j}(L)$
  we have
  \[\KRP_N(L)=\sum_{j=0}^N(p^{j}-p^{j-1})\cP_j.\]
  Further properties of $\cP_j$ are deduced from the Lee--Gornik spectral sequence. As the Lee--Gornik spectral sequence is $\Z_m$-equivariant,
  the submodule $\EKR_N^{\bullet,\bullet,p^j}$ converges to $\ELee_N^{\bullet,p^j}$. Therefore, by the second
  part of Proposition~\ref{prop:lee_spectral}, with $\cS_{jk}=R^{p^j}_k$, we have:
  \[\cP_j=\LeeP_{N,p^j}(L)+\sum_{k=1}^\infty (1+tq^{2Nk})\cS_{jk}.\]
  The sum is finite (the spectral sequence degenerates at a finite page, because its $E_1$ page is finite dimensional over $\C$),
  and all the $\cS_{jk}$ have non-negative coefficients.

  The computation of $\ELee$ in Lemma~\ref{lem:eleep}, together with Lemma~\ref{lem:anyknot}, yields
  \[\LeeP_{N,p^0}(L)=q^{s_N(L)}(q^{-N+1}+q^{-N+3}+\dots+q^{N-1}),\]
  whereas $\LeeP_{N,p^j}(L)=0$ for $j>0$. This proves~\ref{item:present} and~\ref{item:non-negative}.

  The congruence relation of~\ref{item:congruence}
  is proved using Proposition~\ref{prop:jones}. Namely, $(\cP_j-\cP_{j+1})|_{t=-1}$ is, by definition, the difference
  polynomial $\DP_{N,j}$. Proposition~\ref{prop:jones} indicates
  that changing an orbit of crossings on a diagram, does not affect
  $\DP_{N,j}$ modulo the ideal generated by $q^{p^{n-j}}-q^{-p^{n-j}}$. Changing all orbits of crossings turns a link into its mirror.
  By Corollary~\ref{cor:p_mirror}, the polynomial $\DP_{N,j}$ of a $p^n$-symmetric link
  and of its mirror are congruent modulo $q^{p^{n-j}}-q^{-p^{n-j}}$. This proves~\ref{item:congruence} and concludes the proof of the theorem.
\end{proof}

\subsection{Periodicity~3 and 4}

We will now show that the criterion cannot obstruct a knot from being $3$ or $4$-periodic. We begin with the following result.
\begin{theorem}[\cite{stupid_on_homflypt}]\label{thm:stupid}
  If $K$ is a knot and $X$ its HOMFLY-PT polynomial, then $X(a,b)=T(a,b)q(a,b)+1$, where $q(a,b)$ is a Laurent polynomial
  with integer coefficients and $T(a,b)=a^4-2a^2+1-a^2b^2$ is the HOMFLY-PT polynomial for the trefoil.
\end{theorem}
From Theorem~\ref{thm:stupid} and \eqref{eq:RT_N} we deduce the following result.
\begin{corollary}\label{cor:not_that_stupid}
  The $\RT_N$ polynomial of a knot $K$ has form
  \[\RT_N(q)=A(q)(q^{4N}-2q^{2N}+1-q^{2N}(q-q^{-1})^2)+1,\]
  where $A(q)$ is a Laurent polynomial with integer coefficients.
\end{corollary}
Define
\[T_N=q^{4N}-2q^{2N}+1-q^{2N}(q-q^{-1})^2.\]
\begin{lemma}\label{lem:divides}\
  \begin{itemize}
    \item If $\zeta_6$ is a root of unity of order $6$, then $T_N(\zeta_6)=0$ unless $3|N$;
    \item If $\zeta_8$ is a root of unity of order $8$ and $N$ is odd, then $T_N(\zeta_8)=0$.
  \end{itemize}
\end{lemma}
\begin{proof}
  To prove the first part,
  we check the statement directly for $N=1,2$. From this, the statement follows for general $N$ congruent to $1$ or $2$ modulo $3$.
  The second part is proved in the same way.
\end{proof}
\begin{corollary}\label{cor:divisibility}
  For any knot $K$, we have the congruences $\RT_N(q)-\RT_N(q^{-1})\equiv 0\bmod q^3-q^{-3}$,
  $\RT_N(q)-\RT_N(q^{-1})\equiv 0\bmod q^4-q^{-4}$.
\end{corollary}
\begin{proof}
  We begin with the first part.
  This congruence is equivalent to saying that for any root of unity $\zeta_6$ of order $6$, it holds $\RT_N(\zeta_6)-\RT_N(\zeta_6^{-1})=0$.
  Consider two cases. Suppose $N$ is not a multiple of $3$. By Corollary~\ref{cor:not_that_stupid} and 
  Lemma~\ref{lem:divides}, we have $\RT_N(\zeta_6)=1$, so $\RT_N(\zeta_6)-\RT_N(\zeta_6^{-1})=0$.

  Suppose $3|N$. We have the following decomposition of the Khovanov--Rozansky polynomial, see Proposition~\ref{prop:lee_spectral}.
  \[\KRP_N(t,q)=q^s(q^{1-N}+q^{3-N}+\dots+q^{N-1})+\sum_j (1+tq^{2Nj})R_j(t,q).\]
  Now the $\RT_N$ polynomial is obtained by substituting $t=-1$. The term $(1+tq^{2Nj})$ for $t=-1$ and $q=\zeta_6$
  is equal to zero. On the other hand,
  \[q^{1-N}+q^{3-N}+\dots+q^{N-1}=\frac{q^N-q^{-N}}{q-q^{-1}}.\]
  The latter expression is zero, when evaluated at a root of unity of order dividing $2N$. That is to say
  \[\RT_N(\zeta_6)=\KRP_N(-1,\zeta_6)=0.\]

  \smallskip
  The second congruence is proved analogously. First, assume that $N$ is odd. Then, $\RT_N(\zeta_8)=1$ by the same argument
  combining Corollary~\ref{cor:not_that_stupid} and Lemma~\ref{lem:divides}.
  Suppose $4|N$. Then, the same argument as above shows that $\KRP_N(-1,\zeta_8)=0$. It remains to deal with the case $N=4k+2$.
  Suppose $\zeta_8$ is such that $\zeta_8^{4k}=1$.
  Consider $X(a,b)$, the HOMFLYPT polynomial for the knot. 
  From the formula $\RT_N(q)=X(q^N,q-q^{-1})$ we immediately deduce
  that $\RT_N(\zeta_8)=\RT_2(\zeta_8)$. Now, $\RT_2$ is the Jones polynomial. It was proved in \cite[Section 4.6]{BP} that $\RT_2(\zeta_8)-\RT_2(\zeta_8^{-1})=0$.
  Note that the formula $\RT_N(\zeta_8)=\RT_2(\zeta_8)$ holds also if $N=8k+2$ regardless of the sign of $\zeta_8^4$.

  The remaining case is when $N=8k-2$ and $\zeta_8^4=-1$. Write $X(a,b)=\sum \alpha_{ij}a^ib^j$.
  Note that the skein relation~\eqref{eq:hom_skein} implies, that $\alpha_{ij}=0$ unless $j$ is even
  (remember, that we deal with knots).
  As $\RT_N(q)=X(q^N,q-q^{-1})$, we have
  \begin{multline*}
    \RT_N(\zeta_8)-\RT_N(\zeta_8^{-1})=\sum\alpha_{ij} \zeta_8^{Ni}(\zeta_8-\zeta_8^{-1})^j-\zeta_8^{-Ni}(-\zeta_8+\zeta_8^{-1})^j=\\
  \sum\alpha_{ij} (\zeta_8^{-2i}-\zeta_8^{2i})(\zeta_8-\zeta_8^{-1})^j=-\RT_2(\zeta_8)+\RT_2(\zeta_8^{-1}).
\end{multline*}
  However, $\RT_2$ is the Jones polynomial. It was proved in~\cite[Section 4.6]{BP},
  that $\RT_2(q)-\RT_2(q^{-1})$ is divisible by $q^4-q^{-4}$. Therefore, the latter quantity is zero.
\end{proof}
\begin{corollary}
  Suppose $K$ is a knot. Set $\cP_0=\KRP_N$, $\cS_{0j}=R_j$, where $R_j$ is as in Proposition~\ref{prop:lee_spectral}.
  Then $\cS_{0j},\cP_0$ satisfy the statement of Theorem~\ref{thm:periodicity} regardless of whether or not $K$ is really $3$ or $4$-periodic.
\end{corollary}
\begin{proof}
  We prove the corollary for $3$-periodicity, the case of $4$-periodicity is analogous.
  Item~\ref{item:present} is satisfied by definition.
  It follows from Proposition~\ref{prop:lee_spectral} that $\cS_{0j}$ have non-negative coefficients. The congruence~\ref{item:congruence} is
  a direct consequence of Corollary~\ref{cor:divisibility}.
\end{proof}

\def\MR#1{}
\bibliographystyle{abbrv}

\bibliography{homotopy}
\end{document}